\documentclass[reqno]{amsart}

\usepackage[a4paper,margin=22mm]{geometry}

\usepackage[T1]{fontenc}
\usepackage[utf8]{inputenc}
\usepackage{microtype}
\usepackage{amssymb}
\usepackage{mathtools}
\usepackage{mathrsfs}
\usepackage{dsfont}
\usepackage{nicefrac}
\usepackage{enumitem}
\usepackage{csquotes}

\allowdisplaybreaks[3]
\numberwithin{equation}{section}

\makeatletter
\def\paragraph{\@startsection{paragraph}{4}%
  \z@\z@{-\fontdimen2\font}%
  {\normalfont\itshape}}
\makeatother

\usepackage[
  style=numeric,
  sorting=nyt,
  url=false,
  doi=true,
  isbn=false,
  maxbibnames=9,
  minbibnames=3,
  maxcitenames=2,
  mincitenames=1,
  sortcites=true,
  backend=biber,
  giveninits=true,
  abbreviate=false,
  dateabbrev=false
]{biblatex}
\renewbibmacro{in:}{%
  \ifentrytype{article}{}{\printtext{\bibstring{in}\intitlepunct}}%
}
\DeclareFieldFormat{journaltitle}{\mkbibemph{#1\isdot}}
\renewbibmacro*{volume+number+eid}{%
  \printfield{volume}%
  \iffieldundef{number}{}{%
    \setunit*{:}\printfield{number}%
  }%
  \iffieldundef{eid}{}{%
    \setunit{\addcomma\space}\printfield{eid}%
  }%
}
\DeclareFieldFormat{doi}{%
  \mkbibacro{DOI}\addcolon\space
  \ifhyperref
    {\href{https://doi.org/#1}{\nolinkurl{#1}}}
    {\nolinkurl{#1}}%
}

\theoremstyle{plain}
\newtheorem{theorem}{Theorem}[section]
\newtheorem{lemma}[theorem]{Lemma}
\newtheorem{proposition}[theorem]{Proposition}
\newtheorem{corollary}[theorem]{Corollary}

\theoremstyle{definition}
\newtheorem{definition}[theorem]{Definition}

\theoremstyle{remark}
\newtheorem{remark}[theorem]{Remark}

\theoremstyle{plain}
\newtheorem{maintheorem}{Theorem}
\renewcommand*{\themaintheorem}{\Alph{maintheorem}}

\usepackage[hidelinks]{hyperref}
\hypersetup{
  unicode=true,
  bookmarksnumbered=true,
  bookmarksopen=true,
  bookmarksopenlevel=2,
  breaklinks=true,
  pdftitle={Local structure at the maximum and sharp persistence asymptotics of rough fractional Brownian motion},
  pdfauthor={Christian M\"onch},
  pdfsubject={Local structure at the maximum and sharp persistence asymptotics of rough fractional Brownian motion},
  pdfkeywords={fractional Brownian motion, localisability, persistence, hard-wall law, tangent process}
}
\usepackage[nameinlink,noabbrev,capitalize]{cleveref}

\crefname{equation}{equation}{equations}
\Crefname{equation}{Equation}{Equations}
\crefname{maintheorem}{Theorem}{Theorems}
\Crefname{maintheorem}{Theorem}{Theorems}

\renewcommand{\P}{\mathbb{P}}
\newcommand{\PP}{\mathbb{P}}
\newcommand{\EE}{\mathbb{E}}

\newcommand{\deq}{\overset{d}{=}}

\title[Local structure at the maximum]{%
  Local structure at the maximum and sharp persistence asymptotics of rough
  fractional Brownian motion}
\author{Christian M\"onch}
\thanks{ORCID: \href{https://orcid.org/0000-0002-6531-6482}%
  {0000-0002-6531-6482}.}
\email{cmoench25@gmail.com}
\urladdr{https://orcid.org/0000-0002-6531-6482}
\date{August 19, 2026}


\begin{document}

\begin{abstract}
We consider a fractional Brownian motion \(B\) with Hurst index \(0<H<1/2\), and
its maximiser \(\tau\) on \([0,1]\).  We show that the rescaled process
\(a^H(B_{\tau+\,\cdot\,/a}-B_\tau)\) converges in
\(C_{\mathrm{loc}}(\mathbb R)\) to a limiting tangent law that is
\(H\)-self-similar, supported on nonpositive paths pinned at zero, and
rerooting-rescaling invariant: rerooting the limit process at
its maximum on any fixed compact interval separated from zero and rescaling
again asymptotically reproduces the same law. We also identify the tangent
law as the limit of two-sided finite-grid hard-wall laws as the mesh vanishes
and both horizons diverge, which can informally be interpreted as conditioning fractional
Brownian motion on a
nonpositive path. As an application, we consider persistence
probabilities for fractional Brownian motion: a tilted variant of $B$ yields
a different tangent law with a finite left horizon and an infinite right
horizon and we show that
\[
  \mathbb P(B_t\leq1\text{ for all }0\leq t\leq T)
  = \big(C+o(1)\big)\,T^{-(1-H)},\quad \text{as }T\to\infty,
\]
where the leading order coefficient \(C\in(0,\infty)\) has an explicit representation in terms of the expected maximum and the tilted tangent law.
\end{abstract}

\maketitle

\tableofcontents

\clearpage

\section{Introduction}
\label{sec:introduction}

\paragraph{Localisability and tangent processes.} Let \(B=(B_t)_{t\in\mathbb R}\) be two-sided fractional Brownian
motion with Hurst index \(0<H<1/2\), normalised such that \(B_0=0\) and
\begin{equation}
  \EE(B_t-B_s)^2=|t-s|^{2H}.
  \label{eq:fbm-increments}
\end{equation}
At every deterministic base point \(u\in\mathbb R\), stationary increments
and self-similarity give the exact scaling identity
\begin{equation}
  \bigl(r^{-H}(B_{u+rt}-B_u)\bigr)_{t\in\mathbb R}
  \deq (B_t)_{t\in\mathbb R},
  \qquad r>0.
  \label{eq:deterministic-fbm-localisability}
\end{equation}
Thus fractional Brownian motion is its own \emph{tangent process} at a deterministic
point.  More generally, tangent processes describe distributional limits of rescaled increments and are used to formalise the local structure, or \emph{localisability}, of stochastic processes.  Early Gaussian constructions
based on prescribed local scalings include
\cite{benassiJaffardRoux1997elliptic} and the general concept was introduced by Falconer \cite{falconer2002tangentFields, falconer2003local}, cf.\ \cite{falconerLeGuevelLevyVehel2009localizable} for an extension to stable and multistable moving-average processes.

The term tangent process is motivated by the analogous concept of tangent measures in geometric measure theory, which was first discussed in the context of
measure-theoretic blow-ups introduced by \textcite{preiss1987geometry}.  A
tangent measure is a weak limit of renormalised magnifications of a Radon
measure about a base point.  The probabilistic object relevant to repeated
zooming is not an individual tangent measure but a distribution of sceneries
across shrinking scales.  In that setting,
\textcite{moertersPreiss1998tangent} proved that, for every measure on
Euclidean space and every dimension parameter in their construction, at
almost every base point all tangent-measure distributions are Palm
distributions.  Informally, the Palm property expresses consistency between
the view from a typical marked point and an appropriately mass-biased
re-rooting inside the scenery.

In Falconer's deterministic-base-point framework, an essentially unique
tangent process is self-similar with stationary increments.  In particular,
Falconer \cite{falconer2003local} proves that the tangent process \(Y\) in
this setting satisfies, for every \(s\in\mathbb R\) and \(a>0\),
\begin{equation}
  \bigl(a^H\{Y(s+t/a)-Y(s)\}\bigr)_{t\in\mathbb R}
  \deq (Y(t))_{t\in\mathbb R}.
  \label{eq:ordinary-tangent-rerooting-invariance}
\end{equation}

\paragraph{Localisability at the maximum.}
The situation changes when the base point is selected by the sample path.
Let \(\tau\) be the almost surely unique maximiser of \(B\) on \([0,1]\), and
consider the process seen from this maximiser at scale \(a\):
\begin{equation}
  \Xi_a(t):=a^H\bigl(B_{\tau+t/a}-B_\tau\bigr),
  \qquad t\in\mathbb R.
  \label{eq:tangent-zoom}
\end{equation}

This is a path-selected local question rather than an application of
\eqref{eq:deterministic-fbm-localisability}: choosing \(\tau\) retains the
global constraint that the recentred path stay below zero on two expanding
horizons.  For ordinary Brownian motion, the path decomposition at its
maximum is classical: after the appropriate rescaling, the two pieces are
independent Brownian meanders \parencite{denisov1984maximum}, and the meander
is absolutely continuous with respect to the three-dimensional Bessel
process \parencite{imhof1984density}.  Rough fractional Brownian motion has
neither independent increments nor the Markov property underlying this
description.

Here, we show that the maximum-centred zoom nevertheless has a canonical
limit in \(C_{\mathrm{loc}}(\mathbb R)\), which we call the \emph{tangent law
at the maximum}.  Its paths are pinned at zero and strictly negative
elsewhere, and its law is \(H\)-self-similar.  The limit also has a
rerooting-rescaling invariance mirroring
\eqref{eq:ordinary-tangent-rerooting-invariance}: If \(X\) has
the tangent law at the maximum and \(S\) is its maximiser on any fixed compact interval
separated from zero, then \(S\) is almost surely unique and interior, and
the laws of the paths
\(a^H\{X(S+\,\cdot\,/a)-X(S)\}\) converge weakly back to the tangent law.

\paragraph{Persistence of the maximum of fractional Brownian motion.}
Our motivation for studying the tangent law at the maximum comes from the
relation between invariance properties of path-dependent shifts of a process
and persistence probabilities; see
\cite{moench2022localtimePersistence,Moench2024}.  Define
\begin{equation}
  p_H(T):=\PP(B_t\leq1\text{ for every }0\leq t\leq T),
  \qquad T>0.
  \label{eq:introduction-persistence-probability}
\end{equation}
Writing \(M=\max_{0\leq t\leq1}B_t\), self-similarity gives the exact identity
\begin{equation}
  p_H(T)=\PP(M\leq T^{-H}).
  \label{eq:persistence-maximum-lower-tail}
\end{equation}
The study of long no-crossing events for (stationary) Gaussian processes goes
back at least to \textcite{KacSlepian1959largeExcursions}, \textcite{Slepian1961firstPassageTime}, and \textcite{newellRosenblatt1962zerocrossing}.  The
maximum-distribution problem for fractional Brownian motion was raised in
particular by \textcite{sinai1997maximum} and solved at the level of the exact
decay exponent in a fundamental subsequent paper by
\textcite{molchan1999maximum}.  Molchan connected the problem to a reciprocal
exponential functional and identified the persistence exponent as \(1-H\).  More
precisely, he proved, for some \(k_H<\infty\),
\[
  T^{-(1-H)}e^{-k_H\sqrt{\log T}}
  \leq p_H(T)
  \leq T^{-(1-H)}e^{k_H\sqrt{\log T}}
\]
for all sufficiently large \(T\).  \Textcite{aurzada2011onesided} replaced
these stretched-logarithmic losses by powers of \(\log T\): for every
\(\varepsilon>0\),
\[
  c_{H,\varepsilon}T^{-(1-H)}
      (\log T)^{-1/(2H)-\varepsilon}
  \leq p_H(T)
  \leq C_{H,\varepsilon}T^{-(1-H)}
      (\log T)^{(2-H)/H+\varepsilon}.
\]
The same circle of ideas was subsequently adapted to logarithmically moving
boundaries \cite{aurzadaBaumgarten2013moving} and extended to broader classes
of processes
\cite{aurzadaGuillotinPlantard2015reversible,aurzadaGuillotinPlantardPene2018stationary,aurzadaMoench2019hermite}.
For the exact continuous level-one probability in
\eqref{eq:introduction-persistence-probability}, the strongest bounds before
the present work combined the lower bound of
\textcite{aurzadaGuillotinPlantardPene2018stationary} with the upper bound of
\textcite{pengRao2023smallIncrements}.  Writing \(\log_k\) for the \(k\)-fold
iterated logarithm, these are
\[
  \begin{aligned}
    c_H T^{-(1-H)}(\log T)^{-1/(2H)}
    &\leq p_H(T)\\
    &\leq C_{H,N,\varepsilon}T^{-(1-H)}
        (\log T)^{1/H-1}(\log_2T)^{3/(2H)}\\
    &\qquad{}\times(\log_3T\,\log_4T)^{1/H}
        (\log_5T)^{1/(2H)}(\log_NT)^\varepsilon.
  \end{aligned}
\]
This holds for \(0<H<1/2\), every integer \(N\geq6\), and every
\(\varepsilon>0\).  In the smoother, positively correlated range,
\textcite{aurzadaGuillotinPlantardPene2018stationary} proved both bounds
\[
  c_H T^{-(1-H)}(\log T)^{-1/(2H)}
  \leq p_H(T)
  \leq C_H T^{-(1-H)},
  \qquad \tfrac12<H<1.
\]
Note that at \(H=1/2\), the reflection principle gives the exact Brownian formula
\[
  p_{1/2}(T)=2\Phi(T^{-1/2})-1
  \sim \sqrt{\frac{2}{\pi}}\,T^{-1/2},
\]
where \(\Phi\) is the standard normal distribution function.

Here, we prove sharp asymptotics in the rough range \(0<H<1/2\) and show
that there is a constant \(\mathsf C_H\in(0,\infty)\) such that
\begin{equation}
  \label{eq:persasymp}
  p_H(T)
  = \big(\mathsf C_H+o(1)\big)\,T^{-(1-H)}.
\end{equation}

A parallel sharp result was already known for the local time at zero,
\[
  \ell_H(0,T]
  :=\lim_{\varepsilon\downarrow0}\frac{1}{2\varepsilon}
      \int_0^T\mathbf 1_{\{|B_t|<\varepsilon\}}\,\mathrm dt.
\]
It was shown in \cite{moench2022localtimePersistence} that, for any $H\in (0,1)$,
\[
  \PP\left(\ell_H(0,T]\leq1\right)
  = \big(\mathsf C_H^{\mathrm{LT}}+o(1)\big)\,T^{-(1-H)},
\]
where \(\mathsf C_H^{\mathrm{LT}}\) admits a representation as the intensity,
on the local-time scale, of macroscopic excursions from $0$.

In contrast, to identify the constant in \eqref{eq:persasymp}, we study a sequence of Laplace tilts generated by \(M\), which forces the maximum height to scale as \(a^{-H}\) and its location to
scale as \(a^{-1}\) from the left endpoint.  The resulting local law is not
the ordinary two-sided tangent law. Instead, it is a tilted mixture of hard-wall laws with
a finite left horizon and an infinite right horizon.  We identify this mixture
and prove that
\begin{equation}
    \mathsf C_H=\frac{H\,\EE[M]}{\Gamma(1/H)D_H}\in(0,\infty),
  \label{eq:introduction-sharp-persistence}
\end{equation}
where \(D_H\) is defined by the limiting tilted finite-left-horizon law.

\paragraph{Organisation of the paper.}
Section~\ref{sec:main-results} states the principal results in full detail.
Section~\ref{sec:tangent-proofs} imports the rough finite-wall results needed
from the companion paper \cite{moench2026horizonProblem}, retaining the
primary-source Green-kernel statement that is also used independently later,
and then proves the ordinary tangent law and its rerooting-rescaling
invariance.
Section~\ref{sec:rough-persistence} proves the sharp persistence theorem and
develops the required tilted finite-left-horizon law.

\section{Main results}
\label{sec:main-results}
\label{sec:setup}

Throughout the remainder of the paper we always work in the rough range \(0<H<1/2\).
We write \(\mathbb R_+:=[0,\infty)\).

\subsection{The tangent law at the maximum}
\label{subsec:main-tangent-result}

Let
\begin{equation}
  \mathcal X:=C_{\mathrm{loc}}(\mathbb R)
\end{equation}
be the space of continuous real-valued functions on \(\mathbb R\), equipped
with the topology of uniform convergence on compact sets, metrised for example
by
\begin{equation}
  d_{\mathrm{loc}}(f,g)
  :=\sum_{R=1}^\infty 2^{-R}
       \bigl(1\wedge\|f-g\|_{C([-R,R])}\bigr).
  \label{eq:local-uniform-metric}
\end{equation}
For a nondegenerate compact interval \(I\subset\mathbb R\) and
\(f\in\mathcal X\), let
\begin{equation}
  S_I(f):=\operatorname*{argmax}_{s\in I}f(s)
  \label{eq:rough-window-maximizer}
\end{equation}
whenever this maximiser is unique.  For \(a>0\), define
\begin{equation}
  (\mathcal R_a^If)(t)
  :=a^H\bigl(f(S_I(f)+t/a)-f(S_I(f))\bigr),
  \qquad t\in\mathbb R.
  \label{eq:rough-window-rerooting-map}
\end{equation}

\begin{definition}[Rerooting-rescaling invariance]
\label{def:asymptotic-window-rerooting-invariance}
A probability law \(\mu\) on \(\mathcal X\) is called
\emph{rerooting-rescaling invariant at its maximum in \(I\) with exponent
\(H\)} if a path \(X\sim\mu\) has an almost surely unique maximiser on
\(I\), this maximiser belongs to \(\operatorname{int}I\) almost surely, and
\begin{equation}
  \mathcal L(\mathcal R_a^IX)\Longrightarrow\mu
  \qquad\text{on }\mathcal X
  \quad\text{as }a\to\infty.
  \label{eq:rough-window-rerooting-limit}
\end{equation}
\end{definition}

In view of \eqref{eq:ordinary-tangent-rerooting-invariance}, this is the
selected-point, asymptotic analogue of the deterministic
rerooting-rescaling invariance encoded by stationary increments and
self-similarity.

We call the horizontal line \(\mathbb R\times\{0\}\) the zero wall.  By a
dyadic mesh we mean \(h=2^{-m}\) with \(m\in\mathbb Z\).  For such an \(h\)
and \(L^-,L^+>0\) satisfying \(L^-/h,L^+/h\in\mathbb N\), define the
constraint grid
\begin{equation}
  \Lambda(h,L^-,L^+)
  :=h\mathbb Z\cap[-L^-,L^+]\setminus\{0\}
  \label{eq:wall-grid}
\end{equation}
and the finite-grid hard-wall path law
\begin{equation}
  Q_{h,L^-,L^+}
  :=\mathcal L\left(
       B\,\middle|\,
       B_t\leq0\text{ for every }t\in\Lambda(h,L^-,L^+)
     \right).
  \label{eq:pure-wall-law}
\end{equation}
The conditioning event has positive probability; this follows, for example,
from Lemma~\ref{lem:fbm-compact-positivity-support} below.

\begin{maintheorem}
\label{thm:main-tangent-process}
Assume \(0<H<1/2\).  The following assertions hold.
\begin{enumerate}[
  label=\textup{(\themaintheorem.\arabic*)},
  ref=\themaintheorem.\arabic*,
  font=\normalfont\bfseries,
  leftmargin=*
]
\item\label{thm:main-tangent-limit}
\emph{Tangent limit.}
If \(\tau\) is the unique maximiser of \(B\) on \([0,1]\), then there is a
probability law \(Q\) on \(\mathcal X\) such that the zooms \(\Xi_a\)
defined in \eqref{eq:tangent-zoom} satisfy
\begin{equation}
  \mathcal L(\Xi_a)\Longrightarrow Q
  \qquad\text{on }\mathcal X
  \quad\text{as }a\to\infty.
  \label{eq:main-tangent-convergence}
\end{equation}
The limiting law is \(H\)-self-similar: if \(X\) has law \(Q\), then, for
every \(c>0\),
\begin{equation}
  \bigl(c^H X(t/c)\bigr)_{t\in\mathbb R}
  \stackrel d= X.
  \label{eq:limit-selfsimilarity}
\end{equation}

\item\label{thm:main-conditional-limit}
\emph{Two-sided hard-wall limit.}
The tangent law \(Q\) in \ref{thm:main-tangent-limit} is also the path-space
limit of the finite-grid hard-wall laws from \eqref{eq:pure-wall-law}:
\begin{equation}
  Q_{h,L^-,L^+}\Longrightarrow Q
  \label{eq:pure-wall-path-limit}
\end{equation}
whenever \(h\downarrow0\) through dyadic values and
\(L^-\to\infty\), \(L^+\to\infty\), without any restriction on the ratio of
the two horizons. Moreover,
\begin{equation}
  Q\bigl(
    f(0)=0,
    f(t)<0\text{ for every }t\in\mathbb R\setminus\{0\}
  \bigr)=1.
  \label{eq:limit-support}
\end{equation}

\item\label{thm:rough-window-rerooting}
\emph{Rerooting-rescaling invariance.}
For each fixed nondegenerate compact interval
\(I\subset\mathbb R\) with \(0\notin I\), the tangent law \(Q\) is
rerooting-rescaling invariant at its maximum in \(I\) with exponent \(H\).
\end{enumerate}
\end{maintheorem}

\begin{remark}
\label{rem:rough-window-rerooting-scope}
For intervals \(I\) with \(0\in\operatorname{int}I\), rerooting-rescaling
invariance follows immediately from self-similarity and
\eqref{eq:limit-support}; the substantive case is that of intervals separated
from zero.  If \(0\in\partial I\), the same scaling identity holds about the
unique maximiser zero, but Definition~\ref{def:asymptotic-window-rerooting-invariance}
does not apply because that maximiser is not interior.
\end{remark}

\subsection{The tilted finite-left-horizon law and sharp persistence}
\label{subsec:main-persistence-result}

Recall that
\begin{equation}
  M:=\max_{0\leq t\leq1}B_t
  \quad\text{and}\quad
  \tau:=\operatorname*{argmax}_{0\leq t\leq1}B_t.
  \label{eq:persistence-maximum-time}
\end{equation}
The maximiser is almost surely unique by Lemma~\ref{lem:unique-interior-maximizer}
below.  For \(a>0\), define the tilted law
\begin{equation}
  \frac{\mathrm d\widehat{\PP}_a}{\mathrm d\PP}
  :=\frac{e^{-a^HM}}{Z_a},
  \qquad
  Z_a:=\EE e^{-a^HM},
  \label{eq:rough-persistence-tilt}
\end{equation}
and the rescaled maximum time and height
\begin{equation}
  U_a:=a\tau,
  \qquad
  V_a:=a^HM.
  \label{eq:rough-persistence-marks}
\end{equation}
The nonnegative deficit arms are
\begin{equation}
\begin{aligned}
  L_a(s)&:=a^H\bigl(M-B_{\tau-s/a}\bigr),
    \qquad 0\leq s\leq U_a,\\
  R_a(s)&:=a^H\bigl(M-B_{\tau+s/a}\bigr),
    \qquad 0\leq s\leq a-U_a.
\end{aligned}
  \label{eq:rough-persistence-deficit-arms}
\end{equation}
We extend the left arm constantly after \(U_a\), and the right arm
constantly after \(a-U_a\); the extensions are denoted by
\(\overline L_a\) and \(\overline R_a\).  Finally, put
\begin{equation}
  J_a
  :=\int_0^{U_a}e^{-L_a(s)}\,\mathrm ds
    +\int_0^{a-U_a}e^{-R_a(s)}\,\mathrm ds.
  \label{eq:rough-persistence-functional}
\end{equation}

\begin{maintheorem}
\label{thm:rough-sharp-persistence}
There is a probability law \(\mu_H\) on
\begin{equation}
  \mathcal Y
  :=\mathbb R_+^2\times C_{\mathrm{loc}}(\mathbb R_+)^2
  \label{eq:rough-marked-space}
\end{equation}
equipped with its product topology, with local uniform convergence in each
path coordinate,
such that
\begin{equation}
  \mathcal L_{\widehat{\PP}_a}
  (U_a,V_a,\overline L_a,\overline R_a)
  \Longrightarrow \mu_H.
  \label{eq:rough-tilted-marked-convergence}
\end{equation}
Let \((U,V,\overline L,R)\) denote the coordinate random element under
\(\mu_H\).
The left horizon \(U\) is finite almost surely, whereas \(R\) is defined on
the whole half-line.  The path \(\overline L\) is constant on
\([U,\infty)\) almost surely.  Write \(L\) for the restriction of the
constant-extended path \(\overline L\) to \([0,U]\).  Moreover,
\begin{equation}
  J_\infty
  :=\int_0^Ue^{-L(s)}\,\mathrm ds
    +\int_0^\infty e^{-R(s)}\,\mathrm ds
  \label{eq:rough-limit-functional}
\end{equation}
satisfies \(0<J_\infty<\infty\) almost surely and
\begin{equation}
  J_a\Longrightarrow J_\infty,
  \qquad
  \widehat{\EE}_a[J_a^{-1}]
  \longrightarrow
  D_H:=\EE_{\mu_H}[J_\infty^{-1}]\in(0,\infty).
  \label{eq:rough-reciprocal-limit}
\end{equation}
Finally, the persistence probability satisfies
\begin{equation}
  \lim_{T\to\infty}{T^{1-H}}{\PP\bigl(B_t\leq1\text{ for every }0\leq t\leq T\bigr)} = \mathsf C_H,
  \label{eq:rough-sharp-persistence-asymptotic}
\end{equation}
where
\begin{equation}
  \mathsf C_H
  =\frac{H\,\EE[M]}{\Gamma(1/H)D_H}.
  \label{eq:rough-persistence-constant}
\end{equation}
\end{maintheorem}

The law \(\mu_H\) is distinct from the ordinary two-sided tangent law \(Q\).
The likelihood in \eqref{eq:rough-persistence-tilt} selects maxima of height
\(O(a^{-H})\), and the left horizon \(a\tau\) has a finite limit under this
singular tilt.  This finite horizon is intrinsic to the persistence problem.

\begin{remark}[Connection to Molchan's exponential functional]
\label{subsec:molchan-functional-interpretation}

For fractional Brownian motion,
\textcite[Statement~1, equation~(8)]{molchan1999maximum} connected
persistence to the reciprocal exponential functional.  His
endpoint-differentiation argument, using stationary increments and
self-similarity, gives
\[
  I(a):=\int_0^a e^{B_s}\,\mathrm ds,
  \qquad
  a^{1-H}\EE[I(a)^{-1}]\longrightarrow H\,\EE[M].
\]
Proposition~\ref{prop:exact-persistence-compiler} gives this observable the
tangent-coordinate factorisation
\[
\begin{aligned}
  J_a
  &=\int_{-U_a}^{a-U_a}e^{\Xi_a(t)}\,\mathrm dt
    =a e^{-a^HM}\int_0^1e^{a^HB_u}\,\mathrm du,\\
  \EE[I(a)^{-1}]
  &=Z_a\,\widehat{\EE}_a[J_a^{-1}].
\end{aligned}
\]
Thus \(Z_a\) is a smooth persistence weight and \(J_a^{-1}\) records the
local width of the near-maximal region.  The reciprocal limit
\eqref{eq:rough-reciprocal-limit} makes this shape correction asymptotically
constant: it contributes \(D_H^{-1}\) to
\eqref{eq:rough-persistence-constant}, while the Tauberian passage from
\(Z_a\) to the lower tail of \(M\) contributes \(\Gamma(1/H)^{-1}\).
\end{remark}
\section{Localisability at the maximum: proof of Theorem~\ref{thm:main-tangent-process}}
\label{sec:tangent-proofs}

The proof has four stages.  We first establish that conditioning on a finite set of points yields a consistent hard-wall limit in the sense of finite-dimensional distributions.  Bounds on centred fluctuations
and conditional means lift this limit to a nonpositive, reflection-invariant
law \(Q\) on full trajectories.  We then identify \(Q\) with
the tangent law at the maximiser, which shows self-similarity.  Finally, a
compact-window rerooting argument proves strict negativity, uniqueness and
interiority of the window maximiser, and the rerooting-rescaling limit.

\subsection{Finite-dimensional hard-wall convergence}
\label{sec:wall-laws}
\label{sec:finite-grid-source}
\label{sec:finite-dimensional-wall-laws}

Our first goal is to prove the finite-dimensional hard-wall limit in
Theorem~\ref{thm:finite-dimensional-entrance-law}.  Its proof needs several
facts established in \cite{moench2026horizonProblem}: nondegeneracy,
Stieltjes precision and MTP\(_2\), monotonicity under additional wall
constraints, and a uniform one-coordinate tail bound.  We collect these
statements, together with the local results that connect them, before giving
the proof.

\subsubsection{Preliminary results}

Write
\begin{equation}
  \mathbb D:=\mathbb Z[1/2]\setminus\{0\}
\end{equation}
for the nonzero dyadic rationals.  The covariance kernel of \(B\) is
\begin{equation}
  K_H(s,t)
  :=\operatorname{Cov}(B_s,B_t)
  =\frac12\bigl(|s|^{2H}+|t|^{2H}-|s-t|^{2H}\bigr).
  \label{eq:fbm-covariance}
\end{equation}

\begin{lemma}[{\cite[Lemma~4.1]{moench2026horizonProblem}}]
\label{lem:fbm-compact-positivity-support}
Let \(K\subset\mathbb R\setminus\{0\}\) be compact.  Then
\begin{equation}
  \PP(B_s>0\text{ for every }s\in K)>0.
  \label{eq:fbm-compact-strict-positivity}
\end{equation}
Fix \(t\neq0\), and let
\begin{equation}
  \xi_s^{(t)}
  :=B_s-\frac{K_H(s,t)}{|t|^{2H}}B_t
  \label{eq:fbm-bridge-residual-support}
\end{equation}
be the centred fractional Brownian bridge obtained by conditioning at \(t\).
If in addition
\(t\notin K\), then
\begin{equation}
  \PP(\xi_s^{(t)}>0\text{ for every }s\in K)>0.
  \label{eq:fbm-bridge-compact-strict-positivity}
\end{equation}
The same assertions hold with all signs reversed.
\end{lemma}

For \(\alpha\in(1,2)\), let \(Z\) be the symmetric \(\alpha\)-stable process
normalised by
\begin{equation}
  \EE_0 e^{i\xi Z_t}=e^{-t|\xi|^\alpha},
  \label{eq:stable-normalization}
\end{equation}
and let \(T_0=\inf\{t>0:Z_t=0\}\).

\begin{lemma}
\label{lem:stable-green-killed-zero}
Set
\begin{equation}
  c_\alpha
  :=\frac1\pi\int_0^\infty\frac{1-\cos u}{u^\alpha}\,\mathrm du
  =-\frac{1}{2\Gamma(\alpha)\cos(\pi\alpha/2)}>0.
  \label{eq:stable-potential-constant}
\end{equation}
With respect to Lebesgue measure on \(\mathbb R\setminus\{0\}\), the process
\(Z\) killed at \(T_0\) has Green density
\begin{equation}
  g_\alpha(x,y)
  =c_\alpha\bigl(
       |x|^{\alpha-1}+|y|^{\alpha-1}-|x-y|^{\alpha-1}
     \bigr),
  \qquad x,y\neq0.
  \label{eq:stable-killed-green}
\end{equation}
\end{lemma}

\begin{proof}
The symmetric killed-potential identity in
\parencite[proof of Proposition~5.2, pp.~390--391]
{eisenbaum2003squaredGaussian} gives
\(g_\alpha(x,y)=a_\alpha(x)+a_\alpha(y)-a_\alpha(x-y)\), where under
\eqref{eq:stable-normalization}
\[
  a_\alpha(z)
  =\frac1\pi\int_0^\infty\frac{1-\cos(z\xi)}{\xi^\alpha}\,\mathrm d\xi
  =c_\alpha|z|^{\alpha-1}.
\]
Substitution gives \eqref{eq:stable-killed-green}.
\end{proof}

Taking \(\alpha=1+2H\) in
Lemma~\ref{lem:stable-green-killed-zero} and comparing
\eqref{eq:stable-killed-green} with \eqref{eq:fbm-covariance} gives the
covariance--Green identity
\begin{equation}
  g_\alpha(x,y)=2c_\alpha K_H(x,y),
  \qquad x,y\neq0.
  \label{eq:fbm-stable-green-identification}
\end{equation}

\begin{lemma}[{\cite[Theorem~4.3 and Lemma~4.4]
  {moench2026horizonProblem}}]
\label{lem:finite-wall-mtp2-closure}
Let \(F\subset\mathbb R\setminus\{0\}\) be finite and nonempty.  Then
\(B_F\) is nondegenerate and its precision matrix is Stieltjes: it is
symmetric positive definite with nonpositive off-diagonal entries.  Let
\(a\in\mathbb R^F\).  Here a coordinate conditioning means a normalised
density section obtained by fixing any subset of the current coordinates,
and a coordinate rectangle is a product of intervals in the current
coordinate space.  Any law obtained from \(a+B_F\) or \(a-B_F\) by finitely
many such conditionings, coordinate-rectangle restrictions, and coordinate
marginalisations is MTP\(_2\), provided every required section or restriction
normaliser is finite and strictly positive.  Every such law is therefore
associated, i.e., if \(X\) has one of these laws and \(f,g\) are bounded
coordinatewise nondecreasing functions, then
\begin{equation}
  \EE[f(X)g(X)]\geq \EE f(X)\,\EE g(X).
  \label{eq:association}
\end{equation}
Whenever the procedure uses only coordinate conditionings and
marginalisations and at least one coordinate remains, its result is a
nondegenerate Gaussian law with Stieltjes precision.  For a split into
retained and exposed blocks, write this precision as
\[
  P=\begin{pmatrix}P_{RR}&P_{RE}\\P_{ER}&P_{EE}\end{pmatrix}.
\]
Then the regression matrix of the retained block on the exposed block is
\(-P_{RR}^{-1}P_{RE}\), which is entrywise nonnegative because
\(P_{RR}^{-1}\geq0\) and \(P_{RE}\leq0\) entrywise.
\end{lemma}

Lemma~\ref{lem:finite-wall-mtp2-closure} is the key order and precision
interface used below.

\paragraph{Grid selection.}
\label{subsec:exact-grid-disintegration}

We now record the exact finite representation that will be used in the proof of part (A.1) of Theorem A.

\begin{lemma}[{\cite[Lemma~3.1 and Proposition~3.2]
  {moench2026horizonProblem}}]
\label{lem:unique-interior-maximizer}
Almost surely, \(B\) has a unique maximiser \(\tau\) on \([0,1]\), and
\begin{equation}
  0<\tau<1.
\end{equation}
For each fixed finite set of distinct deterministic times, the maximiser is
also unique almost surely.
\end{lemma}

We use a grid whose mesh in tangent coordinates is dyadic even when the zoom
scale is not.

\begin{proposition}
\label{prop:dyadic-grid-diagonal}
Let \(a_n\to\infty\) be any deterministic sequence of positive real numbers.
There exist deterministic dyadic numbers \(h_n\downarrow0\), with
\(h_n\leq a_n\), such that, with
\begin{equation}
  \Delta_n:=\frac{h_n}{a_n},\qquad
  M_n:=\left\lfloor\frac{a_n}{h_n}\right\rfloor,
  \qquad
  \mathcal G_n:=\{j\Delta_n:0\leq j\leq M_n\},
  \label{eq:physical-grid}
\end{equation}
the following holds.  If \(K_n\) is the maximising index of \(B\) on
\(\mathcal G_n\), \(\tau_n=K_n\Delta_n\), and
\begin{equation}
  \Xi_n^{\mathrm{grid}}(t)
  :=a_n^H\bigl(B_{\tau_n+t/a_n}-B_{\tau_n}\bigr),
  \label{eq:grid-centered-zoom}
\end{equation}
then
\begin{equation}
  d_{\mathrm{loc}}(\Xi_n^{\mathrm{grid}},\Xi_{a_n})\longrightarrow0
  \quad\text{almost surely}.
  \label{eq:grid-zoom-coupling}
\end{equation}
Moreover, writing \(T_n=M_n\Delta_n\),
\begin{equation}
  a_n\tau_n\longrightarrow\infty,
  \qquad
  a_n(T_n-\tau_n)\longrightarrow\infty
  \quad\text{almost surely}.
  \label{eq:selected-horizons-diverge}
\end{equation}
\end{proposition}

\begin{proof}
Fix \(\alpha\in(0,H)\) and \(\gamma>H/\alpha\).  For fixed \(n\) and
integer \(r\geq0\), put
\begin{equation}
  \mathcal G_{n,r}:=
  \left\{\frac{j2^{-r}}{a_n}:
  0\leq j\leq\lfloor a_n2^r\rfloor\right\},
  \label{eq:finer-fixed-scale-grids}
\end{equation}
and let \(\tau_{n,r}\) be its leftmost maximising time.  The grids become
dense in \([0,1]\) as \(r\to\infty\), and their right endpoints tend to
\(1\).  A grid point can therefore be chosen to converge to \(\tau\).  Its
value converges to \(B_\tau\); hence every subsequential limit of the grid
maximising times is also a maximiser of \(B\) on \([0,1]\), and so equals
\(\tau\) by Lemma~\ref{lem:unique-interior-maximizer}.  Thus, for every fixed
\(n\),
\begin{equation}
  \tau_{n,r}\longrightarrow\tau
  \quad\text{almost surely, hence in probability, as }r\to\infty.
  \label{eq:fixed-row-grid-argmax}
\end{equation}
The finite-set clause of the same lemma makes all these countably many grid
maximisers unique on one probability-one event.

Set \(r_0:=-1\).  By \eqref{eq:fixed-row-grid-argmax}, choose deterministic
integers \(r_n\) recursively so large that
\begin{equation}
  r_n>r_{n-1},\qquad
  h_n:=2^{-r_n}\leq\min\{2^{-n},a_n\},
  \qquad
  \PP(|\tau_{n,r_n}-\tau|>a_n^{-\gamma})\leq2^{-n}.
  \label{eq:grid-diagonal-choice}
\end{equation}
For the resulting grids in \eqref{eq:physical-grid},
\(\tau_n=\tau_{n,r_n}\) almost surely.  Borel--Cantelli gives
\(|\tau_n-\tau|\leq a_n^{-\gamma}\) eventually almost surely.

For every \(q\geq2\), Gaussian moments and \eqref{eq:fbm-increments} give
\(\EE|B_t-B_s|^q=C_q|t-s|^{qH}\).  Choosing \(q\) with
\(q(H-\alpha)>1\), the Kolmogorov continuity theorem and uniqueness of
continuous modifications show that the fixed continuous version has locally
\(\alpha\)-Hölder paths almost surely.  Work
on the intersection of this event with the preceding Borel--Cantelli event.
For each integer \(R\geq1\), eventually \(R/a_n\leq1\), so all four time
arguments below lie in \([-1,2]\).  If \(C_\alpha<\infty\) is a random
\(\alpha\)-Hölder constant on that interval, then
\begin{equation}
\begin{split}
  \sup_{|t|\leq R}
  |\Xi_n^{\mathrm{grid}}(t)-\Xi_{a_n}(t)|
  &\leq2C_\alpha a_n^H|\tau_n-\tau|^\alpha\\
  &\leq2C_\alpha a_n^{H-\gamma\alpha}\longrightarrow0.
\end{split}
\end{equation}
This holds for every integer \(R\) on the same probability-one event;
dominated convergence in the series \eqref{eq:local-uniform-metric} proves
\eqref{eq:grid-zoom-coupling}.  Finally,
\begin{equation}
  0\leq1-T_n<\Delta_n=\frac{h_n}{a_n}
  \leq\frac{2^{-n}}{a_n}\longrightarrow0,
\end{equation}
so \(\tau_n\to\tau\) and \(T_n-\tau_n\to1-\tau\).  Since
\(0<\tau<1\), eventually
\(\tau_n\geq\tau/2\) and
\(T_n-\tau_n\geq(1-\tau)/2\).  Multiplication by \(a_n\to\infty\)
proves both limits in \eqref{eq:selected-horizons-diverge}.
\end{proof}

\begin{proposition}
\label{prop:exact-pure-wall-mixture}
For the grids of Proposition~\ref{prop:dyadic-grid-diagonal}, define
\begin{equation}
  \Lambda_{n,k}:=\{h_n(j-k):0\leq j\leq M_n\},
  \qquad 0\leq k\leq M_n,
  \label{eq:relative-tangent-grid}
\end{equation}
and
\begin{equation}
  Q_{n,k}:=\mathcal L\left(
    B\,\middle|\, B_t\leq0\text{ for every }t\in\Lambda_{n,k}
  \right),
  \qquad
  p_{n,k}:=\PP(B_t\leq0:t\in\Lambda_{n,k}).
  \label{eq:mixture-components}
\end{equation}
Then
\(p_{n,k}=\PP(K_n=k)>0\),
\(\sum_{k=0}^{M_n}p_{n,k}=1\), and
\begin{equation}
  \mathcal L(\Xi_n^{\mathrm{grid}})
  =\sum_{k=0}^{M_n}p_{n,k}Q_{n,k}.
  \label{eq:exact-mixture}
\end{equation}
Write
\begin{equation}
  L_{n,k}^-=h_nk,
  \qquad L_{n,k}^+=h_n(M_n-k).
  \label{eq:component-horizons}
\end{equation}
For \(0<k<M_n\), the component \(Q_{n,k}\) is
\(Q_{h_n,L_{n,k}^-,L_{n,k}^+}\) from \eqref{eq:pure-wall-law}.  The endpoint
components \(k=0,M_n\) retain their direct conditional definitions in
\eqref{eq:mixture-components}; they are full-path laws constrained on only
one side of the pin and are not covered by the positive-two-horizon
definition \eqref{eq:pure-wall-law}.
\end{proposition}

\begin{proof}
Apply {\cite[Proposition~3.3]{moench2026horizonProblem}} with \(a=a_n\) and
\(\mathcal G=\mathcal G_n\).  Since its grid points are
\(t_j=j\Delta_n\), its relative wall is
\[
  \{a_n(t_j-t_k):j\neq k\}
  =\Lambda_{n,k}\setminus\{0\}.
\]
The omitted zero constraint is deterministic, so the cited component laws,
weights, and mixture are exactly \eqref{eq:mixture-components} and
\eqref{eq:exact-mixture}.  The extreme relative wall points are \(-h_nk\)
and \(h_n(M_n-k)\), which gives \eqref{eq:component-horizons} and the stated
interior and endpoint identifications.
\end{proof}

\paragraph{Finite-wall conditioning: order properties and tail bounds.}
\label{subsec:wall-order-structure}

For finite \(F\subset\mathbb D\), put
\begin{equation}
  A_F:=\{B_t\leq0:t\in F\}.
  \label{eq:reflected-wall-event}
\end{equation}
If \(\varnothing\neq T\subseteq F\subset\mathbb D\) are finite, write
\begin{equation}
  \nu_F^T:=\mathcal L(-B_T\mid A_F).
  \label{eq:reflected-coarse-law}
\end{equation}
For probability measures \(\mu,\nu\) on \([0,\infty)^T\), write
\(\mu\leq_{\mathrm{st}}\nu\) if
\(\int f\,\mathrm d\mu\leq\int f\,\mathrm d\nu\) for every bounded
coordinatewise nondecreasing Borel function \(f\).

\begin{proposition}[{\cite[Proposition~4.5]{moench2026horizonProblem}}]
\label{prop:wall-monotonicity}
If \(\varnothing\neq T\subseteq F\subseteq F'\subset\mathbb D\) are finite,
then
\begin{equation}
  \nu_F^T\leq_{\mathrm{st}}\nu_{F'}^T
  \label{eq:wall-stochastic-order}
\end{equation}
in the coordinatewise order on \([0,\infty)^T\).
\end{proposition}

For \(N\geq1\), define the nested cofinal constraint grids
\begin{equation}
  E_N:=\bigl(2^{-N}\mathbb Z\cap[-2^N,2^N]\bigr)
       \setminus\{0\}.
  \label{eq:canonical-walls}
\end{equation}
Then \(E_N\subseteq E_{N+1}\), \(\bigcup_NE_N=\mathbb D\), and every
finite subset of \(\mathbb D\) is contained in \(E_N\) for all sufficiently
large \(N\).

\begin{lemma}
\label{lem:cofinal-tightness}
Fix finite nonempty \(T\subset\mathbb D\), and choose \(N_T\) so that
\(T\subseteq E_N\) for every \(N\geq N_T\).  If
\((\nu_{E_N}^T)_{N\geq N_T}\) is tight, then, with
\begin{equation}
  \mathcal I_T
  :=\{F\subset\mathbb D:F\text{ finite and }T\subseteq F\}
\end{equation}
ordered by inclusion, the full net \((\nu_F^T)_{F\in\mathcal I_T}\) is tight
and has a unique weak accumulation point \(\nu^T\).  In particular,
if \((F_n)\) is any sequence of finite subsets of \(\mathbb D\) such that
every finite subset of \(\mathbb D\) is eventually contained in \(F_n\), then
\(T\subseteq F_n\) eventually and, after discarding the preceding terms,
\begin{equation}
  \nu_{F_n}^T\Longrightarrow\nu^T.
  \label{eq:cofinal-wall-limit}
\end{equation}
\end{lemma}

\begin{proof}
The index set \(\mathcal I_T\) is directed because it contains
\(F_1\cup F_2\) whenever it contains \(F_1\) and \(F_2\).  Put
\(S_T:=[0,\infty)^T\).  For \(R>0\), the set
\begin{equation}
  U_R:=\{x\in S_T:\max_{t\in T}x_t>R\}
\end{equation}
is increasing.  Tightness of the canonical tail implies
\[
  \sup_{M\geq N_T}\nu_{E_M}^T(U_R)\longrightarrow0
  \qquad(R\to\infty).
\]
Indeed, a compact set supplied by tightness is contained in
\([0,R]^T\) for some \(R\).  Given \(F\in\mathcal I_T\), canonical
cofinality gives an \(N\geq N_T\) with \(F\subseteq E_N\).
Proposition~\ref{prop:wall-monotonicity} therefore gives
\begin{equation}
  \nu_F^T(U_R)\leq\nu_{E_N}^T(U_R)
  \leq\sup_{M\geq N_T}\nu_{E_M}^T(U_R).
\end{equation}
This proves tightness of the full family.

If \(f\) is bounded, continuous, and increasing, then
\[
  a_F(f):=\int f\,\mathrm d\nu_F^T
\]
is an increasing bounded net by
Proposition~\ref{prop:wall-monotonicity}.  Consequently
\(a_F(f)\to a(f):=\sup_{F\in\mathcal I_T}a_F(f)\).
The space \(S_T\) is Polish, so Prokhorov's theorem and the tightness just
proved make the weak closure of the family compact.  In particular, weakly
convergent subnets exist.  Every such subnet has a limit \(\rho\) satisfying
\[
  \int f\,\mathrm d\rho=a(f)
\]
for every bounded continuous increasing \(f\).

These tests determine the probability measure.  Indeed, for
\(a=(a_t)_{t\in T}\in\mathbb R^T\), the functions
\[
  \phi_{m,a}(x)
  :=\prod_{t\in T}\bigl(1\wedge m(x_t-a_t)_+\bigr),
  \qquad m\geq1,
\]
are bounded, continuous, and increasing, and increase pointwise to the
indicator of \(\prod_{t\in T}(a_t,\infty)\).  Thus any two possible subnet
limits agree on every upper orthant.  These orthants form a \(\pi\)-system that
generates the Borel \(\sigma\)-field of \(S_T\), so the limits are equal.  Denote
the unique possible limit by \(\nu^T\).  Compactness now forces the full net
to converge to \(\nu^T\): otherwise a subnet outside some weak neighbourhood
of \(\nu^T\) would have a further convergent subnet whose limit is not
\(\nu^T\), a contradiction.

Finally, let \((F_n)\) satisfy the stated cofinality condition and discard
its finite prefix before \(F_n\in\mathcal I_T\).  For every
\(F\in\mathcal I_T\) and all sufficiently large \(n\),
\begin{equation}
  a_F(f)\leq a_{F_n}(f)\leq a(f).
\end{equation}
Taking the lower limit and then the supremum over \(F\) proves
\(a_{F_n}(f)\to a(f)\).  This sequence is tight because it belongs to the
full tight family.  Hence every subsequence has a weakly convergent
subsubsequence, and the preceding integral limit and determining-class
argument identify its limit as \(\nu^T\).  Were the whole sequence not to
converge to \(\nu^T\), a subsequence staying outside some weak neighbourhood
of \(\nu^T\) would contradict this conclusion.  This proves
\eqref{eq:cofinal-wall-limit}.
\end{proof}

\begin{lemma}[{\cite[Corollary~4.13]{moench2026horizonProblem}}]
\label{lem:uniform-anchor-tail}
For every \(t\in\mathbb D\), there are constants \(c_t>0\) and
\(M_t<\infty\) such that, for every finite \(F\subset\mathbb D\) containing
\(t\),
\begin{equation}
  \PP(-B_t>x\mid A_F)\leq e^{-c_tx},
  \qquad x\geq M_t.
  \label{eq:uniform-anchor-tail}
\end{equation}
\end{lemma}

\subsubsection{Proof of the finite-dimensional hard-wall limit}

The preceding results now give the finite-dimensional entrance law.

\begin{theorem}
\label{thm:finite-dimensional-entrance-law}
For every finite nonempty \(T\subset\mathbb D\), the full family
\(\{\nu_F^T:F\in\mathcal I_T\}\) is tight, and the directed net
\((\nu_F^T)_{F\in\mathcal I_T}\) converges to a unique law \(\nu^T\) as
\(F\uparrow\mathbb D\).  The family \((\nu^T)\) is projectively consistent.
Consequently, if
\begin{equation}
  h_n\downarrow0,
  \qquad L_n^-\to\infty,
  \qquad L_n^+\to\infty
  \label{eq:wall-parameter-limit}
\end{equation}
with dyadic, grid-aligned parameters, then
\begin{equation}
  \mathcal L_{Q_{h_n,L_n^-,L_n^+}}(B_T)
  \Longrightarrow\mu^T:=(-\operatorname{id})_*\nu^T.
  \label{eq:finite-dyadic-limit}
\end{equation}
The limit is independent of the relative rates in
\eqref{eq:wall-parameter-limit}.
\end{theorem}

\begin{proof}
Fix finite nonempty \(T\subset\mathbb D\), and choose \(N_T\) so large that
\(T\subset E_N\) for every \(N\geq N_T\).  By
Lemma~\ref{lem:uniform-anchor-tail}, for
\(R\geq\max_{t\in T}M_t\),
\[
  \sup_{N\geq N_T}
  \nu_{E_N}^T\{y:\max_{t\in T}y_t>R\}
  \leq\sum_{t\in T}e^{-c_tR}
  \longrightarrow0
  \qquad(R\to\infty).
\]
Thus the canonical tail is tight.  Lemma~\ref{lem:cofinal-tightness} now
gives tightness of the full family and the unique directed limit \(\nu^T\).

If \(T\subset T'\), choose a canonical tail containing \(T'\).  On that
tail, coordinate projection \(\pi_{T',T}:[0,\infty)^{T'}\to[0,\infty)^T\)
satisfies
\[
  (\pi_{T',T})_*\nu_{E_N}^{T'}=\nu_{E_N}^T.
\]
Both canonical sequences converge to their directed limits, so the
continuous mapping theorem gives
\((\pi_{T',T})_*\nu^{T'}=\nu^T\).  This proves projective consistency.

Finally, for a parameter sequence satisfying
\eqref{eq:wall-parameter-limit}, put
\(F_n=\Lambda(h_n,L_n^-,L_n^+)\).  Write \(h_n=2^{-m_n}\).  Since
\(h_n\downarrow0\), one has \(m_n\to\infty\).  If
\(s=k2^{-q}\in\mathbb D\), then \(s/h_n=k2^{m_n-q}\in\mathbb Z\) for all
large \(n\); the appropriate horizon also eventually contains \(s\).
Hence every finite subset of \(\mathbb D\) is eventually contained in
\(F_n\).  In particular, eventually \(T\subset F_n\) and
\[
  \mathcal L_{Q_{h_n,L_n^-,L_n^+}}(-B_T)=\nu_{F_n}^T.
\]
The cofinal-sequence conclusion of Lemma~\ref{lem:cofinal-tightness}, followed
by the continuous sign change, proves \eqref{eq:finite-dyadic-limit}.

For precision, let \(d_{\mathrm{BL}}\) be any bounded-Lipschitz metric for
weak convergence on \(\mathbb R^T\).  The final uniformity statement means
that, for every \(\varepsilon>0\), there are \(m_0\) and \(R_0<\infty\) such
that
\[
  d_{\mathrm{BL}}\!\left(
    \mathcal L_{Q_{h,L^-,L^+}}(B_T),\mu^T\right)<\varepsilon
\]
whenever \(h=2^{-m}\) with \(m\geq m_0\),
\(L^-,L^+\geq R_0\), and \(L^-/h,L^+/h\in\mathbb N\).  If this failed,
fix an integer \(r_0\) and recursively choose a violating triple with mesh
exponent \(r_n\geq\max\{n,r_{n-1}+1\}\) and
\(L_n^-,L_n^+\geq n\).  Then \(2^{-r_n}\downarrow0\), so these triples form
an admissible sequence and contradict the convergence just proved.  The
converse implication is immediate, so this is exactly sequential uniformity
and imposes no restriction on the relative rates.
\end{proof}

\subsection{Path-space tightness}
\label{sec:path-tightness}

We now
lift the finite-dimensional convergence to locally uniform convergence.
Fix an arbitrary parameter sequence satisfying
\eqref{eq:wall-parameter-limit}, and abbreviate
\begin{equation}
  Q_n:=Q_{h_n,L_n^-,L_n^+},
  \qquad
  m_n(t):=\EE_{Q_n}B_t,
  \qquad
  Z_n(t):=B_t-m_n(t).
  \label{eq:centered-wall-process}
\end{equation}
We control \(Z_n\) and \(m_n\) separately.

\subsubsection{Centred fluctuations}

\begin{lemma}
\label{lem:harge-convex-domination}
Let \(Y\) have a centred, possibly singular, Gaussian law on
\(\mathbb R^d\), let \(C\subseteq\mathbb R^d\) be Borel and convex with
\(\PP(Y\in C)>0\), and let
\(b=\EE(Y\mid Y\in C)\).  For every nonnegative convex Borel function
\(\varphi:\mathbb R^d\to[0,\infty)\),
\begin{equation}
  \EE\left[\varphi(Y-b)\,\middle|\,Y\in C\right]
  \leq \EE\varphi(Y),
  \label{eq:harge-full-convex-domination}
\end{equation}
whenever the right-hand side is finite.  In particular, for every linear
functional \(\ell\) and every \(\lambda\in\mathbb R\),
\begin{equation}
  \EE\left[
    e^{\lambda\ell(Y-b)}\,\middle|\,Y\in C
  \right]
  \leq
  \exp\left(\frac{\lambda^2}{2}
  \operatorname{Var}(\ell(Y))\right).
  \label{eq:harge-mgf-bound}
\end{equation}
\end{lemma}

\begin{proof}
Theorem~1.1 of \textcite{harge2004convexGaussian} allows \(\gamma\) to be
singular and states that, if \(f\geq0\) is log-concave and \(g\) is convex,
then
\begin{equation}
  \frac{\int g(x+l_\gamma-m_f)f(x)\,\mathrm d\gamma(x)}
       {\int f\,\mathrm d\gamma}
  \leq\int g\,\mathrm d\gamma.
  \label{eq:harge-cited-theorem}
\end{equation}
Here
\begin{equation}
  l_\gamma:=\int x\,\mathrm d\gamma(x),
  \qquad
  m_f:=\frac{\int xf(x)\,\mathrm d\gamma(x)}{\int f\,\mathrm d\gamma},
\end{equation}
and all displayed integrals must be well defined.

Take \(\gamma=\mathcal L(Y)\) and \(f=\mathds1_C\).  In the extended-log
convention, \(\log\mathds1_C\) is zero on \(C\) and \(-\infty\) outside;
convexity of \(C\) therefore makes \(f\) log-concave.  Moreover,
\(\int f\,\mathrm d\gamma=\PP(Y\in C)>0\), the Gaussian first moment makes
\(b\) finite, and \(l_\gamma=0\), \(m_f=b\).  Hence
\eqref{eq:harge-cited-theorem} gives
\eqref{eq:harge-full-convex-domination} whenever \(\varphi\) has at most
linear growth.

For the general test in the statement, let
\[
  \varphi^*(p):=\sup_x\{\langle p,x\rangle-\varphi(x)\}
\]
and define
\[
  \varphi_k(x):=max\left\{0,
    \sup_{|p|\leq k}\{\langle p,x\rangle-\varphi^*(p)\}\right\}.
\]
Fenchel--Moreau and nonnegativity of \(\varphi\) give
\(0\leq\varphi_k\uparrow\varphi\).  Each \(\varphi_k\) is convex and
\[
  \varphi_k(x)\leq \varphi(0)+k|x|,
\]
because \(\varphi^*(p)\geq-\varphi(0)\).  Thus all integrals in Harg\'e's
theorem are finite for \(g=\varphi_k\), both under \(\gamma\) and under the
conditioned translate, and
\[
  \EE[\varphi_k(Y-b)\mid Y\in C]\leq\EE\varphi_k(Y).
\]
Monotone convergence proves \eqref{eq:harge-full-convex-domination}; its
right-hand side is finite by hypothesis.

Finally take \(\varphi(x)=e^{\lambda\ell(x)}\).  This is nonnegative and
convex for every \(\lambda\in\mathbb R\), and centred Gaussianity gives
\[
  \EE e^{\lambda\ell(Y)}
  =\exp\left(\frac{\lambda^2}{2}\operatorname{Var}(\ell(Y))\right),
\]
which is \eqref{eq:harge-mgf-bound}, also when the variance is zero.
\end{proof}

\begin{lemma}
\label{lem:uniform-centered-modulus}
For every \(R<\infty\) and \(\varepsilon>0\),
\begin{equation}
  \lim_{\delta\downarrow0}\sup_n
  Q_n\left(
    \sup_{\substack{s,t\in[-R,R]\\|s-t|\leq\delta}}
    |Z_n(t)-Z_n(s)|>\varepsilon
  \right)=0.
  \label{eq:centered-modulus}
\end{equation}
\end{lemma}

\begin{proof}
Put \(F_n:=\Lambda(h_n,L_n^-,L_n^+)\) and
\[
  A_n:=\{B_u\leq0:u\in F_n\}.
\]
This event has positive probability by the observation following
\eqref{eq:pure-wall-law}.  For fixed \(s,t\), form the centred, possibly
singular, Gaussian vector
\[
  Y=(B_u)_{u\in F_n\cup\{s,t\}}.
\]
In this coordinate space, \(A_n=\{Y\in C_n\}\), where \(C_n\) is the closed
convex Borel set imposing \(x_u\leq0\) for \(u\in F_n\).  The conditional
barycentre of \(Y\), evaluated by the linear functional
\(\ell(x)=x_t-x_s\), is \(m_n(t)-m_n(s)\).  Thus
Lemma~\ref{lem:harge-convex-domination} gives
\begin{equation*}
  \EE_{Q_n}e^{\lambda(Z_n(t)-Z_n(s))}
  \leq e^{\lambda^2|t-s|^{2H}/2},
  \qquad \lambda\in\mathbb R.
\end{equation*}
Here the right-hand variance is exactly
\(\operatorname{Var}(B_t-B_s)=|t-s|^{2H}\).  If \(s\neq t\), Chernoff's
bound with \(\lambda=u/|t-s|^{2H}\) applied to the increment and its negative
yields
\begin{equation}
  Q_n(|Z_n(t)-Z_n(s)|>u)
  \leq2\exp\left(-\frac{u^2}{2|t-s|^{2H}}\right),
  \qquad u>0,\quad s\neq t.
  \label{eq:centered-increment-tail}
\end{equation}
For \(s=t\), the probability on the left is zero, so the zero-variance branch
requires no division.

For \(k\geq0\), let
\begin{equation}
  D_k:=2^{-k}\mathbb Z\cap[-R-1,R+1]
\end{equation}
and define
\[
  M_{n,k}:=\max\left\{
    |Z_n(v)-Z_n(u)|:u,v\in D_k,\ v-u=2^{-k}
  \right\}.
\]
Choose \(c_R\geq1\) so that there are at most \(c_R2^k\) such edges, and put
\(\sigma_k=2^{-Hk}\) and
\[
  a_k:=\sigma_k\sqrt{2\log(2c_R2^k)}.
\]
The union bound and \eqref{eq:centered-increment-tail} give
\[
  Q_n(M_{n,k}>u)
  \leq\min\left\{1,
    2c_R2^k e^{-u^2/(2\sigma_k^2)}
  \right\}.
\]
Since, for \(a>0\),
\[
  \int_a^\infty e^{-u^2/(2\sigma^2)}\,\mathrm du
  \leq\frac{\sigma^2}{a}e^{-a^2/(2\sigma^2)},
\]
integration at the threshold \(a_k\) yields
\[
  \EE_{Q_n}M_{n,k}
  \leq a_k+\frac{\sigma_k^2}{a_k}.
\]
Consequently, after enlarging a constant depending only on \(R\) and \(H\),
\begin{equation}
  \EE_{Q_n}M_{n,k}\leq C_R2^{-Hk}\sqrt{k+1}.
  \label{eq:dyadic-edge-maximum}
\end{equation}

For completeness, \(Z_n\) has continuous paths.  The paths of \(B\) are
continuous, while, for fixed \(n\),
\begin{equation}
  |m_n(t)-m_n(s)|
  \leq\EE_{Q_n}|B_t-B_s|
  \leq\frac{|t-s|^H}{\PP(A_n)^{1/2}},
\end{equation}
by Cauchy--Schwarz under the unconditioned law.  The positive quantity
\(\PP(A_n)\) may depend on \(n\); this bound is used only for continuity, not
uniformly in \(n\).

Let \(\pi_k(r)=2^{-k}\lfloor2^kr\rfloor\).  Continuity and telescoping over
dyadic ancestors give
\begin{equation}
  |Z_n(r)-Z_n(\pi_m(r))|
  \leq\sum_{k=m+1}^\infty M_{n,k}.
\end{equation}
Indeed,
\(\pi_{k+1}(r)-\pi_k(r)\in\{0,2^{-(k+1)}\}\), and the one-unit padding in
\(D_k\) contains every ancestor of every \(r\in[-R,R]\).  If
\(|s-t|\leq2^{-m}\), the integers indexing \(\pi_m(s)\) and \(\pi_m(t)\)
differ by at most one.  The two ancestor tails and this possible level-\(m\)
edge therefore give
\begin{equation}
  \sup_{\substack{s,t\in[-R,R]\\|s-t|\leq2^{-m}}}
  |Z_n(t)-Z_n(s)|
  \leq2\sum_{k=m}^\infty M_{n,k}.
\end{equation}
By \eqref{eq:dyadic-edge-maximum}, Tonelli's theorem, and Markov's
inequality, the probability of this modulus exceeding \(\varepsilon\) is at
most
\begin{equation}
  \frac{C_R}{\varepsilon}
  \sum_{k=m}^\infty2^{-Hk}\sqrt{k+1},
\end{equation}
which tends to zero as \(m\to\infty\).  Monotonicity in \(\delta\) proves
\eqref{eq:centered-modulus}.  The proof used no feature of \(A_n\) beyond
being a positive-probability Borel convex constraint on finitely many
coordinates.  Its constants are therefore uniform over every such finite
conditioning, including the shifted orthants and rescaled finite walls used
below, after centring under the corresponding conditional law.
\end{proof}

\subsubsection{Conditional means}

For finite \(F\subset\mathbb D\), retain the notation \(A_F\) from
\eqref{eq:reflected-wall-event} and put
\begin{equation}
  m_F(t):=\EE(B_t\mid A_F).
  \label{eq:reflected-conditional-mean}
\end{equation}
For the constraint set defining \(Q_n\), one has \(m_n=m_F\).

\begin{lemma}
\label{lem:positive-boundary-flux}
Assume \(0<H<1/2\).  For every finite nonempty
\(F\subset\mathbb D\), there are coefficients
\(\lambda_{F,u}\geq0\), \(u\in F\), such that
\begin{equation}
  -m_F(t)=\sum_{u\in F}\lambda_{F,u}K_H(t,u),
  \qquad t\in\mathbb R.
  \label{eq:positive-covariance-mixture}
\end{equation}
In particular, \(m_F(t)\leq0\).
\end{lemma}

\begin{proof}
Let \(X=-B_F\).  By Lemma~\ref{lem:finite-wall-mtp2-closure}, its covariance
\(C\) is positive definite.  Thus \(P=C^{-1}\) exists, and \(X\) has an
everywhere positive smooth Gaussian density \(p\).  In particular,
\[
  q_F:=\PP(X\geq0)=\PP(A_F)>0,
  \qquad \bar x:=\EE(X\mid X\geq0)
\]
is finite.  The density satisfies
\(\partial_i p(x)=-(Px)_ip(x)\).  Moreover,
\(|(Px)_i|p(x)\) is integrable, and, for fixed \(x_{-i}\),
\(p(x_i,x_{-i})\to0\) as \(x_i\to\infty\).  Fubini's theorem and the
one-dimensional fundamental theorem of calculus therefore give
\begin{equation}
\begin{split}
  (P\bar x)_i
  &=\frac1{q_F}
    \int_{[0,\infty)^F}(Px)_ip(x)\,\mathrm dx\\
  &=-\frac1{q_F}
    \int_{[0,\infty)^{F\setminus\{i\}}}
      \int_0^\infty\partial_i p(x_i,x_{-i})
      \,\mathrm dx_i\,\mathrm dx_{-i}\\
  &=\frac1{q_F}
    \int_{[0,\infty)^{F\setminus\{i\}}}
      p(0,x_{-i})\,\mathrm dx_{-i}>0.
  \label{eq:positive-boundary-flux}
\end{split}
\end{equation}
The strict inequality also covers singleton \(F\), with the usual
zero-dimensional integral convention.

Set \(\lambda_F=P\bar x\).  Since \(A_F=\{X\geq0\}\) and
\[
  \operatorname{Cov}(-B_t,X)
  =\operatorname{Cov}(-B_t,-B_F)=K_H(t,F),
\]
Gaussian regression gives
\[
  \EE(-B_t\mid X)=K_H(t,F)PX.
\]
Averaging this identity over \(X\geq0\) yields
\begin{equation}
  -m_F(t)=K_H(t,F)P\bar x=K_H(t,F)\lambda_F,
\end{equation}
which is \eqref{eq:positive-covariance-mixture}.  Finally, with
\(\alpha=2H\in(0,1)\), subadditivity gives, for all real \(t,u\),
\[
  |t-u|^\alpha
  \leq(|t|+|u|)^\alpha
  \leq|t|^\alpha+|u|^\alpha.
\]
Hence \(K_H(t,u)\geq0\), and the representation proves \(m_F(t)\leq0\).
In fact, strict subadditivity and \(\lambda_{F,u}>0\) give
\(m_F(t)<0\) whenever \(t\neq0\), although only the weak form is used below.
\end{proof}

\begin{lemma}
\label{lem:fixed-reference-mean}
For every fixed \(r\in\mathbb D\),
\begin{equation}
  \sup\{-m_F(r):F\subset\mathbb D\text{ finite}\}<\infty.
  \label{eq:fixed-reference-mean-bound}
\end{equation}
\end{lemma}

\begin{proof}
Fix finite \(F\subset\mathbb D\), put \(G=F\cup\{r\}\), and let
\(Y=-B_r\).  Every nonempty finite wall event has positive probability by
Lemma~\ref{lem:finite-wall-mtp2-closure} and positivity of its Gaussian density, while
\(\PP(A_\varnothing)=1\).  Thus
\[
  \rho_F:=\mathcal L(-B_G\mid A_F)
\]
is well defined and is MTP\(_2\), hence associated, by
Lemma~\ref{lem:finite-wall-mtp2-closure}.  For
\(C=\{Y\geq0\}\),
\[
  \rho_F(C)=\frac{\PP(A_G)}{\PP(A_F)}>0,
  \qquad A_F\cap C=A_G.
\]
Put \(\phi_M(y)=\max\{-M,\min\{y,M\}\}\).  Apply association to the
bounded increasing function \(\phi_M\) and the increasing indicator
\(\mathds1_C\).  Since \(|\phi_M(Y)|\leq|Y|\), conditional Gaussian
first-moment integrability and dominated convergence under \(\rho_F\) and
\(\rho_F(\,\cdot\mid C)\) give
\begin{equation}
  -m_F(r)=\EE_{\rho_F}Y
  \leq\EE_{\rho_F}(Y\mid C)
  =-m_G(r).
\end{equation}
Lemma~\ref{lem:uniform-anchor-tail} supplies \(c_r>0\) and
\(M_r<\infty\), independent of \(G\), such that
\(\PP(-B_r>x\mid A_G)\leq e^{-c_rx}\) for \(x\geq M_r\).
Since \(-B_r\geq0\) under \(A_G\), the layer-cake formula gives
\[
  -m_F(r)\leq-m_G(r)
  \leq M_r+\int_{M_r}^{\infty}e^{-c_rx}\,\mathrm dx<\infty.
\]
The final bound is independent of \(F\), proving
\eqref{eq:fixed-reference-mean-bound}.
\end{proof}

\begin{lemma}
\label{lem:covariance-kernel-ratio}
Fix \(0<a<R<\infty\) and \(r\in(0,a)\).  There is
\(C=C(H,a,R,r)<\infty\) such that, for all \(s,t\in[a,R]\) with
\(|s-t|\leq1\) and all \(u\neq0\),
\begin{equation}
  |K_H(t,u)-K_H(s,u)|
  \leq C|t-s|^{2H}K_H(r,u).
  \label{eq:annular-kernel-ratio}
\end{equation}
The reflected statement holds on \([-R,-a]\) with a fixed
\(r\in(-a,0)\).
\end{lemma}

\begin{proof}
Put \(\beta=2H\in(0,1)\).  For \(u\neq0\), the triangle inequality and
strict subadditivity give
\[
  |r-u|^\beta\leq(r+|u|)^\beta<r^\beta+|u|^\beta.
\]
Consequently,
\begin{equation}
  K_H(r,u)>0,\qquad u\neq0.
  \label{eq:strict-positive-reference-kernel}
\end{equation}

Choose \(\eta\in(0,\min\{r/2,a/2,1\})\) sufficiently small.  The
mean-value theorem on \([r/2,3r/2]\) gives
\(r^\beta-|r-u|^\beta=O(|u|)\) for \(|u|\leq\eta\).  Hence, since
\(|u|=o(|u|^\beta)\) as \(u\to0\), shrinking \(\eta\) if necessary gives
\begin{equation}
  K_H(r,u)=\frac12|u|^\beta+O(|u|),
  \qquad K_H(r,u)\geq c_0|u|^\beta,
  \qquad 0<|u|\leq\eta.
  \label{eq:reference-kernel-near-zero}
\end{equation}
For such \(u\), put \(f_u(x)=x^\beta-(x-u)^\beta\) on \([a,R]\); the choice
of \(\eta\) ensures \(x-u>0\).  Since
\[
  f_u'(x)=\beta\bigl(x^{\beta-1}-(x-u)^{\beta-1}\bigr)
\]
and the derivative of \(z\mapsto z^{\beta-1}\) is bounded on
\([a/2,R+a/2]\), one has
\[
  \sup_{x\in[a,R]}|f_u'(x)|\leq C_0|u|.
\]
The common
\(|u|^\beta\) term cancels from the kernel difference.  Thus, if
\(d=|t-s|\leq1\), then \(d\leq d^\beta\) and
\(|u|\leq|u|^\beta\), so
\begin{equation}
\begin{split}
  |K_H(t,u)-K_H(s,u)|
  &=\frac12|f_u(t)-f_u(s)|\leq C_0d|u|\\
  &\leq C_0d^\beta|u|^\beta
  \leq C_1d^\beta K_H(r,u),
  \qquad 0<|u|\leq\eta.
  \label{eq:kernel-ratio-near-zero}
\end{split}
\end{equation}

As \(u\to+\infty\), the difference
\(u^\beta-(u-r)^\beta\) tends to zero; as \(u\to-\infty\), writing
\(v=-u\), the difference \(v^\beta-(v+r)^\beta\) also tends to zero.
Therefore
\[
  K_H(r,u)\longrightarrow\frac12r^\beta
  \qquad(u\to+\infty\text{ or }u\to-\infty).
\]
Together with continuity and \eqref{eq:strict-positive-reference-kernel},
this gives
\[
  c_\eta:=\inf_{|u|\geq\eta}K_H(r,u)>0.
\]
The elementary power inequality
\begin{equation}
  \bigl||x|^\beta-|y|^\beta\bigr|\leq|x-y|^\beta
  \label{eq:power-holder}
\end{equation}
follows directly from subadditivity and the reverse triangle inequality.
It remains valid when \(u=s\) or \(u=t\), so it covers the moving cusp.
For every \(|u|\geq\eta\), it gives
\begin{equation}
  |K_H(t,u)-K_H(s,u)|
  \leq\frac12(d^\beta+d^\beta)
  \leq c_\eta^{-1}d^\beta K_H(r,u).
\end{equation}
The near-zero and away-from-zero estimates prove
\eqref{eq:annular-kernel-ratio}.  Finally,
\(K_H(-x,-y)=K_H(x,y)\).  For negative-annulus points, set
\(s'=-s\), \(t'=-t\), \(r'=-r\), and \(u'=-u\), and apply the positive-
annulus estimate to \((s',t',r',u')\).  This proves the reflected statement.
\end{proof}

\begin{proposition}
\label{prop:conditional-mean-modulus}
For every \(0<R<\infty\),
\begin{equation}
  \lim_{\delta\downarrow0}
  \sup_{\substack{F\subset\mathbb D\\F\text{ finite}}}\;
  \sup_{\substack{s,t\in[-R,R]\\|s-t|\leq\delta}}
  |m_F(t)-m_F(s)|=0.
  \label{eq:conditional-mean-modulus}
\end{equation}
In particular, the same conclusion holds with the outer supremum replaced
by \(\sup_n\) and \(m_F\) replaced by \(m_n\) for every sequence satisfying
\eqref{eq:wall-parameter-limit}.
\end{proposition}

\begin{proof}
First fix \(0<a<R\), and choose once and for all a positive dyadic
\(r<a\).  For nonempty finite \(F\),
Lemmas~\ref{lem:positive-boundary-flux},~\ref{lem:fixed-reference-mean},
and~\ref{lem:covariance-kernel-ratio} give, whenever
\(s,t\in[a,R]\) and \(d:=|t-s|\leq1\),
\begin{equation}
\begin{split}
  |m_F(t)-m_F(s)|
  &\leq\sum_{u\in F}\lambda_{F,u}
       |K_H(t,u)-K_H(s,u)|\\
  &\leq C|t-s|^{2H}
       \sum_{u\in F}\lambda_{F,u}K_H(r,u)\\
  &=C|t-s|^{2H}(-m_F(r))
  \leq C_{H,a,R}|t-s|^{2H}.
  \label{eq:annular-mean-modulus}
\end{split}
\end{equation}
The constant is independent of \(F\), and the empty-wall case is the zero
function.  If \(d>1\), partition \([s,t]\) into at most
\(\lceil R-a\rceil+1\) subintervals of length at most one (with the
endpoints ordered if necessary).  Summing the
local estimate and using \(d^{2H}\geq1\), then increasing
\(C_{H,a,R}\), proves \eqref{eq:annular-mean-modulus} for every
\(s,t\in[a,R]\).  Applying the reflected part of
Lemma~\ref{lem:covariance-kernel-ratio} with reference \(-r\), followed by the
same finite-chain argument, gives the same bound on \([-R,-a]\), after taking
the maximum of the two constants.

We next control the pin.  Specialise the preceding construction to
\([1,2]\), with reference \(r=1/2\), and use the base point \(1\).
Put
\(M_1:=\sup\{-m_F(1):F\subset\mathbb D\text{ finite}\}<\infty\) by
Lemma~\ref{lem:fixed-reference-mean}.  This bound and
\eqref{eq:annular-mean-modulus} give, uniformly over finite \(F\),
\[
  0\leq-m_F(t)
  \leq-m_F(1)+|m_F(t)-m_F(1)|
  \leq M_1+C_{H,1,2},
  \qquad t\in[1,2],
\]
where the empty-wall case is again immediate.  Thus, with a constant
depending only on \(H\),
\begin{equation}
  \sup_{\substack{F\subset\mathbb D\\F\text{ finite}}}\;
  \sup_{t\in[1,2]}(-m_F(t))\leq B_H<\infty.
  \label{eq:unit-annulus-mean-bound}
\end{equation}
Let \(t\in\mathbb D\cap(0,\infty)\) and choose \(k\in\mathbb Z\) such that
\(\theta=2^{-k}t\in[1,2)\); put \(q=2^k\).  Self-similarity gives the exact
joint-vector identity
\[
  (B_{qx})_{x\in\{\theta\}\cup q^{-1}F}
  \deq q^H(B_x)_{x\in\{\theta\}\cup q^{-1}F}.
\]
Since multiplication by \(q^H>0\) preserves the wall inequalities,
the event \(A_F\) on the left corresponds exactly to
\(A_{q^{-1}F}\) on the right.  Hence
\begin{equation}
  m_F(t)=q^Hm_{q^{-1}F}(\theta).
  \label{eq:conditional-mean-scaling}
\end{equation}
Because \(q\) is a power of two, the set \(q^{-1}F\) is finite and dyadic.
Moreover, \(q=t/\theta\leq t\), since \(\theta\geq1\).  Therefore
\begin{equation}
  0\leq-m_F(t)\leq B_Ht^H.
  \label{eq:positive-pin-mean-bound}
\end{equation}
Time reflection sends the wall as well as the target:
\[
  m_F(t)=m_{-F}(-t),
  \qquad t<0.
\]
This proves the same bound for negative dyadic \(t\).  For nonempty \(F\),
formula \eqref{eq:positive-covariance-mixture} makes \(m_F\) a finite sum of
continuous kernels; for \(F=\varnothing\), it is identically zero.  Dyadic
approximation, together with \(m_F(0)=0\), therefore extends the estimate to
every real \(t\):
\begin{equation}
  |m_F(t)|\leq C_H|t|^H.
  \label{eq:pin-mean-bound}
\end{equation}

It remains to prove the full-family modulus.  Fix \(R>0\) and
\(\varepsilon>0\).  Choose
\(a\in(0,\min\{R,1\})\) so small that
\(2C_H(2a)^H<\varepsilon/2\).  Let \(C^{\mathrm{ann}}_{H,a,R}\) be the maximum
of the positive- and negative-annulus constants above, and choose
\(\delta_0\in(0,a)\) such that
\(C^{\mathrm{ann}}_{H,a,R}\delta_0^{2H}<\varepsilon/2\).

Fix any finite \(F\), \(0<\delta\leq\delta_0\), and
\(s,t\in[-R,R]\) with \(|s-t|\leq\delta\).  If
\(\min\{|s|,|t|\}\leq a\), then the triangle inequality puts both points
in \([-2a,2a]\), and \eqref{eq:pin-mean-bound} gives
\[
  |m_F(t)-m_F(s)|
  \leq |m_F(t)|+|m_F(s)|
  \leq2C_H(2a)^H<\varepsilon/2.
\]
Otherwise \(|s|,|t|>a\).  They cannot have opposite signs, because then
\(|s-t|=|s|+|t|>2a>\delta\).  Thus they lie together in either
\([a,R]\) or \([-R,-a]\), and
\eqref{eq:annular-mean-modulus} gives
\[
  |m_F(t)-m_F(s)|
  \leq C^{\mathrm{ann}}_{H,a,R}\delta^{2H}<\varepsilon/2.
\]
Both estimates are independent of \(F\), proving
\eqref{eq:conditional-mean-modulus}.  Since every \(m_n\) is one of the
functions \(m_F\), the stated sequential consequence follows as well.
\end{proof}

\begin{theorem}
\label{thm:path-space-entrance-law}
There is a unique probability law \(Q\) on \(\mathcal X\) whose marginal on
every finite nonempty \(T\subset\mathbb D\) is \(\mu^T\) from
Theorem~\ref{thm:finite-dimensional-entrance-law}.  For every sequence
satisfying \eqref{eq:wall-parameter-limit}, the laws \((Q_n)\) are tight in
\(\mathcal X\) and converge weakly to \(Q\).  The convergence is uniform in
the sequential sense over all dyadic, grid-aligned parameters.  Moreover,
\begin{equation}
  Q\bigl(f(0)=0,\ f(t)\leq0\text{ for every }t\in\mathbb R\bigr)=1.
  \label{eq:limit-nonpositive-support}
\end{equation}
\end{theorem}

\begin{proof}
For \(R\geq1\), write
\[
  \omega_R(f,\delta)
  :=\sup_{\substack{s,t\in[-R,R]\\|s-t|\leq\delta}}|f(t)-f(s)|.
\]
Lemma~\ref{lem:uniform-centered-modulus},
Proposition~\ref{prop:conditional-mean-modulus}, and
\(B=m_n+Z_n\) give, for every \(\varepsilon>0\),
\[
  \lim_{\delta\downarrow0}\sup_n
  Q_n\bigl(\omega_R(B,\delta)>\varepsilon\bigr)=0.
\]
Indeed, first make the deterministic quantity
\(\sup_n\omega_R(m_n,\delta)\) at most \(\varepsilon/2\), and then apply the
centred estimate at threshold \(\varepsilon/2\).

We verify the compact-interval tightness criterion explicitly.  Fix
\(\eta>0\), and choose \(\delta_k\downarrow0\) so that
\[
  \sup_n Q_n\bigl(\omega_R(B,\delta_k)>k^{-1}\bigr)
  \leq\eta2^{-k},
  \qquad k\geq1.
\]
The set
\[
  K_{R,\eta}:=\left\{
    f\in C([-R,R]):f(0)=0,\quad
    \omega_R(f,\delta_k)\leq k^{-1}\text{ for every }k
  \right\}
\]
is closed, because every displayed modulus is continuous under uniform
convergence, and is uniformly equicontinuous.  Its \(k=1\) condition, applied
along a finite \(\delta_1\)-chain from zero, also makes it uniformly bounded.
Hence \(K_{R,\eta}\) is compact by Arzelà--Ascoli.  Since \(B_0=0\) under
every \(Q_n\), the union bound gives
\[
  \inf_n Q_n\{B|_{[-R,R]}\in K_{R,\eta}\}\geq1-\eta.
\]
Thus the restriction laws are tight in \(C([-R,R])\).

This restriction statement yields one compact set in the local-uniform path
space, as follows.  Given \(\eta>0\), for each integer \(R\geq1\) choose a
compact \(K_R\subset C([-R,R])\) whose complement has \(Q_n\)-probability at
most \(\eta2^{-R}\), uniformly in \(n\), and put
\[
  K:=\bigcap_{R=1}^{\infty}
  \{f\in\mathcal X:f|_{[-R,R]}\in K_R\}.
\]
Every sequence in \(K\) has, by successive restriction to
\([-1,1],[-2,2],\ldots\), a diagonal subsequence converging uniformly on
every compact interval to a compatible limit that still belongs to every
\(K_R\).  Hence \(K\) is compact in \(\mathcal X\), and
\[
  \inf_nQ_n(K)\geq1-\sum_{R=1}^{\infty}\eta2^{-R}=1-\eta.
\]
This proves tightness in \(\mathcal X\).

The metric space \((\mathcal X,d_{\mathrm{loc}})\) is Polish.  Indeed, a
\(d_{\mathrm{loc}}\)-Cauchy sequence is uniformly Cauchy on each
\([-R,R]\), and its compatible interval limits define a member of
\(\mathcal X\); piecewise-linear functions with finitely many rational
breakpoints and rational values, extended constantly beyond their extreme
breakpoints, form a countable dense subset.  Prokhorov's theorem therefore
gives weakly convergent subsequences.

Suppose \(Q_{n_j}\Rightarrow Q'\).  For every finite nonempty
\(T\subset\mathbb D\), the evaluation map
\(e_T(f)=(f(t))_{t\in T}\) is continuous.  The continuous mapping theorem
and Theorem~\ref{thm:finite-dimensional-entrance-law}, whose cofinality
argument shows that \(T\) is eventually contained in the constraint grid,
give
\[
  (e_T)_*Q'=\mu^T.
\]

These marginals determine a law on \(\mathcal X\), not merely an individual
path.  Let \(\mathcal G=\sigma(e_t:t\in\mathbb D)\).  Continuity of each
evaluation gives \(\mathcal G\subseteq\mathcal B(\mathcal X)\).  Conversely,
for every \(f,g\in\mathcal X\) and integer \(R\geq1\), density of
\(\mathbb D\) and continuity give
\[
  \|f-g\|_{C([-R,R])}
  =\sup_{t\in\mathbb D\cap[-R,R]}|f(t)-g(t)|.
\]
Thus \(f\mapsto d_{\mathrm{loc}}(f,g)\) is \(\mathcal G\)-measurable.  Since
\(\mathcal X\) is separable, it has a countable base of such metric balls,
so \(\mathcal B(\mathcal X)\subseteq\mathcal G\).  Equality of every finite
dyadic marginal therefore makes two probability laws agree on the cylinder
\(\pi\)-system generating \(\mathcal G=\mathcal B(\mathcal X)\), and hence agree
on all Borel sets.

Let \(d_{\mathrm{BL}}\) be any bounded-Lipschitz metric for weak convergence
on \(\mathcal X\).
Tightness and Prokhorov give at least one subsequential limit.  The preceding
identification shows that every possible limit, from the fixed parameter
sequence or from any other admissible sequence, is the same law \(Q\).  If
the original sequence did not converge to \(Q\), a subsequence would stay a
positive bounded-Lipschitz distance from \(Q\); tightness would give it a
further convergent subsequence, whose limit would have to be \(Q\), a
contradiction.  This proves \eqref{eq:pure-wall-path-limit} for every
admissible sequence.

For precision, sequential uniformity means that, for every
\(\varepsilon>0\), there are \(r_0\) and \(R_0<\infty\) such that
\[
  d_{\mathrm{BL}}(Q_{h,L^-,L^+},Q)<\varepsilon
\]
whenever \(h=2^{-r}\) with \(r\geq r_0\),
\(L^-,L^+\geq R_0\), and \(L^-/h,L^+/h\in\mathbb N\).  If this failed, fix
an integer \(q_0\) and, for \(n\geq1\), recursively choose a violating triple
with exponent \(q_n\geq\max\{n,q_{n-1}+1\}\) and
\(L_n^-,L_n^+\geq n\).  These triples
would satisfy \eqref{eq:wall-parameter-limit} but fail to converge to \(Q\),
contradicting what was just proved.

Finally fix one admissible sequence converging to \(Q\).  For
\(t\in\mathbb D\), eventual wall membership gives
\(Q_n\{f:f(t)\leq0\}=1\) for all large \(n\).  This coordinate half-space is
closed in \(\mathcal X\), so Portmanteau gives
\[
  1=\limsup_{n\to\infty}Q_n\{f:f(t)\leq0\}
  \leq Q\{f:f(t)\leq0\}.
\]
The closed pin set is treated identically using \(B_0=0\) under every
\(Q_n\).  Intersecting over the countable set \(\mathbb D\), and then using
continuity at every real time, proves
\eqref{eq:limit-nonpositive-support}.
\end{proof}

\begin{proof}[Proof of Theorem~\ref{thm:main-conditional-limit}]
The convergence, its sequential uniformity, and nonpositive support are
Theorem~\ref{thm:path-space-entrance-law}.

Let \((\Theta f)(t):=f(-t)\), a continuous map on \(\mathcal X\), and take the
admissible symmetric sequence
\[
  h_m=2^{-m},\qquad L_m^-=L_m^+=2^m.
\]
The centred Gaussian law of \(B\) is invariant under \(\Theta\), and both the
grid \(\Lambda(h_m,2^m,2^m)\) and its wall event are preserved by reflection.
Consequently
\(\Theta_*Q_{h_m,2^m,2^m}=Q_{h_m,2^m,2^m}\).  The path-space hard-wall
limit theorem gives \(Q_{h_m,2^m,2^m}\Rightarrow Q\), while the continuous
mapping theorem gives
\(\Theta_*Q_{h_m,2^m,2^m}\Rightarrow\Theta_*Q\).  Uniqueness of weak limits
therefore yields
\begin{equation}
  \Theta X\deq X.
  \label{eq:rough-tangent-reflection-invariance}
\end{equation}
This establishes the hard-wall convergence in
Theorem~\ref{thm:main-conditional-limit} and the reflection invariance needed
later.  The identification of \(Q\) with the tangent law follows from the
proof of part~\ref{thm:main-tangent-limit} immediately below, while the
remaining strict-support assertion \eqref{eq:limit-support} is proved at the
beginning of Section~\ref{subsec:rough-continuum-window-selection}.
\end{proof}

\subsection{Assembly at the random maximiser}
\label{sec:assembly}

\begin{proof}[Proof of Theorem~\ref{thm:main-tangent-limit}]
Theorem~\ref{thm:path-space-entrance-law} constructs a unique two-sided
hard-wall law \(Q\).  It remains to identify that law as the tangent limit at
the random maximiser and prove self-similarity.

Let \(a_n\to\infty\) be an arbitrary deterministic sequence of positive
numbers.  Choose the
dyadic tangent-grid diagonal of Proposition~\ref{prop:dyadic-grid-diagonal}
and use the exact mixture in Proposition~\ref{prop:exact-pure-wall-mixture}.
Along the selected centre, \eqref{eq:component-horizons} and
\eqref{eq:selected-horizons-diverge} give
\begin{equation}
  Y_n:=\min(L_{n,K_n}^-,L_{n,K_n}^+)
  =\min\{a_n\tau_n,a_n(T_n-\tau_n)\}\longrightarrow\infty
  \quad\text{almost surely}.
  \label{eq:random-minimum-horizon}
\end{equation}
For each integer \(j\geq1\), dominated convergence gives
\(\PP(Y_n<j)\to0\).  Choose deterministic integers
\(N_1<N_2<\cdots\), with \(N_j\geq j\), such that
\[
  \PP(Y_n<j)\leq2^{-j}\qquad(n\geq N_j),
\]
and set \(R_n=1\) for \(n<N_1\) and \(R_n=j\) when
\(N_j\leq n<N_{j+1}\).  Then \(R_n\to\infty\), and
Proposition~\ref{prop:exact-pure-wall-mixture} gives the exact identity
\begin{equation}
  \alpha_n
  :=\sum_{\substack{0\leq k\leq M_n:\\
       \min(L_{n,k}^-,L_{n,k}^+)<R_n}}p_{n,k}
  =\PP(Y_n<R_n)\longrightarrow0.
  \label{eq:short-horizon-mixture-mass}
\end{equation}
Indeed, on \(N_j\leq n<N_{j+1}\), the last probability is at most \(2^{-j}\).

Use the standard bounded-Lipschitz metric
\[
  d_{\mathrm{BL}}(\mu,\nu)
  :=\sup_{\substack{\|\varphi\|_\infty\leq1\\
        \operatorname{Lip}_{d_{\mathrm{loc}}}(\varphi)\leq1}}
      \left|\int\varphi\,\mathrm d\mu-\int\varphi\,\mathrm d\nu\right|,
\]
which metrises weak convergence on the Polish space \(\mathcal X\) and is at
most \(2\).  Since \(R_n\geq1\), every component with both horizons at least
\(R_n\) is an interior component and hence, by
Proposition~\ref{prop:exact-pure-wall-mixture}, is the two-horizon law
\(Q_{h_n,L_{n,k}^-,L_{n,k}^+}\).  Its horizons are grid-aligned.  The uniform
form of Theorem~\ref{thm:path-space-entrance-law}, together with
\(h_n\downarrow0\) through dyadic values and \(R_n\to\infty\), therefore gives
\begin{equation}
  \sup_{\substack{0\leq k\leq M_n:\\
       \min(L_{n,k}^-,L_{n,k}^+)\geq R_n}}
  d_{\mathrm{BL}}(Q_{n,k},Q)\longrightarrow0.
  \label{eq:long-horizon-uniform-limit}
\end{equation}
The supremum is over a nonempty set for all large \(n\), since otherwise
\(\alpha_n=1\), contrary to \eqref{eq:short-horizon-mixture-mass}.

For every admissible test function in the definition of
\(d_{\mathrm{BL}}\), linearity in the exact mixture
\eqref{eq:exact-mixture} and the triangle inequality apply.  Taking the
supremum over those functions gives, for all sufficiently large \(n\),
\[
\begin{split}
  d_{\mathrm{BL}}(\mathcal L(\Xi_n^{\mathrm{grid}}),Q)
  &\leq\sum_{k=0}^{M_n}p_{n,k}d_{\mathrm{BL}}(Q_{n,k},Q)\\
  &\leq2\alpha_n+
    \sup_{\substack{0\leq k\leq M_n:\\
          \min(L_{n,k}^-,L_{n,k}^+)\geq R_n}}
       d_{\mathrm{BL}}(Q_{n,k},Q).
\end{split}
\]
Consequently, \eqref{eq:short-horizon-mixture-mass} and
\eqref{eq:long-horizon-uniform-limit} imply
\begin{equation}
  d_{\mathrm{BL}}(\mathcal L(\Xi_n^{\mathrm{grid}}),Q)
  \longrightarrow0.
  \label{eq:grid-zoom-law-limit}
\end{equation}
The metric \(d_{\mathrm{loc}}\) is bounded by \(1\), so the almost-sure coupling
\eqref{eq:grid-zoom-coupling} and dominated convergence yield
\[
  d_{\mathrm{BL}}\bigl(
     \mathcal L(\Xi_n^{\mathrm{grid}}),\mathcal L(\Xi_{a_n})
  \bigr)
  \leq\EE d_{\mathrm{loc}}(\Xi_n^{\mathrm{grid}},\Xi_{a_n})
  \longrightarrow0.
\]
Together with \eqref{eq:grid-zoom-law-limit}, this proves
\(\mathcal L(\Xi_{a_n})\Rightarrow Q\).  Since every deterministic positive
sequence \(a_n\to\infty\) has this limit, the sequential criterion for the
metrised weak topology proves \eqref{eq:main-tangent-convergence}.

Finally, fix \(c>0\) and define
\((\mathcal D_cf)(t):=c^Hf(t/c)\).  This map is continuous on \(\mathcal X\):
if \(f_m\to f\) locally uniformly, then, for every \(R<\infty\),
\[
  \|\mathcal D_cf_m-\mathcal D_cf\|_{C([-R,R])}
  \leq c^H\|f_m-f\|_{C([-R/c,R/c])}\longrightarrow0.
\]
The zooms satisfy the pathwise identity
\begin{equation}
  \Xi_{ca}=\mathcal D_c\Xi_a,
  \qquad c>0.
  \label{eq:zoom-scaling-identity}
\end{equation}
As \(a\to\infty\), the continuous mapping theorem gives
\(\mathcal D_c\Xi_a\Rightarrow(\mathcal D_c)_*Q\), whereas
\(ca\to\infty\) and \eqref{eq:main-tangent-convergence} give
\(\Xi_{ca}\Rightarrow Q\).  The identity
\eqref{eq:zoom-scaling-identity} and uniqueness of weak limits therefore give
\((\mathcal D_c)_*Q=Q\), which is \eqref{eq:limit-selfsimilarity}.
\end{proof}

\subsection{Rerooting-rescaling invariance}
\label{sec:rough-window-rerooting-proof}

Throughout this section put \(\alpha=2H\in(0,1)\), and fix for now
\begin{equation}
  0<u<v,\qquad I=[u,v].
  \label{eq:rough-fixed-positive-window}
\end{equation}
Choose a fixed dyadic \(s_0\in(u,v)\).  Constants may depend on this fixed
window.  Negative windows and windows containing zero are treated at the end.
We first isolate the exact finite experiment created by a second
maximum-centred zoom.  Let \(F\subset\mathbb D\) be finite, let
\(D\subset F\cap I\) be nonempty, and work under
\begin{equation}
  \mu_F:=\mathcal L(B\mid B_u\leq0:\ u\in F).
  \label{eq:reroot-finite-wall-law}
\end{equation}
Finite Gaussian ties are null events, so the maximiser \(S_D\) on \(D\) is unique.
For \(s\in D\) and \(b>0\), define the reflected zoom
\begin{equation}
  G^{b,s}(t):=b^H\bigl(B_s-B_{s+t/b}\bigr),
  \qquad t\in\mathbb R,
  \label{eq:reroot-reflected-zoom}
\end{equation}
and the relative grids
\begin{equation}
  J_{b,s}:=b(D-s),\qquad K_{b,s}:=b(F-s),\qquad v_{b,s}:=-bs.
  \label{eq:reroot-relative-grids}
\end{equation}
Since every \(w\in D\) satisfies \(u\leq w\leq v\),
\begin{equation}
  \operatorname{dist}(v_{b,s},J_{b,s})
  =b\min_{w\in D}w
  \geq bu
  \geq\frac{u}{v}|v_{b,s}|.
  \label{eq:rough-general-window-old-root-separation}
\end{equation}
Before conditioning, \(G^{b,s}\) is a standard two-sided fractional Brownian
motion.  The event
consisting of the original wall and the selection \(S_D=s\) is exactly
\begin{equation}
 \begin{split}
  A_{b,s}
  :=\{G^{b,s}(j)&\geq0:\ j\in J_{b,s}\}\\
   &\cap
   \{G^{b,s}(k)\geq G^{b,s}(v_{b,s}):\ k\in K_{b,s}\}.
 \end{split}
 \label{eq:reroot-two-root-event}
\end{equation}
Indeed, the first wall says that \(B_s\) is the window-grid maximum, while
\begin{equation}
  G^{b,s}(v_{b,s})=b^HB_s\leq0
  \label{eq:reroot-old-root-height}
\end{equation}
and the second wall is exactly \(B_u\leq0\) for \(u\in F\).  This also
fixes the sign: the rerooted path in
\eqref{eq:rough-window-rerooting-map} is \(-G^{b,s}\).

\subsubsection{The selected-germ comparison}
\label{subsec:rough-selected-germ}

\paragraph{Finite-dimensional factorisation.}

The distant root in \eqref{eq:reroot-two-root-event} is handled by time
inversion.  For a continuous path pinned at zero, put
\begin{equation}
  (\mathcal I_Hf)(u):=|u|^\alpha f(1/u),\qquad u\neq0.
  \label{eq:rough-time-inversion-map}
\end{equation}
The covariance identity
\begin{equation}
  |uw|^\alpha K_H(1/u,1/w)=K_H(u,w)
  \label{eq:rough-time-inversion-covariance}
\end{equation}
shows that \(\mathcal I_HB\) is again a fractional Brownian motion.  Under
this map the old root
\(v=v_{b,s}\) becomes \(\rho=1/v\), and
\begin{equation}
  B_k\geq B_v
  \quad\Longleftrightarrow\quad
  \widehat B_r\geq
  \left(\frac{|r|}{|\rho|}\right)^\alpha\widehat B_\rho,
  \qquad r=1/k,
  \label{eq:rough-inverted-selected-wall}
\end{equation}
where \(\widehat B=\mathcal I_HB\).

Since the selection probability of a prescribed fine-grid point may vanish
with the mesh, we use a relative comparison.

Fix constants \(0<c_0<C_0<\infty\), a nonempty finite set
\(T=\{t_1,\ldots,t_m\}\subset\mathbb R\setminus\{0\}\), and
\(A,M\in[0,\infty)\).  For each sufficiently small \(\delta>0\), choose
\(\rho\) and a finite set \(N_\delta\) satisfying
\begin{equation}
  c_0\delta\leq|\rho|\leq C_0\delta,\qquad
  N_\delta\subset(-C_0\delta,C_0\delta)\setminus\{0\},
  \label{eq:selected-germ-geometry}
\end{equation}
with \(\rho\in N_\delta\), and let
\(P_\delta\subset\mathbb R\setminus(-c_0\delta,c_0\delta)\) be an arbitrary
finite constraint set containing \(T\) and disjoint from \(N_\delta\).  Put
\[
  E_\delta:=\{\widehat B_p\geq0:\ p\in P_\delta\setminus T\}.
\]
For \(z\in[-M\delta^H,0]\), set
\begin{equation}
 \begin{split}
  D_{\delta,z}
  :=\{\widehat B_r&\geq(|r|/|\rho|)^\alpha z:\
       r\in N_\delta\setminus\{\rho\}\},\\
  g_{\delta,z}^T(\mathbf x)
  :=\PP\bigl(&D_{\delta,z}\mid
       E_\delta,\widehat B_T=\mathbf x,\widehat B_\rho=z\bigr).
 \end{split}
 \label{eq:selected-germ-continuation}
\end{equation}
At exact values
\((\widehat B_T,\widehat B_\rho)=(\mathbf x,z)\), we use the corresponding
normalised Gaussian density section, and we write \(j_T(\mathbf x,z)\) for
the joint density of \((\widehat B_T,\widehat B_\rho)\) under \(E_\delta\).
Lemma~\ref{lem:finite-wall-mtp2-closure} makes the underlying finite Gaussian
law nondegenerate, so
\(g_{\delta,z}^T(\mathbf x)\) and \(j_T(\mathbf x,z)\) are strictly positive
throughout the domains below and all logarithmic ratios are pointwise defined.

\begin{lemma}
\label{lem:rough-selected-germ-factorization}
Under the preceding assumptions and notation, as \(\delta\downarrow0\),
\begin{equation}
 \sup_{\substack{\mathbf x,\mathbf y\in[-A,A]^m\\
                   -M\delta^H\leq z\leq0}}
 \left|
  \log\frac{g_{\delta,z}^T(\mathbf y)}
            {g_{\delta,z}^T(\mathbf x)}
 \right|\longrightarrow0
 \label{eq:selected-germ-vector-harnack}
\end{equation}
uniformly in the choices of the finite grids and their cardinalities.
Moreover,
\begin{equation}
 \sup_{\substack{\mathbf x,\mathbf y\in[-A,A]^m\\
                  |z|,|z'|\leq M\delta^H}}
 \left|
  \log\frac{j_T(\mathbf y,z)j_T(\mathbf x,z')}
            {j_T(\mathbf x,z)j_T(\mathbf y,z')}
 \right|\longrightarrow0
 \label{eq:selected-germ-pin-factorization}
\end{equation}
with the same uniformity.
\end{lemma}

\begin{proof}
We first give the scalar argument and then a finite-vector extension.  Fix
\(t\neq0\), put
\begin{equation}
  c_r=(|r|/|\rho|)^\alpha,\qquad
  L_r=\widehat B_r-c_r\widehat B_\rho,
  \label{eq:selected-germ-linear-forms}
\end{equation}
and expose \(\widehat B_\rho=z\).  Time inversion gives the exact formulas
\begin{equation}
 \begin{split}
  L_r&=|r|^\alpha
       \{B_{1/r}-B_{1/\rho}\},\\
  \operatorname{Var}(L_r)
   &=\left(\frac{|r|\,|r-\rho|}{|\rho|}\right)^\alpha.
 \end{split}
 \label{eq:selected-germ-exact-variance}
\end{equation}
Moreover, after writing \(r=\delta q\) and \(\rho=\delta q_0\), Taylor
expansion at the fixed point \(t\) gives
\begin{equation}
  \frac{|\operatorname{Cov}(L_r,\widehat B_t)|}
       {|t|^\alpha\sqrt{\operatorname{Var}(L_r)}}
  \leq C_t\delta^{1-H}
  \frac{|q-q_0(|q|/|q_0|)^\alpha|}
       {|q|^H|q-q_0|^H}.
  \label{eq:selected-germ-standardized-response}
\end{equation}
The quotient extends boundedly across \(q=0,q_0\): its numerator is
\(O(|q|^\alpha)\) at zero and \(O(|q-q_0|)\) at \(q_0\).

Once the height \(z\) is fixed, every remaining constraint is a coordinate
lower wall.  Lemma~\ref{lem:finite-wall-mtp2-closure}, applied after fixing
the exposed coordinates, shows that the residual law is MTP\(_2\).  When
\(x\geq0\), fractional Brownian motion rooted at \(1/\rho\) has exposed
values
\begin{equation}
  -z/|\rho|^\alpha\geq0,\qquad
  x/|t|^\alpha-z/|\rho|^\alpha\geq0.
  \label{eq:selected-germ-positive-exposed-values}
\end{equation}
The regression matrix onto these coordinates is entrywise nonnegative, so
the conditional mean of every \(L_r\) is then nonnegative.  For arbitrary
\(x\in\mathbb R\), put
\[
 \begin{aligned}
  \mu_{x,z}(r)
   &:=\EE(\widehat B_r\mid\widehat B_t=x,\widehat B_\rho=z),\\
  \xi_r&:=\widehat B_r-\mu_{x,z}(r),\qquad
  \sigma_r:=\|\xi_r\|_2,\\
  h_{x,z}(r)&:=\frac{c_rz-\mu_{x,z}(r)}{\sigma_r}.
 \end{aligned}
\]
With
\[
 \begin{aligned}
  K_\rho(r,t)
   &:=K_H(r,t)
      -\frac{K_H(r,\rho)K_H(\rho,t)}{|\rho|^\alpha},\\
  V_\rho(t)
   &:=|t|^\alpha-\frac{K_H(t,\rho)^2}{|\rho|^\alpha},
 \end{aligned}
\]
successive Gaussian regression gives the exact identities
\[
 \begin{aligned}
  \mu_{y,z}(r)-\mu_{x,z}(r)
   &=(y-x)\frac{K_\rho(r,t)}{V_\rho(t)},\\
  \sigma_r^2
   &=|r|^\alpha-\frac{K_H(r,\rho)^2}{|\rho|^\alpha}
     -\frac{K_\rho(r,t)^2}{V_\rho(t)}.
 \end{aligned}
\]
Lemma~\ref{lem:finite-wall-mtp2-closure} makes the conditional bivariate
Gaussian law MTP\(_2\) and associated.  Applying association to bounded
truncations of its coordinate maps and then using Gaussian integrability
gives \(K_\rho(r,t)\geq0\).
The centred residual process \(\xi\) and its covariance do not depend on
\((x,z)\).  Consequently \(h_{x,z}(r)\) is nonincreasing in \(x\), and the
preceding sign check gives \(h_{x,z}(r)\leq0\) whenever \(x\geq0\).  Thus
\(D_{\delta,z}=\{\xi_r/\sigma_r\geq h_{x,z}(r):
r\in N_\delta\setminus\{\rho\}\}\) after the two values have been exposed.

Partition
\((-C_0\delta,C_0\delta)\setminus\{0,\rho\}\) into intervals \(J\).  For each
interval put \(a_J:=\operatorname{dist}(J,\{0,\rho\})\), and choose the
partition so that \(\operatorname{diam}(J)\asymp a_J\).  Near each pin, take
dyadic intervals with \(a_J\asymp2^{-j}\delta\), \(j\geq0\); cover the part
whose distance from both pins is comparable to \(\delta\) by finitely many
intervals with \(a_J\asymp\delta\).  Discard intervals that do not meet
\(N_\delta\setminus\{\rho\}\).  On an interval \(J\) of scale \(a=a_J\), the
two-pin residual \(\xi_r\) satisfies
\begin{equation}
  c a^H\leq\|\xi_r\|_2\leq Ca^H,\qquad
  \left\|\frac{\xi_r}{\|\xi_r\|_2}
          -\frac{\xi_s}{\|\xi_s\|_2}\right\|_2
  \leq C(|r-s|/a)^H.
  \label{eq:selected-germ-whitney-metric}
\end{equation}
On intervals approaching zero or \(\rho\), these estimates follow from
concavity of \(u^\alpha\): the one-pin conditional variance is comparable
to \(a^\alpha\), whereas \(K_\rho(r,t)=O(a^\alpha)\), so the final
rank-one correction above is \(O(a^{2\alpha})=o(a^\alpha)\).  On the
finitely many remaining intervals they follow by compactness after scaling
by \(\delta\).  Orthogonal projection contracts \(L^2\)-increments, and
\(|\sigma_r-\sigma_s|\leq\|\xi_r-\xi_s\|_2\); these facts and the variance
bounds give the normalised metric estimate.  The normalised bridge
covariance kernels form a compact nondegenerate family.  The
Cameron--Martin argument in the proof of
\cite[Lemma~4.1]{moench2026horizonProblem} applies to every member: in
the rescaled interval choose a smooth bump equal to two on the interval and
supported away from all exposed pins, so that it belongs to the
corresponding finite-pin bridge Cameron--Martin space.  Weak continuity of
this compact Gaussian family and
a finite covering give a uniform positive open-ball probability.  The
metric estimate in \eqref{eq:selected-germ-whitney-metric} gives a uniform
entropy bound for the expected supremum.  Consequently there are constants
\(p,C>0\), uniform over every nonempty finite intersection
\(J\cap(N_\delta\setminus\{\rho\})\), such that
\begin{equation}
  \PP\left(
    \frac{\xi_r}{\|\xi_r\|_2}\geq1:
    r\in J\cap N_\delta
  \right)\geq p,\qquad
  \EE\sup_{r\in J\cap N_\delta}\frac{-\xi_r}{\|\xi_r\|_2}\leq C.
  \label{eq:selected-germ-whitney-baseline}
\end{equation}

Suppose first that \(-A\leq x\leq y\leq A\), and put
\(x^+=\max(x,0)\).  Since the regression response to \(\widehat B_t\) is
nonnegative,
\(h_{y,z}\leq h_{x,z}\) and \(h_{x^+,z}\leq0\).  Uniformly for small
\(\delta\), \(V_\rho(t)\) is bounded away from zero.  On a pin interval
\(K_\rho(r,t)=O(a^\alpha)\), while on a remaining interval it is
\(O(\delta^\alpha)\); division by the corresponding standard deviation in
\eqref{eq:selected-germ-whitney-metric} gives
\[
 \begin{aligned}
  0\leq h_{x,z}(r)-h_{y,z}(r)&\leq\eta_J,\\
  0\leq\bigl(h_{x,z}(r)\bigr)_+
    &\leq h_{x,z}(r)-h_{x^+,z}(r)\leq\varepsilon_J,\\
  \max(\eta_J,\varepsilon_J)&\leq
  \begin{cases}
    C_Aa^H,&J\text{ approaches }0\text{ or }\rho,\\
    C_A\delta^H,&\operatorname{dist}(J,\{0,\rho\})\asymp\delta,
  \end{cases}
 \end{aligned}
\]
and both \(h_{x,z}\) and \(h_{x,z}-h_{y,z}\), after rescaling \(J\) by
\(a\), have \(H\)-Hölder modulus bounded by \(C_{A,M}\), uniformly in
\(z\in[-M\delta^H,0]\).  This deterministic modulus is needed because the
wall drift need not be constant on \(J\).

For \(\eta_J>0\), use a net whose spacing after rescaling \(J\) by \(a\) is
\(\eta_J^{2/H}\).  The deterministic modulus just recorded and
\eqref{eq:selected-germ-whitney-metric}, chaining, and the Borell--TIS
inequality
\parencite[Theorem~2.1.1]{adlerTaylor2007randomFields}
show that the probability of a net-approximation error exceeding
\(\eta_J\) is \(O(e^{-c/\eta_J^2})\).  On the net we use the
Nazarov inequality
\parencite[Proposition~3.1]{chernozhukovChetverikovKatoKoike2023bootstrap}:
for unit-variance Gaussian coordinates,
\begin{equation}
  \PP\{u<\max_{1\leq i\leq p}Z_i\leq u+\eta\}
  \leq\eta(\sqrt{2\log p}+2).
  \label{eq:selected-germ-nazarov}
\end{equation}
Thus the wall-relaxation probability on \(J\) is at most
\(C\eta_J\sqrt{\log(e/\eta_J)}\).  Uniformly in \(J\),
\(\varepsilon_J\leq C_A\delta^H\), so for small \(\delta\) the positive-level
event in \eqref{eq:selected-germ-whitney-baseline} is contained in the
stronger wall.  Its probability is therefore at least \(p\).

If \(C_1\subset C_2\) are the stronger and weaker constraints on \(J\), and
\(D\) contains every other increasing constraint, association under the
restriction to \(C_2\) gives
\begin{equation}
  \frac{\PP(C_2\cap D)}{\PP(C_1\cap D)}
  \leq\frac{\PP(C_2)}{\PP(C_1)}.
  \label{eq:selected-germ-associated-shell-ratio}
\end{equation}
The denominator on the right is bounded below by
\eqref{eq:selected-germ-whitney-baseline}.  Summing over the dyadic interval
families approaching the two pins uses
\begin{equation}
  \sum_{j\geq0}(2^{-j}\delta)^H
       \sqrt{\log(e2^j/\delta)}
  \leq C_H\delta^H\sqrt{\log(e/\delta)}.
  \label{eq:selected-germ-whitney-sum}
\end{equation}
Let \(\nu\) be the common centred conditional Gaussian residual law after
\(\widehat B_t\) and \(\widehat B_\rho\) have been exposed.  In its
coordinates, write \(N_x\) for all selected-germ walls and \(P_x\) for all
outer walls when the exposed value is \(x\).  Nonnegative regression gives
\[
  N_x\subset N_y,\qquad P_x\subset P_y,\qquad
  g_{\delta,z}^{\{t\}}(x)=\nu(N_x\mid P_x).
\]
Lemma~\ref{lem:finite-wall-mtp2-closure}, first restricting by the outer walls
and then marginalising, shows that the last conditional probability is
nondecreasing in \(x\).  Conversely, association under \(\nu(\,\cdot\mid
P_y)\) shows that further conditioning on the stronger increasing event
\(P_x\) can only raise the probability of \(N_y\).  Hence
\[
  1\leq
  \frac{g_{\delta,z}^{\{t\}}(y)}
       {g_{\delta,z}^{\{t\}}(x)}
  \leq
  \frac{\nu(N_y\mid P_x)}{\nu(N_x\mid P_x)}
  \leq
  \exp\!\left\{C_{A,M}\delta^H
                   \sqrt{\log(e/\delta)}\right\}.
\]
The last inequality is the shell telescope above, with \(P_x\) retained
among the other increasing constraints at every step.  Interchanging
\(x,y\) proves the scalar version of
\eqref{eq:selected-germ-vector-harnack}.  At no point is a lower bound used
for the probability of the full selected-minimum event.

For completeness, the integration over \(z\) uses the following mixed
log-density estimate for an orthant-conditioned Gaussian law.  If
\((X,Z,Y)\) is Gaussian with Stieltjes precision, \(X,Z\) are scalar, and
\(Y\) is conditioned to the positive orthant, write its full precision as
\(P\), let \(Q=P_{YY}\) be the conditional precision of \(Y\) given
\((X,Z)\), and let \(S\) be the marginal precision of \((X,Z)\).
The case of an empty \(Y\)-block is the unconstrained Gaussian calculation
and is immediate.  Otherwise the natural parameter of
\(Y\mid X=x,Z=z\) is \(b_Xx+b_Zz\), where
\(b_X=-P_{YX}\geq0\) and \(b_Z=-P_{YZ}\geq0\).  Differentiation gives
\begin{equation}
 \partial_x\partial_z\log\PP(Y\geq0\mid X=x,Z=z)
 =-b_X^{\mathsf T}Q^{-1}b_Z
  +b_X^{\mathsf T}\operatorname{Cov}_{x,z}(Y\mid Y\geq0)b_Z.
 \label{eq:selected-germ-orthant-mixed-derivative}
\end{equation}
Association gives the lower covariance bound.  Translating after increasing
one natural parameter and using rectangle monotonicity gives the entrywise
upper bound
\begin{equation}
  0\leq\operatorname{Cov}_{x,z}(Y\mid Y\geq0)\leq Q^{-1}.
  \label{eq:selected-germ-orthant-covariance-contraction}
\end{equation}
Indeed, increasing the \(k\)-th natural parameter by \(a>0\) translates the
untruncated Gaussian by \(aQ^{-1}e_k\geq0\); replacing the translated lower
wall by the stronger zero wall raises every coordinate.  Dividing the
resulting mean inequality by \(a\) and sending \(a\downarrow0\) gives the
upper bound, while association gives the lower one.

The Schur complement identity is
\[
  b_X^{\mathsf T}Q^{-1}b_Z=P_{XZ}-S_{XZ}.
\]
Both \(P_{XZ}\) and \(S_{XZ}\) are nonpositive.  Adding the mixed derivative
\(-S_{XZ}\) of the unconstrained marginal Gaussian density to
\eqref{eq:selected-germ-orthant-mixed-derivative} therefore yields
\begin{equation}
  0\leq -P_{XZ}
  \leq\partial_x\partial_z\log j(x,z)\leq-S_{XZ},
  \label{eq:selected-germ-density-cross-harnack}
\end{equation}
where \(j\) is the joint density of \((X,Z)\) under the outer hard-wall
event.  For \(X=\widehat B_t\), \(Z=\widehat B_\rho\), the marginal precision
is the inverse of their two-point covariance matrix, so
\[
  -S_{XZ}
  =\frac{K_H(t,\rho)}
  {|t|^\alpha|\rho|^\alpha-K_H(t,\rho)^2}
  =O_t(1).
\]
Integration over
\(|z|,|z'|\leq M\delta^H\) proves the scalar form of
\eqref{eq:selected-germ-pin-factorization}.

For a fixed vector \(T\), write \(h_{\mathbf x,z}\) for the standardised
selected wall after exposing
\((\widehat B_T,\widehat B_\rho)=(\mathbf x,z)\), and let
\(\mathbf x^+\) be the coordinatewise positive part of \(\mathbf x\).
In the process rooted at \(R=1/\rho\), the exposed values corresponding to
\((\mathbf x^+,z)\) are all nonnegative:
\[
  -\frac{z}{|\rho|^\alpha}\geq0,\qquad
  \frac{x_i^+}{|t_i|^\alpha}-\frac{z}{|\rho|^\alpha}\geq0,
  \quad 1\leq i\leq m.
\]
The nonnegative regression matrix from
Lemma~\ref{lem:finite-wall-mtp2-closure} therefore gives
\(h_{\mathbf x^+,z}\leq0\).  Near either shrinking pin, the covariance with
each additional fixed exposed coordinate is \(O(a^\alpha)\); on a remaining
interval it is \(O(\delta^\alpha)\).  Division by the residual standard
deviation and summation over the fixed set \(T\) therefore give
\[
 \sup_{r\in J}\bigl(h_{\mathbf x,z}(r)\bigr)_+
 \leq
 \begin{cases}
   C_{T,A}a^H,&J\text{ approaches }0\text{ or }\rho,\\
   C_{T,A}\delta^H,
      &\operatorname{dist}(J,\{0,\rho\})\asymp\delta.
 \end{cases}
\]
At each coordinatewise comparison, order the two exposed values.
Lemma~\ref{lem:finite-wall-mtp2-closure} then gives an entrywise nonnegative
regression response, nested walls, and the same association comparison as
in the scalar argument.  The positive-level block baseline above replaces
the former use of a nonpositive wall.

The finite-rank correction from the additional exposed coordinates is
\(O(a^{2\alpha})=o(a^\alpha)\) near either pin, while the leading residual
variance is comparable to \(a^\alpha\).  On the remaining intervals,
rescaling by \(\delta\) gives a compact nondegenerate family.  Thus the
variance, metric, deterministic-threshold modulus, and positive-baseline
estimates remain uniform, with constants depending on \(T\).  Applying the
scalar wall comparison successively to the \(m\) coordinates proves
\eqref{eq:selected-germ-vector-harnack}.

For the density factorisation, let \(C_T\) be the covariance matrix of
\(\widehat B_T\), let
\(k_\rho=(K_H(t_i,\rho))_{i=1}^m\), and let \(S\) be the marginal precision
of \((\widehat B_T,\widehat B_\rho)\).  This marginal is independent of the
outer grid.  Block inversion gives
\begin{equation}
  S_{TZ}=-\frac{C_T^{-1}k_\rho}
   {|\rho|^\alpha-k_\rho^{\mathsf T}C_T^{-1}k_\rho},
  \label{eq:selected-germ-vector-precision}
\end{equation}
where \(k_\rho=O(|\rho|^\alpha)\) and the denominator is comparable to
\(|\rho|^\alpha\).  By Lemma~\ref{lem:finite-wall-mtp2-closure}, the Gaussian
marginal has Stieltjes precision \(S\), so
\(0\leq-S_{t_iZ}\leq C(T)\) uniformly.  The vector version of
\eqref{eq:selected-germ-density-cross-harnack}, applied one coordinate at a
time, and coordinatewise integration therefore give
\[
 \left|
  \log\frac{j_T(\mathbf y,z)j_T(\mathbf x,z')}
            {j_T(\mathbf x,z)j_T(\mathbf y,z')}
 \right|
 \leq C_{T,A,M}\delta^H
\]
on the rectangle in \eqref{eq:selected-germ-pin-factorization}.  This proves
that estimate and completes the finite-vector argument.
\end{proof}

\paragraph{Path-space completion.}
\label{subsec:rough-selected-germ-path-completion}

\begin{lemma}
\label{lem:rough-two-pin-kernel-ratio}
Put
\begin{equation}
  \Gamma_\rho(t,u)
  :=K_H(t,u)-\frac{K_H(t,\rho)K_H(u,\rho)}{|\rho|^\alpha}.
  \label{eq:rough-two-pin-bridge-kernel}
\end{equation}
Fix \(0<a<R<\infty\) and \(r_*>0\).  There are
\(\rho_0,C>0\) such that, if \(0<|\rho|\leq\rho_0\),
\(s,t\in[a,R]\), and \(u\notin\{0,\rho\}\), then
\begin{equation}
  |\Gamma_\rho(t,u)-\Gamma_\rho(s,u)|
  \leq C|t-s|^\alpha\Gamma_\rho(r_*,u).
  \label{eq:rough-two-pin-kernel-ratio}
\end{equation}
If instead \(r_*<0\), the reflected estimate holds for
\(s,t\in[-R,-a]\), with the same quantifiers.

If \(v=1/\rho\) and
\begin{equation}
  \overline\Gamma_v(t,u)
  :=K_H(t,u)-\frac{K_H(t,v)K_H(u,v)}{|v|^\alpha},
  \label{eq:rough-far-pin-bridge-kernel}
\end{equation}
then, for every \(\bar r_*>0\), there are \(v_0,C>0\) such that
\[
 |\overline\Gamma_v(t,u)-\overline\Gamma_v(s,u)|
 \leq C|t-s|^\alpha\overline\Gamma_v(\bar r_*,u)
\]
whenever \(|v|\geq v_0\), \(s,t\in[a,R]\), and
\(u\notin\{0,v\}\).  The reflected far-pin estimate holds for
\(\bar r_*<0\) and \(s,t\in[-R,-a]\).  Every reference kernel on the
right is strictly positive on its displayed domain.
\end{lemma}

\begin{proof}
We prove the positive-annulus assertion first.  Choose
\(u_0<\min\{a,r_*,1\}/4\), reduce \(\rho_0\) so that
\(\rho_0\leq u_0\), put \(\epsilon=|\rho|\) and
\(e=\rho/\epsilon\in\{-1,1\}\), and write, for
\(\tau\in\{r_*\}\cup[a,R]\) and \(|x|\leq2u_0\),
\begin{equation}
  K_H(\tau,x)=\frac12|x|^\alpha+h_\tau(x),\qquad
  h_\tau(x)=\frac12\{\tau^\alpha-|\tau-x|^\alpha\}.
  \label{eq:rough-two-pin-kernel-decomposition}
\end{equation}
The exact decomposition
\begin{equation}
 \begin{split}
  \Gamma_\rho(t,u)&=\Phi_\rho(u)+\Psi_{\rho,t}(u),\\
  \Phi_\rho(u)&=\frac14\{|u|^\alpha-\epsilon^\alpha+|u-\rho|^\alpha\},\\
  \Psi_{\rho,t}(u)&=h_t(u)-h_t(\rho)
                  \frac{K_H(u,\rho)}{\epsilon^\alpha}
 \end{split}
  \label{eq:rough-two-pin-phi-psi}
\end{equation}
Strict subadditivity gives \(\Phi_\rho(u)>0\) exactly when
\(u\notin\{0,\rho\}\).  With
\(k_e(q)=K_H(q,e)\) and
\(\phi_e(q)=\{|q|^\alpha-1+|q-e|^\alpha\}/4\), direct expansion at
the two zeros and the two tails gives
\begin{equation}
 \begin{split}
  |q-ek_e(q)|
   &\leq C(1+|q|)^{1-\alpha}\phi_e(q),\\
  |q^2-k_e(q)|+|q-e|k_e(q)
   &\leq C(1+|q|)^{2-\alpha}\phi_e(q).
 \end{split}
  \label{eq:rough-two-pin-normalized-bounds}
\end{equation}
Indeed, as \(q\to0\),
\[
 \phi_e(q)=\tfrac14|q|^\alpha+O(|q|),\quad
 k_e(q)=\tfrac12|q|^\alpha+O(|q|),
\]
so \(q-ek_e(q)=-e|q|^\alpha/2+O(|q|)\),
\(q^2-k_e(q)=-|q|^\alpha/2+O(|q|)\), and
\(|q-e|k_e(q)=|q|^\alpha/2+O(|q|)\).  If \(h=q-e\to0\), then
\[
 \phi_e(e+h)=\tfrac14|h|^\alpha+O(|h|),\quad
 k_e(e+h)=1-\tfrac12|h|^\alpha+O(|h|),
\]
so the first two numerators in
\eqref{eq:rough-two-pin-normalized-bounds} are respectively
\(e|h|^\alpha/2+O(|h|)\) and
\(|h|^\alpha/2+O(|h|)\), while
\(|h|k_e(e+h)=O(|h|)=o(|h|^\alpha)\).  Finally, as
\(|q|\to\infty\),
\[
 \phi_e(q)\sim\tfrac12|q|^\alpha,\qquad
 k_e(q)=\tfrac12+O(|q|^{\alpha-1});
\]
the two weighted denominators in
\eqref{eq:rough-two-pin-normalized-bounds} are therefore of orders
\(|q|\) and \(|q|^2\), matching their numerators.  Compactness on the
remaining set proves both bounds, uniformly for \(e=\pm1\).

Taylor's formula gives
\(h_\tau(x)=d_\tau x+x^2b_\tau(x)\), with the coefficients, their
first \(\tau\)-derivatives on \([a,R]\), and the required first \(x\)- and
mixed derivatives of \(b_\tau\) bounded uniformly on the displayed compact
sets.
Substituting \(u=\epsilon q\) in
\eqref{eq:rough-two-pin-phi-psi} and using
\eqref{eq:rough-two-pin-normalized-bounds} yields, for
\(L=|u|+\epsilon\),
\begin{equation}
 \begin{split}
  |\Psi_{\rho,r_*}(u)|
   &\leq CL^{1-\alpha}\Phi_\rho(u),\\
  |\Psi_{\rho,t}(u)-\Psi_{\rho,s}(u)|
   &\leq C|t-s|L^{1-\alpha}\Phi_\rho(u).
 \end{split}
  \label{eq:rough-two-pin-psi-bounds}
\end{equation}
For clarity, the linear Taylor terms give
\(CL^{1-\alpha}\Phi_\rho\), while the quadratic terms first give
\(CL^{2-\alpha}\Phi_\rho\); the latter are absorbed because
\(L\leq u_0+\rho_0\leq1\).  Since \(1-\alpha>0\), shrinking
\(u_0,\rho_0\) makes the first line at most \(\Phi_\rho/2\).  Hence
\(\Gamma_\rho(r_*,u)\geq\Phi_\rho(u)/2\) when \(|u|\leq u_0\).
The second line and \(|t-s|\leq|t-s|^\alpha\) prove
\eqref{eq:rough-two-pin-kernel-ratio} there whenever \(|t-s|\leq1\).

For \(|u|\geq u_0\), strict positivity, continuity, and
\(K_H(r_*,u)\to r_*^\alpha/2\) at both tails give
\[
 c_*:=\inf_{|u|\geq u_0}K_H(r_*,u)>0.
\]
The power inequality and subadditivity give
\[
 0\leq K_H(u,\rho)\leq\epsilon^\alpha,\qquad
 0\leq K_H(r_*,\rho)\leq\epsilon^\alpha.
\]
Since \(K_H(r_*,\rho)\to0\), a further reduction of \(\rho_0\) gives
\[
 \Gamma_\rho(r_*,u)
 \geq K_H(r_*,u)-K_H(r_*,\rho)\geq c_*/2
 \qquad(|u|\geq u_0).
\]
Moreover,
\[
 |K_H(t,u)-K_H(s,u)|\leq C|t-s|^\alpha
\]
uniformly in \(u\).  The common \(\epsilon^\alpha/2\) term cancels from
\(K_H(t,\rho)-K_H(s,\rho)\).  For small \(\rho\), the derivative on
\([a,R]\) of \(x^\alpha-|x-\rho|^\alpha\) is bounded by
\(C\epsilon\), and hence
\[
 |K_H(t,\rho)-K_H(s,\rho)|\leq C\epsilon|t-s|.
\]
Multiplication by \(K_H(u,\rho)/\epsilon^\alpha\leq1\), followed by the
lower bound above, proves
\eqref{eq:rough-two-pin-kernel-ratio} for \(|u|\geq u_0\) and
\(|t-s|\leq1\).  A finite subdivision of the fixed interval \([a,R]\)
handles \(|t-s|>1\), with a changed constant.  This proves the complete
positive-annulus assertion.  The identity
\[
 \Gamma_\rho(t,u)=\Gamma_{-\rho}(-t,-u)
\]
proves its reflected version, including both signs of \(\rho\).

Finally,
\begin{equation}
  \overline\Gamma_v(t,u)
  =|tu|^\alpha\Gamma_{1/v}(1/t,1/u).
  \label{eq:rough-inverted-bridge-kernel}
\end{equation}
Fix \(\bar r_*>0\), put
\(q_*=1/\bar r_*\), and, for positive \(s,t\), set
\(x=1/t\), \(y=1/s\), and \(w=1/u\).  The points
\(x,y,q_*\) lie in one fixed compact positive interval.  The assertion just
proved, applied on that interval with reference \(q_*\), gives
\[
 |\Gamma_{1/v}(x,w)-\Gamma_{1/v}(y,w)|
 \leq C|x-y|^\alpha\Gamma_{1/v}(q_*,w),
 \qquad
 \Gamma_{1/v}(y,w)\leq C\Gamma_{1/v}(q_*,w).
\]
The second inequality follows by comparing \(y\) with \(q_*\); the preceding
finite-subdivision argument covers their whole compact interval.  Using
\eqref{eq:rough-inverted-bridge-kernel}, smoothness of inversion on
\([a,R]\), and
\(\bigl||t|^\alpha-|s|^\alpha\bigr|\leq|t-s|^\alpha\), we obtain
\[
\begin{split}
 |\overline\Gamma_v(t,u)-\overline\Gamma_v(s,u)|
 &\leq |u|^\alpha\bigl(
  |t|^\alpha|\Gamma_{1/v}(x,w)-\Gamma_{1/v}(y,w)|\\
 &\hspace{5em}
  +\bigl||t|^\alpha-|s|^\alpha\bigr|\Gamma_{1/v}(y,w)\bigr)\\
 &\leq C|t-s|^\alpha|u|^\alpha\Gamma_{1/v}(q_*,w)\\
 &=C|t-s|^\alpha\overline\Gamma_v(\bar r_*,u).
\end{split}
\]
This is the stated far-pin estimate.  Simultaneous reflection of all
variables proves its negative-annulus counterpart.
\end{proof}

We now formulate the path-space comparison only for the exact family needed
in the theorem.  For the canonical constraint grids \(E_m\) from
\eqref{eq:canonical-walls}, put \(D_m=E_m\cap I\).  Fix \(m_I\) so large
that \(s_0\in D_m\) for \(m\geq m_I\), let
\(p_{m,s}=\mu_{E_m}(S_{D_m}=s)\), and, for
\(0<\eta<(v-u)/2\), put \(I_\eta=[u+\eta,v-\eta]\).  Define two finite
measures on \(\mathcal X\):
\begin{equation}
 \begin{split}
  \mathsf P_{b,m,\eta}^{\mathrm{sel}}
   &:=\sum_{s\in D_m\cap I_\eta}p_{m,s}
       \mathcal L(G^{b,s}\mid A_{b,s}),\\
  \mathsf P_{b,m,\eta}^{\mathrm{out}}
   &:=\sum_{s\in D_m\cap I_\eta}p_{m,s}
       \mathcal L(G^{b,s}\mid G^{b,s}(j)\geq0:\ j\in J_{b,s}).
 \end{split}
 \label{eq:rough-selected-outer-submixtures}
\end{equation}
They have the same total mass.
For finite measures \(\nu_1,\nu_2\) of equal mass on \(\mathcal X\), let
\(d_{\mathrm{BL}}(\nu_1,\nu_2)\) denote the bounded-Lipschitz distance
obtained from test functions bounded by one and \(1\)-Lipschitz for
\(d_{\mathrm{loc}}\).

\begin{proposition}
\label{prop:rough-selected-germ-invisibility}
For every fixed \(\eta\in(0,(v-u)/2)\),
\begin{equation}
  \sup_{\substack{m\geq m_I\\2^{-m}\leq\eta/2}} d_{\mathrm{BL}}
  \bigl(\mathsf P_{b,m,\eta}^{\mathrm{sel}},
        \mathsf P_{b,m,\eta}^{\mathrm{out}}\bigr)
  \longrightarrow0,
  \qquad b\to\infty,
  \label{eq:rough-selected-germ-invisibility}
\end{equation}
\end{proposition}

\begin{proof}
For \(s\in I_\eta\), time inversion sends the old root to
\(\rho=-1/(bs)\), and
\begin{equation}
  \frac{1}{v}b^{-1}\leq|\rho|\leq\frac{1}{u}b^{-1}.
  \label{eq:rough-general-window-root-scale}
\end{equation}
If the original grid mesh is at most \(\eta/2\), the inverted left- and
right-gap radii, divided by \(|\rho|\), are bounded by \(2v/\eta\).
More explicitly, every inverted zero-wall point coming from \(D_m-s\) has
absolute value at least \(b^{-1}/(v-u)\), whereas every inverted competitor
coming from \(E_m\setminus I\) has absolute value at most \(b^{-1}/\eta\).
The two sets are disjoint on each side of the corresponding asymmetric gap.
Together with \eqref{eq:rough-general-window-root-scale}, these bounds allow
fixed constants \(0<c_0<C_0<\infty\) for which the outer set lies outside
\((-c_0b^{-1},c_0b^{-1})\), while the selected set and \(\rho\) lie inside
\((-C_0b^{-1},C_0b^{-1})\).  Thus
Lemma~\ref{lem:rough-selected-germ-factorization} applies with
\(\delta=b^{-1}\), uniformly in the displayed grid family.

We first pass from compact likelihood factorisation to arbitrary fixed
finite-dimensional laws.  Lemma~\ref{lem:rough-selected-germ-factorization}
is uniform on signed compact cubes, so it applies to every fixed observation
set whether or not its times are wall sites.

For one component, adjoin \(T\) to the outer coordinate set and omit it from
the conditioning event, as in the definition of \(E_\delta\).  If an
observation time already is an outer zero-wall site, let
\(\chi_{b,m,s,T}(\mathbf x)\) impose its nonnegative coordinate; otherwise
that factor is one.  It is common to the selected and outer laws.  For
\(X=\widehat B_T\) and \(Z=\widehat B_\rho\), the joint selected density is
proportional to
\begin{equation}
  \chi_{b,m,s,T}(\mathbf x)j_T(\mathbf x,z)
  g_{\delta,z}^T(\mathbf x)\mathds1_{\{z\leq0\}},
  \label{eq:rough-selected-joint-density}
\end{equation}
whereas the outer density is proportional to
\(\chi_{b,m,s,T}(\mathbf x)j_T(\mathbf x,z)\).
On the signed compact rectangle, the lemma's estimates write
\begin{equation}
  j_T(\mathbf x,z)=a_\delta(\mathbf x)b_\delta(z)e^{o(1)},
  \qquad
  g_{\delta,z}^T(\mathbf x)
   =g_{\delta,z}^T(\mathbf 0)e^{o(1)},
  \label{eq:rough-selected-compact-factorization}
\end{equation}
uniformly in the component; the second relation is used only for \(z\leq0\).
After integration in \(z\), the compact outer and selected \(X\)-densities,
normalised separately, are both the normalisation of
\(\chi_{b,m,s,T}(\mathbf x)a_\delta(\mathbf x)\), up to \(e^{o(1)}\).
Their two compact conditional \(X\)-laws therefore have total-variation
distance \(o(1)\), uniformly in the component.

We remove the cutoffs in a fixed order.  The pure zero-wall mean bound
\eqref{eq:pin-mean-bound} and
Lemma~\ref{lem:harge-convex-domination} apply uniformly to arbitrary finite
wall sets: approximate distinct wall coordinates by distinct dyadic ones
and pass through the nondegenerate finite Gaussian densities.  Under every
outer component they make \(X\) tight and make \(Z/\delta^H\) tight.  Under
the selected law, the exact identity
\begin{equation}
  \frac{|\widehat B_\rho|}{|\rho|^H}
  =\frac{|B_s|}{s^H}
  \label{eq:rough-selected-depth-identity}
\end{equation}
and the window-maximum inequality
\begin{equation}
  \frac{|B_s|}{s^H}\leq\frac{|B_{s_0}|}{u^H}
  \label{eq:rough-general-window-depth-bound}
\end{equation}
give the same natural-scale depth tightness after summing with the exact
weights \(p_{m,s}\).  Here \(B_s\geq B_{s_0}\) and both values are
nonpositive, while the pure-wall mean and centred bounds make
\(B_{s_0}\) tight uniformly in \(m\).

Conditional on \(|Z|\leq M\delta^H\), raise the nonpositive selected
thresholds to zero, then raise the exposed pin from \(z\leq0\) to zero, and
finally mix the pin over its nonnegative values.  Finite-wall MTP\(_2\)
monotonicity raises every remaining coordinate at each step.  Hence each
selected coordinate mean is bounded above by its ordinary pure zero-wall
mean and below by its untruncated bridge mean, whose absolute value is at
most \(M|t|^H\).  The moment-generating-function estimate
\eqref{eq:harge-mgf-bound}, applied coordinatewise to each convexly conditioned
Gaussian bridge after centring at its conditional barycentre, then makes \(X\)
tight under the selected law on this depth event.  First choose \(M\), then choose
the signed observation cutoff \(A\), and only then send \(b\to\infty\) in
\eqref{eq:rough-selected-compact-factorization}.  Letting \(A,M\to\infty\)
afterwards proves, uniformly in the exact mixture,
\[
 \left\|
  (\widehat{\operatorname{ev}}_T)_*
     \mathsf P_{b,m,\eta}^{\mathrm{sel}}
  -(\widehat{\operatorname{ev}}_T)_*
     \mathsf P_{b,m,\eta}^{\mathrm{out}}
 \right\|_{\mathrm{TV}}\longrightarrow0,
 \qquad
 \widehat{\operatorname{ev}}_T(f)
   :=\bigl(|t|^\alpha f(1/t)\bigr)_{t\in T}.
\]
The compact conditional errors are uniform in the component, while the
selected and outer tail masses vanish after weighting by \(p_{m,s}\).
Since \(\sum_sp_{m,s}\leq1\), the last display is a finite-measure, not a
claim of componentwise uniform tail control.

We next prove the conditional path modulus.  Expose
\(\widehat B_\rho=z\), \(|z|\leq M|\rho|^H\), remove the deterministic
zero coordinate and the exposed pin, and collect the distinct remaining
constraints in \(F\), with lower thresholds \(b_u(z)\).  Let \(Q\) be the
precision of the conditional bridge vector on \(F\).  Integration of
\(-\partial_up= (Q(y-\mu_F))_up\) over the lower rectangle gives
\[
 \left[Q\{\EE(Y_F\mid Y_F\geq b)-\mu_F\}\right]_u
 =\frac{1}{\PP(Y_F\geq b)}
   \int_{\{y_{F\setminus\{u\}}\geq b\}}
       p(b_u,y_{F\setminus\{u\}})\,\mathrm dy_{F\setminus\{u\}}
 \geq0.
\]
Lemma~\ref{lem:finite-wall-mtp2-closure} makes the conditional bridge vector,
after adjoining a test coordinate, MTP\(_2\) and associated.  Bounded
coordinate truncation followed by Gaussian integrability therefore makes
every conditional bridge covariance \(\Gamma_\rho(t,u)\) nonnegative.
Gaussian regression gives
\begin{equation}
 \begin{split}
  m_{\rho,z,F}(t)
  &=\mu_{\rho,z}(t)
    +\sum_{u\in F}\lambda_u\Gamma_\rho(t,u),\\
  \mu_{\rho,z}(t)&=\frac{K_H(t,\rho)}{|\rho|^\alpha}z,\qquad
  \lambda_u\geq0.
 \end{split}
 \label{eq:rough-selected-boundary-flux}
\end{equation}
This calculation permits unequal negative thresholds and uses no
zero-orthant substitution.

Raise the selected thresholds to zero, raise the pin from \(z\) to zero,
and then mix it over nonnegative values.  The same MTP\(_2\) monotonicity as
above and the pure-wall bound, extended to arbitrary finite coordinate sets
by dyadic approximation, give
\begin{equation}
  \mu_{\rho,z}(t)\leq m_{\rho,z,F}(t)\leq C_H|t|^H,
  \qquad |\mu_{\rho,z}(t)|\leq M|t|^H.
  \label{eq:rough-selected-mean-pin-bound}
\end{equation}
The last inequality is Cauchy--Schwarz:
\(|K_H(t,\rho)|\leq|t|^H|\rho|^H\).
In particular,
\(|m_{\rho,z,F}(t)|\leq C_{H,M}|t|^H\).

Fix a positive annulus \([a,R]\) and a fixed positive reference \(r_*\).
At that reference,
\[
 \sum_{u\in F}\lambda_u\Gamma_\rho(r_*,u)
 =m_{\rho,z,F}(r_*)-\mu_{\rho,z}(r_*)
 \leq C_{H,M,a,R}.
\]
Lemma~\ref{lem:rough-two-pin-kernel-ratio} therefore controls the complete
boundary-flux sum.  The pin mean must be estimated separately.  Cancellation
of the common \(|\rho|^\alpha/2\) term gives, for \(|t-s|\leq1\),
\[
 |\mu_{\rho,z}(t)-\mu_{\rho,z}(s)|
 \leq CM|\rho|^{1-H}|t-s|
 \leq CM|t-s|^\alpha.
\]
Finite subdivision handles larger separations, and simultaneous reflection
gives the negative-annulus estimate.  Hence
\begin{equation}
  |m_{\rho,z,F}(t)-m_{\rho,z,F}(s)|
  \leq C_{H,M,a,R}|t-s|^\alpha
  \label{eq:rough-selected-inverted-mean-modulus}
\end{equation}
on either annulus.

The passage through zero is now explicit.  Given a compact interval and a
modulus tolerance, first choose \(a>0\) so that
\(2C_{H,M}(2a)^H\) is below half the tolerance.  If one of two sufficiently
close points lies in \([-a,a]\), both lie in \([-2a,2a]\) and the pointwise
bound in \eqref{eq:rough-selected-mean-pin-bound} applies.  Otherwise the
points have the same sign and lie in one of the two fixed annuli, where
\eqref{eq:rough-selected-inverted-mean-modulus} applies.  This proves a
uniform mean modulus on every compact interval.

After the pin has been exposed, the remaining event is a convex translated
lower rectangle for the centred Gaussian bridge.  Applying
Lemma~\ref{lem:harge-convex-domination} at its barycentre gives
\begin{equation}
 \EE\exp\left\{\lambda\bigl[
   (\widehat B_t-m_{\rho,z,F}(t))
   -(\widehat B_s-m_{\rho,z,F}(s))\bigr]\right\}
 \leq\exp\{\lambda^2|t-s|^\alpha/2\},
 \label{eq:rough-selected-centered-modulus}
\end{equation}
because
\[
 \operatorname{Var}(\widehat B_t-\widehat B_s\mid\widehat B_\rho)
 =|t-s|^\alpha
  -\frac{\{K_H(t,\rho)-K_H(s,\rho)\}^2}{|\rho|^\alpha}
 \leq|t-s|^\alpha.
\]
The dyadic chaining proof of
Lemma~\ref{lem:uniform-centered-modulus}, combined with the mean modulus,
proves compact-depth tightness in the inverted coordinates, uniformly in
the finite wall and its cardinality.

It remains to establish tightness in the original coordinates without using
pathwise continuity of inversion at zero.  Put \(v=1/\rho\), expose
\(G(v)=x\), where \(x\leq0\) and \(|x|\leq M|v|^H\), and collect the
remaining zero and old-root lower walls in a finite set \(F\), again deleting
the deterministic zero and the exposed pin.  The same face integration and
regression give
\[
 m_{v,x,F}(t)
 =\mu_{v,x}(t)+\sum_{u\in F}\bar\lambda_u
    \overline\Gamma_v(t,u),\qquad
 \mu_{v,x}(t)=\frac{K_H(t,v)}{|v|^\alpha}x,\qquad
 \bar\lambda_u\geq0.
\]
Raising the negative thresholds and then the far pin to zero, followed by
the nonnegative-pin mixture, bounds the mean above by the ordinary pure
zero-wall mean and below by its bridge mean.  Cauchy--Schwarz therefore gives
\begin{equation}
  |m_{v,x,F}(t)|\leq C_{H,M}|t|^H.
  \label{eq:rough-selected-original-pin-bound}
\end{equation}

On a fixed annulus, the far-pin assertion of
Lemma~\ref{lem:rough-two-pin-kernel-ratio} controls the boundary-flux sum,
because at a fixed same-side reference
\[
 \sum_{u\in F}\bar\lambda_u\overline\Gamma_v(r_*,u)
 =m_{v,x,F}(r_*)-\mu_{v,x}(r_*)\leq C_{H,M,a,R}.
\]
The far-pin bridge mean is separate and satisfies, for \(|v|\geq1\),
\[
 |\mu_{v,x}(t)-\mu_{v,x}(s)|
 \leq\frac{|x|}{|v|^\alpha}
      |K_H(t,v)-K_H(s,v)|
 \leq C_{H,M,a,R}|t-s|^\alpha,
\]
using the global power inequality and
\(|x|/|v|^\alpha\leq M|v|^{-H}\).
Reflection and finite subdivision cover both annuli and all separations.
Together with \eqref{eq:rough-selected-original-pin-bound}, the same
two-region argument gives a uniform mean modulus on every compact
original-coordinate interval.

Convex Gaussian domination applies to this translated far-pin bridge as
well, and its increment variance is at most \(|t-s|^\alpha\).  The same
chaining argument proves compact-depth tightness of the selected original
paths.  Ordinary outer components are pure finite zero-wall laws; the
pure-wall mean modulus and the same centred chaining bound make them tight
uniformly, including for nondyadic finite wall locations by approximation.
Finally \eqref{eq:rough-selected-depth-identity} and
\eqref{eq:rough-general-window-depth-bound} remove the selected depth cutoff
only after the compact-depth modulus has been established.  Thus both
finite-measure families are tight in \(\mathcal X\), uniformly in the
displayed parameters.

For an arbitrary fixed finite set of nonzero original times, apply the
signed finite-dimensional comparison above at their reciprocal times and
use the exact inversion identity coordinatewise.  The value at zero is
deterministically zero.  Hence the original-coordinate finite-dimensional
pushforwards of the two finite measures converge in total variation.

To finish the uniform path-space statement, suppose it failed along some
\(b_n\to\infty\) and admissible \(m_n\).  The two tight subprobability
families have equal masses in \([0,1]\); after passage to a subsequence, both
measures converge weakly.  Their limiting finite-dimensional distributions
agree on every finite subset of a fixed countable dense set by the preceding
comparison.  Continuous evaluations generate the Borel sigma-field of
\(\mathcal X\), so the two finite limit measures are equal.  Since the
displayed bounded-Lipschitz metric metrises weak convergence for uniformly
bounded finite measures, this contradicts the assumed failure and proves
\eqref{eq:rough-selected-germ-invisibility}.
\end{proof}
\subsubsection{Continuum window selection and endpoint exclusion}
\label{subsec:rough-continuum-window-selection}

\paragraph{Strict negativity.}

The quantitative finite-wall estimate needed to exclude contact with the wall
in the limit is available from the companion paper.

\begin{proposition}[{\cite[Proposition~4.14]{moench2026horizonProblem}}]
\label{prop:rough-annular-anticontact-import}
Put \(I_0=[1/2,1]\) and \(p_0=2\).  There are constants \(c,C>0\) and
\(x_0\in(0,1)\), depending only on \(H\), such that, for every finite dyadic
\(F\subset\mathbb D\) containing \(p_0\),
\begin{equation}
 \PP\left(
  \left|\min_{s\in I_0}(-B_s)\right|\leq x\,\middle|\,A_F
 \right)
 \leq C\exp\{-c\sqrt{\log(1/x)}\},
 \qquad 0<x\leq x_0.
 \label{eq:rough-annular-anticontact-import}
\end{equation}
The reflected assertion holds with \((I_0,p_0)\) replaced by
\((-I_0,-p_0)\).
\end{proposition}

\begin{proof}[Proof of \eqref{eq:limit-support}]
For \(f\in\mathcal X\), set
\[
 W(f):=\min_{s\in I_0}(-f(s)).
\]
This functional is continuous in the local-uniform topology.  The canonical
hard-wall laws satisfy
\(\mathcal L(B\mid A_{E_m})\Rightarrow Q\), and \(p_0\in E_m\) for all
sufficiently large \(m\).  Hence, for \(0<x\leq x_0\), the open-set
Portmanteau inequality and Proposition~
\ref{prop:rough-annular-anticontact-import} give
\begin{align*}
 Q\{f:|W(f)|<x\}
 &\leq\liminf_{m\to\infty}
   \PP\left(
    \left|\min_{s\in I_0}(-B_s)\right|<x\,\middle|\,A_{E_m}
   \right)\\
 &\leq C\exp\{-c\sqrt{\log(1/x)}\}.
\end{align*}
By \eqref{eq:limit-nonpositive-support}, \(W\geq0\) almost surely under
\(Q\).  Letting \(x\downarrow0\) therefore shows that
\begin{equation}
 Q\left\{f:\max_{s\in I_0}f(s)<0\right\}=1.
 \label{eq:rough-base-annulus-strict-negativity}
\end{equation}
For each \(k\in\mathbb Z\), self-similarity
\eqref{eq:limit-selfsimilarity}, applied with \(c=2^{-k-1}\), transports
\eqref{eq:rough-base-annulus-strict-negativity} to
\[
 Q\left\{f:\max_{s\in[2^k,2^{k+1}]}f(s)<0\right\}=1.
\]
Intersecting these countably many events gives strict negativity on
\((0,\infty)\).  Reflection invariance
\eqref{eq:rough-tangent-reflection-invariance} gives the same conclusion on
\((-\infty,0)\), while \eqref{eq:limit-nonpositive-support} includes the
deterministic pin \(f(0)=0\).  This proves \eqref{eq:limit-support} and
completes Theorem~\ref{thm:main-conditional-limit}.
\end{proof}

\paragraph{Uniqueness.}

We next prove uniqueness without assuming local absolute continuity of \(Q\).
Choose \(p\in\mathbb D\cap(v,\infty)\), and, for a finite
\(F\subset\mathbb D\) containing \(p\), retain \(A_F\) from
\eqref{eq:reflected-wall-event} and put
\begin{equation}
  r_p(t):=\frac{K_H(t,p)}{p^\alpha}.
  \label{eq:rough-window-anchor-profile}
\end{equation}
Strict subadditivity gives \(r_p(t)>0\) for \(t\neq0\), and
\begin{equation}
  r_p'(t)=\frac{\alpha}{2p^\alpha}
    \{t^{\alpha-1}+(p-t)^{\alpha-1}\}>0,\qquad 0<t<p.
  \label{eq:rough-window-anchor-monotonicity}
\end{equation}

\begin{lemma}
\label{lem:rough-window-separated-anticoncentration}
Let \(I_1,I_2\subset I\) be nonempty compact intervals with
\(\sup I_1<\inf I_2\).  There are constants
\(C_{I_1,I_2,M}<\infty\) and a nonincreasing \(\varepsilon_M\downarrow0\)
as \(M\to\infty\), both independent of the finite
\(F\subset\mathbb D\) containing \(p\), such that, for every \(M\geq0\)
and \(\epsilon>0\),
\begin{equation}
 \PP\left(
  \left|\min_{t\in I_1}(-B_t)-\min_{t\in I_2}(-B_t)\right|
  \leq\epsilon\,\middle|\,A_F
 \right)
 \leq C_{I_1,I_2,M}\epsilon+\varepsilon_M.
 \label{eq:rough-window-separated-anticoncentration}
\end{equation}
\end{lemma}

\begin{proof}
Let \(Z=-B_p\) and \(Y_t=-B_t-r_p(t)Z\).  The whole process \(Y\) is
independent of \(Z\).  Since \(r_p(t)>0\) for every \(t\in\mathbb D\),
while \(r_p(p)=1\) and \(Y_p=0\), conditional on \(Y\) one has
\begin{equation}
  A_F=\{Z\geq\ell_F(Y)\},
  \qquad
  \ell_F(Y)=\max_{t\in F}\frac{-Y_t}{r_p(t)}.
  \label{eq:rough-window-scalar-truncation}
\end{equation}
In particular, \(\ell_F(Y)\geq0\).  Let \(\varphi_{p^\alpha}\) be the
centred normal density of variance \(p^\alpha\), and put
\(\overline\Phi_{p^\alpha}(a)=\int_a^\infty
\varphi_{p^\alpha}(z)\,\mathrm dz\).  Thus the conditional density of
\(Z\) given \((Y,A_F)\) is exactly
\begin{equation}
 z\longmapsto
 \frac{\varphi_{p^\alpha}(z)
       \mathds1_{\{z\geq\ell_F(Y)\}}}
      {\overline\Phi_{p^\alpha}(\ell_F(Y))}.
 \label{eq:rough-window-truncated-anchor-density}
\end{equation}

For fixed \(Y\), set
\begin{equation}
  d_Y(z)=\min_{t\in I_1}\{Y_t+r_p(t)z\}
        -\min_{t\in I_2}\{Y_t+r_p(t)z\}.
\end{equation}
If \(z_2>z_1\), then
\begin{align*}
 \min_{t\in I_1}\{Y_t+r_p(t)z_2\}
 &\leq \min_{t\in I_1}\{Y_t+r_p(t)z_1\}
      +(z_2-z_1)\sup_{t\in I_1}r_p(t),\\
 \min_{t\in I_2}\{Y_t+r_p(t)z_2\}
 &\geq \min_{t\in I_2}\{Y_t+r_p(t)z_1\}
      +(z_2-z_1)\inf_{t\in I_2}r_p(t).
\end{align*}
Consequently,
\begin{equation}
  d_Y(z_2)-d_Y(z_1)
  \leq-c_{I_1,I_2}(z_2-z_1),\qquad
  c_{I_1,I_2}=\inf_{I_2}r_p-\sup_{I_1}r_p>0.
  \label{eq:rough-window-minimum-slope}
\end{equation}
Each minimum is Lipschitz in \(z\), so \(d_Y\) is continuous.  Hence the
inverse image of \([-\epsilon,\epsilon]\) has length at most
\(2\epsilon/c_{I_1,I_2}\).  If it meets
\([\ell_F(Y),M]\), then \(\ell_F(Y)\leq M\), and the normalised truncated
Gaussian density in \eqref{eq:rough-window-truncated-anchor-density} is
bounded there by
\begin{equation}
  \frac{(2\pi p^\alpha)^{-1/2}}
       {\overline\Phi_{p^\alpha}(M)}.
\end{equation}
Integration first in \(Z\) therefore bounds the event in
\eqref{eq:rough-window-separated-anticoncentration}, intersected with
\(\{Z\leq M\}\), by
\begin{equation}
 \frac{2\epsilon}{c_{I_1,I_2}}
 \frac{(2\pi p^\alpha)^{-1/2}}
      {\overline\Phi_{p^\alpha}(M)}.
 \label{eq:rough-window-truncated-small-interval}
\end{equation}
Moreover,
\begin{equation}
 \varepsilon_M
 :=\sup_{\substack{F\subset\mathbb D\text{ finite}\\p\in F}}
    \PP(Z>M\mid A_F)
 =\sup_{\substack{F\subset\mathbb D\text{ finite}\\p\in F}}
    \nu_F^{\{p\}}((M,\infty))
 \longrightarrow0.
  \label{eq:rough-window-anchor-tail}
\end{equation}
Indeed, this is Lemma~\ref{lem:uniform-anchor-tail}.  Combining
\eqref{eq:rough-window-truncated-small-interval} with
\eqref{eq:rough-window-anchor-tail} proves the claim.  No lower bound on
\(\PP(A_F)\) is used.
\end{proof}

\begin{proposition}
\label{prop:rough-unique-window-maximizer}
If \(X\sim Q\), then \(X\) has a unique maximiser on \(I\) almost surely.
Put \(D_m=E_m\cap I\).  For all sufficiently large \(m\), this set is
nonempty; let \(S_m(f)\) denote the leftmost maximiser of \(f\) on \(D_m\).
If \(x_m,x\in\mathcal X\), \(x_m\to x\) locally uniformly, and \(x\) has a
unique maximiser \(S_I(x)\) on \(I\), then
\(S_m(x_m)\to S_I(x)\).
\end{proposition}

\begin{proof}
Start the canonical sequence after the fixed dyadic anchor \(p\) belongs to
\(E_m\).  The path-space hard-wall limit theorem gives
\(\mathcal L(-B\mid A_{E_m})\Rightarrow\mathcal L(-X)\).  For fixed
separated compact intervals \(I_1,I_2\subset I\), put
\[
 \Delta_{I_1,I_2}(f)
 :=\left|\min_{t\in I_1}f(t)-\min_{t\in I_2}f(t)\right|.
\]
This functional is continuous in the local-uniform topology.  If
\(0\leq\epsilon<\eta\), then the closed event at radius \(\epsilon\) is
contained in the open event at radius \(\eta\).  The open-set
Portmanteau inequality and
Lemma~\ref{lem:rough-window-separated-anticoncentration} therefore give,
for every fixed \(M\geq0\),
\begin{align*}
 \PP_Q\{\Delta_{I_1,I_2}(-X)\leq\epsilon\}
 &\leq \PP_Q\{\Delta_{I_1,I_2}(-X)<\eta\}\\
 &\leq\liminf_{m\to\infty}
   \PP\{\Delta_{I_1,I_2}(-B)<\eta\mid A_{E_m}\}\\
 &\leq C_{I_1,I_2,M}\eta+\varepsilon_M.
\end{align*}
Letting \(\eta\downarrow\epsilon\) yields
\begin{equation}
  \PP_Q\{\Delta_{I_1,I_2}(-X)\leq\epsilon\}
  \leq C_{I_1,I_2,M}\epsilon+\varepsilon_M.
  \label{eq:rough-window-limit-anticoncentration}
\end{equation}
Taking \(\epsilon=0\) and then \(M\to\infty\) shows that the two minima are
almost surely unequal for each fixed separated pair.  If a continuous path
had two distinct minimisers \(s<t\) on \(I\), choose rationals
\(s<q_1<q_2<t\).  The minima on \([u,q_1]\) and \([q_2,v]\) would then agree.
There are only countably many such rational pairs, so a countable union
proves uniqueness.

Finally fix deterministic \(x_m,x\in\mathcal X\) as in the statement and
write \(s=S_I(x)\), \(s_m=S_m(x_m)\).  For all sufficiently large \(m\), the
grid \(D_m\) is nonempty and every point of \(I\) lies within \(2^{-m}\) of
some point of \(D_m\).  The leftmost rule makes \(S_m\) Borel, because each
of its finitely many fibres is an intersection of evaluation inequalities.
Choose \(r_m\in D_m\) with
\(|r_m-s|\leq2^{-m}\), and put
\(\delta_m=\|x_m-x\|_{C(I)}\to0\).  Maximality of \(s_m\) gives
\[
 x(s_m)\geq x_m(s_m)-\delta_m
          \geq x_m(r_m)-\delta_m
          \geq x(r_m)-2\delta_m
          \longrightarrow x(s).
\]
Since \(x(s_m)\leq x(s)\), one has \(x(s_m)\to x(s)\).  Every subsequence of
\((s_m)\) has, by compactness of \(I\), a further subsequence converging to
some \(t\in I\); continuity then gives \(x(t)=x(s)\), and uniqueness forces
\(t=s\).  Hence the full sequence satisfies \(s_m\to s\), with no asserted
rate.
\end{proof}

\paragraph{Ordinary germs and endpoint exclusion.}
\label{subsec:rough-ordinary-germs}
\label{subsec:rough-fixed-point-germs}

To prove that the \(Q\)-maximiser on \(I=[u,v]\) is not attained at either
endpoint \(u\) or \(v\), we need ordinary local oscillation at deterministic
positive times under \(Q\).  The Gaussian input is the following exterior
prediction lemma.

\begin{lemma}
\label{lem:rough-vanishing-exterior-prediction}
Let \(B\) be a two-sided fractional Brownian motion and let \(\Pi_R\) be
orthogonal projection in its Gaussian \(L^2\)-space onto
\begin{equation}
  \mathscr E_R
  :=\overline{\operatorname{span}}^{L^2}\{B_u:\ |u|\geq R\}.
\end{equation}
For fixed \(s,t\neq0\),
\begin{equation}
  \operatorname{Cov}(\Pi_RB_s,\Pi_RB_t)\longrightarrow0,
  \qquad R\to\infty.
  \label{eq:rough-vanishing-exterior-prediction}
\end{equation}
\end{lemma}

\begin{proof}
Put \(\beta=1+\alpha\in(1,2)\), let \(Z\) be the symmetric
\(\beta\)-stable process of Lemma~\ref{lem:stable-green-killed-zero}, and
write \(D=\mathbb R\setminus\{0\}\).  By
\eqref{eq:fbm-stable-green-identification},
\begin{equation}
  K_H(x,y)=\frac{1}{2c_\beta}g_D(x,y).
  \label{eq:rough-green-covariance}
\end{equation}
For \(D_R=(-R,R)\setminus\{0\}\), let \(g_{D_R}\) be the Green density
killed on leaving \(D_R\), extended by zero when either argument is outside
\(D_R\).  For \(x\in D_R\), set
\begin{equation*}
  \sigma_R:=\inf\{t>0:|Z_t|\geq R\},\qquad
  \tau_R:=T_0\wedge\sigma_R,
\end{equation*}
and define the subprobability exit law on \(\{|z|\geq R\}\) by
\begin{equation*}
  H_R(x,A):=\PP_x(Z_{\sigma_R}\in A,\ \sigma_R<T_0).
\end{equation*}
Thus the mass corresponding to a hit of zero is sent to the cemetery.  The
strong Markov property at \(\tau_R\) gives
\begin{equation}
  g_D(x,y)=g_{D_R}(x,y)+\int g_D(z,y)H_R(x,\mathrm dz),
  \qquad x\in D_R, y\in D.
  \label{eq:rough-green-balayage}
\end{equation}

We next justify the Gaussian exit integral.  The L\'evy measure of \(Z\) is
of the form \(\nu_\beta(\mathrm dh)=a_\beta|h|^{-1-\beta}\,\mathrm dh\).
Until \(\tau_R\), a jump of magnitude greater than \(2R\) forces exit from
\((-R,R)\).  Hence \(\tau_R\) is bounded by the first such jump time, which
is exponential with finite mean.  In particular \(\EE_x\tau_R<\infty\).
Moreover, if \(L>2R\), an exit location of modulus greater than \(L\) must
arise from a jump of magnitude greater than \(L-R\).  The L\'evy-system
formula therefore gives
\begin{equation}
  H_R(x,\{|z|>L\})
  \leq C_\beta\EE_x\tau_R(L-R)^{-\beta},\qquad L>2R.
  \label{eq:rough-stable-exit-tail}
\end{equation}
Since \(H<\beta\), the layer-cake formula and
\eqref{eq:rough-stable-exit-tail} yield
\begin{equation*}
 \int |z|^H H_R(x,\mathrm dz)
 \leq (2R)^H
   +H\int_{2R}^{\infty}L^{H-1}H_R(x,\{|z|>L\})\,\mathrm dL
 <\infty.
\end{equation*}
The map \(z\mapsto B_z\) is continuous in Gaussian \(L^2\), because
\(\|B_z-B_w\|_2=|z-w|^H\), and \(\|B_z\|_2=|z|^H\).  Consequently the
Bochner integral
\begin{equation*}
  P_RB_x:=\int B_zH_R(x,\mathrm dz)
\end{equation*}
is well-defined.  Approximation by simple functions whose values are
exterior coordinates shows that \(P_RB_x\in\mathscr E_R\).

For every \(y\in D\), continuity of covariance as a linear functional and
\eqref{eq:rough-green-balayage} give
\begin{equation}
 \operatorname{Cov}(B_x-P_RB_x,B_y)
 =\frac{1}{2c_\beta}\left\{
   g_D(x,y)-\int g_D(z,y)H_R(x,\mathrm dz)
  \right\}
 =\frac{1}{2c_\beta}g_{D_R}(x,y).
 \label{eq:rough-exterior-residual-against-coordinate}
\end{equation}
If \(|y|\geq R\), the last member is zero.  Thus \(B_x-P_RB_x\) is
orthogonal to every generator of \(\mathscr E_R\), and hence to
\(\mathscr E_R\) itself.  Since \(P_RB_x\in\mathscr E_R\), this proves
\(P_RB_x=\Pi_RB_x\).  For \(x,y\in D_R\), the first residual is orthogonal
to \(P_RB_y\), so \eqref{eq:rough-exterior-residual-against-coordinate}
also gives
\begin{equation}
  \operatorname{Cov}(B_x-P_RB_x,B_y-P_RB_y)
  =\frac{1}{2c_\beta}g_{D_R}(x,y).
  \label{eq:rough-exterior-residual-covariance}
\end{equation}
Subtracting this residual covariance from \eqref{eq:rough-green-covariance}
gives the exact prediction covariance
\begin{equation}
 \operatorname{Cov}(\Pi_RB_x,\Pi_RB_y)
 =\frac{1}{2c_\beta}\{g_D(x,y)-g_{D_R}(x,y)\},
 \qquad x,y\in D_R.
 \label{eq:rough-exterior-projection-covariance}
\end{equation}

For fixed \(y\), subadditivity gives
\begin{equation*}
  0\leq g_D(z,y)\leq2c_\beta|y|^\alpha.
\end{equation*}
For \(\beta\in(1,2)\), the symmetric \(\beta\)-stable process hits each
point almost surely \parencite{blumenthalGetoorRay1961FirstHitsSymmetricStable};
in particular \(T_0<\infty\) almost surely.  Every c\`adl\`ag path is bounded
on the compact random interval \([0,T_0]\).  Therefore the decreasing events
\(\{\sigma_R<T_0\}\) have empty intersection almost surely, and
\begin{equation*}
  H_R(x,\mathbb R)=\PP_x(\sigma_R<T_0)\longrightarrow0.
\end{equation*}
The remainder in \eqref{eq:rough-green-balayage} is bounded by
\(2c_\beta|y|^\alpha H_R(x,\mathbb R)\), and hence vanishes.  Thus, for fixed
\(x,y\neq0\),
\begin{equation*}
  g_{D_R}(x,y)\longrightarrow g_D(x,y);
\end{equation*}
domain monotonicity also shows that this convergence is increasing.  Finally,
fixed \(s,t\neq0\) belong to \(D_R\) for all sufficiently large \(R\), so
\eqref{eq:rough-exterior-projection-covariance} proves
\eqref{eq:rough-vanishing-exterior-prediction}.
\end{proof}

\begin{proposition}
\label{prop:rough-ordinary-fixed-point-germs}
If \(X\sim Q\), then, for both \(\varepsilon\in\{-1,+1\}\),
\begin{equation}
  \delta^{-H}
  \{X(1+\varepsilon\delta t)-X(1)\}_{0\leq t\leq1}
  \Longrightarrow \{B_t\}_{0\leq t\leq1}
  \quad\text{in }C([0,1])
  \label{eq:rough-ordinary-fixed-point-germs}
\end{equation}
as \(\delta\downarrow0\).
\end{proposition}

\begin{proof}
Write \(G=-B\).  We first work under the finite-grid hard-wall law
conditioned on
\(A_F=\{B_s\leq0:s\in F\}=\{G_s\geq0:s\in F\}\).  Put
\(r_\delta=\delta^{1/2}\), take \(\delta<1/2\), and suppose that a finite
dyadic constraint set \(F\) satisfies
\begin{equation}
  1\in F,\qquad F\cap(1-r_\delta,1+r_\delta)=\{1\}.
  \label{eq:rough-germ-outer-wall-geometry}
\end{equation}
For a fixed finite \(T\subset(0,1]\), define
\begin{equation}
  U_{\delta,t}
  :=\delta^{-H}\{B_1-B_{1+\varepsilon\delta t}\},
  \qquad t\in T,
  \label{eq:rough-germ-vector}
\end{equation}
and let \(W=G_F=-B_F\).  Gaussian regression gives
\begin{equation}
  U_\delta=L_\delta+R_\delta,\qquad
  L_\delta=\EE(U_\delta\mid W),
  \label{eq:rough-germ-regression}
\end{equation}
where \(R_\delta\) is centred Gaussian and independent of \(W\).

Set \(\widetilde B_u=B_1-B_{1+u}\).  It has the two-sided fractional Brownian
law and
\begin{equation}
  -B_1=-\widetilde B_{-1},\qquad
  -B_s=-\widetilde B_{-1}+\widetilde B_{s-1}.
  \label{eq:rough-germ-wall-span}
\end{equation}
Because \(r_\delta<1\), the first coordinate on the right of
\eqref{eq:rough-germ-wall-span} is an exterior coordinate; every other
coordinate is a linear combination of two exterior coordinates under
\eqref{eq:rough-germ-outer-wall-geometry}.  Thus the Gaussian span of \(W\)
is contained in that of
\(\{\widetilde B_u:|u|\geq r_\delta\}\).  Joint self-similarity of the target
and exterior coordinates turns this into the exterior span of
Lemma~\ref{lem:rough-vanishing-exterior-prediction} at radius
\(r_\delta/\delta=\delta^{-1/2}\).  Projection onto the smaller finite span
cannot increase variance.  Consequently, for every fixed
\(\theta\in\mathbb R^T\),
\begin{equation}
\begin{split}
  \sup_F\operatorname{Var}(\theta^{\mathsf T}L_\delta)
  &\leq
  \operatorname{Var}\left(
    \Pi_{\delta^{-1/2}}
    \sum_{t\in T}\theta_tB_{\varepsilon t}
  \right)\\
  &=\sum_{s,t\in T}\theta_s\theta_t
    \operatorname{Cov}\left(
      \Pi_{\delta^{-1/2}}B_{\varepsilon s},
      \Pi_{\delta^{-1/2}}B_{\varepsilon t}
    \right)
  \longrightarrow0,
\end{split}
  \label{eq:rough-germ-prediction-variance}
\end{equation}
where the supremum is over the finite dyadic \(F\) satisfying
\eqref{eq:rough-germ-outer-wall-geometry}.  The final limit is entrywise
Lemma~\ref{lem:rough-vanishing-exterior-prediction}, applied only at the fixed
nonzero times \(\varepsilon t\).

Rare orthant conditioning does not amplify this small projection.  The
event \(A_F\) is the positive-orthant event for \(W\), and independence and
centering of \(R_\delta\) give
\(\EE(L_\delta\mid A_F)=\EE(U_\delta\mid A_F)\).  The annular mean estimate
therefore gives
\begin{equation}
  |\EE(L_{\delta,t}\mid A_F)|
  =|\EE(U_{\delta,t}\mid A_F)|
  \leq C\delta^{-H}(\delta t)^\alpha
  \leq C_T\delta^H.
  \label{eq:rough-germ-conditioned-mean}
\end{equation}
Apply Lemma~\ref{lem:harge-convex-domination} to the positive-orthant
truncation of \(W\) and the linear functional
\(\theta^{\mathsf T}L_\delta\).  Then
\begin{equation}
 \EE\left[
  e^{\lambda\theta^{\mathsf T}
   \{L_\delta-\EE(L_\delta\mid A_F)\}}
  \,\middle|\,A_F\right]
 \leq
 e^{\lambda^2\operatorname{Var}(\theta^{\mathsf T}L_\delta)/2}.
 \label{eq:rough-germ-projection-harge}
\end{equation}
Equations \eqref{eq:rough-germ-prediction-variance}--%
\eqref{eq:rough-germ-projection-harge} imply
\(L_\delta\to0\) in probability under the hard-wall law, uniformly in \(F\).
Indeed, the two signs of \(\lambda\) in
\eqref{eq:rough-germ-projection-harge} give sub-Gaussian tails for the
centred projection, and the variance parameter and its conditional mean both
vanish uniformly.  Before conditioning, \(U_\delta\) has exactly the centred
standard fractional-Brownian law on \(T\), for either sign of
\(\varepsilon\).  Orthogonal regression gives
\(\operatorname{Cov}(R_\delta)
=\operatorname{Cov}(U_\delta)-\operatorname{Cov}(L_\delta)\), so the
covariance of \(R_\delta\) tends uniformly to that of
\((B_t)_{t\in T}\).  Since \(R_\delta\) is independent of \(W\), conditioning
on \(A_F\in\sigma(W)\) leaves its centred Gaussian law unchanged and keeps it
independent of \(L_\delta\).  Slutsky's theorem therefore proves uniform
finite-dimensional convergence, in bounded-Lipschitz distance, for the laws
with only the outer constraints.

We next justify uniformly the deletion of the local constraints just used.
For finite dyadic \(F\ni1\) and \(0<r<1/2\), put
\begin{equation*}
  F_r^{\mathrm{out}}:=(F\setminus(1-r,1+r))\cup\{1\}.
\end{equation*}
Under the one-site hard-wall law \(G_1\mid\{G_1\geq0\}\), the variable
\(G_1\) is standard half-normal.  Adding the exterior constraint sites
stochastically raises it by Proposition~\ref{prop:wall-monotonicity}; hence,
for \(0<a\leq1\),
\begin{equation*}
  \sup_{F\ni1}\PP(G_1\leq a\mid A_{F_r^{\mathrm{out}}})\leq Ca.
\end{equation*}
On \([1/2,3/2]\), \eqref{eq:annular-mean-modulus}, the centred
moment-generating-function bound in
Lemma~\ref{lem:harge-convex-domination}, and the chaining proof of
Lemma~\ref{lem:uniform-centered-modulus} give, for every fixed \(a>0\),
\begin{equation*}
 \omega(a,r):=\sup_E\PP\left(
  \sup_{|u|\leq r}|B_{1+u}-B_1|>a\,\middle|\,A_E
 \right)\longrightarrow0
 \qquad(r\downarrow0)
\end{equation*}
over finite dyadic \(E\).  To see the uniformity directly, the deterministic
mean oscillation is at most \(C_Hr^\alpha\), and the centred chaining bound is
then applied at the remaining fixed threshold.

Let \(N_r(F)=F\setminus F_r^{\mathrm{out}}\).  The events are nested and
\(A_F=A_{F_r^{\mathrm{out}}}\cap A_{N_r(F)}\).  If \(G_1>a\) and the displayed
oscillation is at most \(a\), then \(G_s\geq0\) for every
\(s\in N_r(F)\).  Hence the exact nested-event total-variation identity gives
\begin{equation*}
\begin{split}
 &d_{\mathrm{TV}}\left(
  \mathcal L(G\mid A_F),
  \mathcal L(G\mid A_{F_r^{\mathrm{out}}})
 \right)\\
 &\qquad=
 \PP\left(A_{N_r(F)}^c\,\middle|\,A_{F_r^{\mathrm{out}}}\right)
 \leq Ca+\omega(a,r).
\end{split}
\end{equation*}
Taking first \(r\downarrow0\) and then \(a\downarrow0\) proves
\begin{equation*}
 \sup_{F\ni1}d_{\mathrm{TV}}\left(
  \mathcal L(G\mid A_F),
  \mathcal L(G\mid A_{F_r^{\mathrm{out}}})
 \right)\longrightarrow0.
\end{equation*}

We now make the order of limits explicit.  Define the continuous map
\begin{equation*}
 \Phi_{\delta,\varepsilon}(f)(t)
 :=\delta^{-H}\{f(1+\varepsilon\delta t)-f(1)\},
 \qquad 0\leq t\leq1,
\end{equation*}
from the local-uniform path space to \(C([0,1])\), and let \(\gamma_T\) be
the law of \((B_t)_{t\in T}\).  The deletion estimate at
\(r=r_\delta\) and the outer-wall regression give deterministic errors
\(\eta(\delta),\rho_T(\delta)\downarrow0\) such that, uniformly in the
canonical wall index \(m\),
\begin{equation*}
 d_{\mathrm{BL}}\left(
  \mathcal L\left(
   (\Phi_{\delta,\varepsilon}(G)(t))_{t\in T}
   \,\middle|\,A_{E_m}
  \right),\gamma_T
 \right)
 \leq\eta(\delta)+\rho_T(\delta).
\end{equation*}
For each fixed \(\delta\), Theorem~\ref{thm:path-space-entrance-law} gives
\(\mathcal L(G\mid A_{E_m})\Rightarrow\mathcal L(-X)\), and continuity of
\(\Phi_{\delta,\varepsilon}\) passes this inequality to \(m=\infty\).
Thus one first lets \(m\to\infty\), with \(\delta\) fixed, and only then lets
\(\delta\downarrow0\).  Since
\(\Phi_{\delta,\varepsilon}(-X)=-\Phi_{\delta,\varepsilon}(X)\) and
\(\gamma_T\) is symmetric, this proves the finite-dimensional form of
\eqref{eq:rough-ordinary-fixed-point-germs}.

It remains to prove path-space tightness.  Recall that
\(m_F(t)=\EE(B_t\mid A_F)\), and, under the finite-wall law, put
\(Z_F:=G+m_F\), so that \(G=-m_F+Z_F\) and \(Z_F\) is centred.  Uniformly
for \(s,t\in[0,1]\) and \(0<\delta\leq1/2\),
\begin{equation}
 \delta^{-H}
 |m_F(1+\varepsilon\delta t)
   -m_F(1+\varepsilon\delta s)|
 \leq C_H\delta^H|t-s|^\alpha,
  \label{eq:rough-germ-rescaled-mean-modulus}
\end{equation}
and the same centred moment-generating-function bound gives
\begin{equation}
 \EE\left[\exp\left\{\lambda\delta^{-H}
 [Z_F(1+\varepsilon\delta t)-Z_F(1+\varepsilon\delta s)]\right\}
 \,\middle|\,A_F\right]
 \leq e^{\lambda^2|t-s|^\alpha/2}.
  \label{eq:rough-germ-rescaled-centered-bound}
\end{equation}
Choose \(p>1/H\).  The two-sided sub-Gaussian bound gives the corresponding
absolute \(p\)-moment bound for the centred increment.  Combining it with
\eqref{eq:rough-germ-rescaled-mean-modulus}, using
\(|t-s|^{2H}\leq|t-s|^H\) on \([0,1]\), yields
\begin{equation*}
 \EE\left[
  |\Phi_{\delta,\varepsilon}(G)(t)
    -\Phi_{\delta,\varepsilon}(G)(s)|^p
  \,\middle|\,A_F
 \right]
 \leq C_{p,H}|t-s|^{pH},
\end{equation*}
uniformly in \(F\) and \(\delta\leq1/2\).  For each fixed \(\delta\), the
continuous-mapping consequence of
Theorem~\ref{thm:path-space-entrance-law} and Portmanteau for the nonnegative
continuous increment functional pass this same upper bound to
\(-\Phi_{\delta,\varepsilon}(X)\), and hence to
\(\Phi_{\delta,\varepsilon}(X)\).  These processes vanish at \(t=0\), and
\(pH>1\).  The Kolmogorov tightness criterion therefore proves tightness in
\(C([0,1])\), uniformly as \(\delta\downarrow0\).  Together with the
finite-dimensional limit, this proves
\eqref{eq:rough-ordinary-fixed-point-germs}.
\end{proof}

\begin{corollary}
\label{cor:rough-interior-window-maximizer}
The unique \(Q\)-maximiser on \(I=[u,v]\) belongs to \((u,v)\) almost
surely.
\end{corollary}

\begin{proof}
For fixed \(q>0\), \(\varepsilon\in\{-1,+1\}\), and \(\delta>0\), put
\[
 Y_{q,\delta}^{\varepsilon}(t)
 :=\delta^{-H}\{X(q+\varepsilon\delta t)-X(q)\},
 \qquad 0\leq t\leq1.
\]
The self-similarity in \eqref{eq:limit-selfsimilarity} gives
\begin{equation}
 \{Y_{q,\delta}^{\varepsilon}(t)\}_{0\leq t\leq1}
 \deq
 (\delta/q)^{-H}
 \{X(1+\varepsilon(\delta/q)t)-X(1)\}_{0\leq t\leq1}.
 \label{eq:rough-general-deterministic-time-germ}
\end{equation}
Since \(\delta/q\downarrow0\),
Proposition~\ref{prop:rough-ordinary-fixed-point-germs} therefore gives
\begin{equation*}
 Y_{q,\delta}^{\varepsilon}
 \Longrightarrow \{B_t\}_{0\leq t\leq1}
 \quad\text{in }C([0,1])
 \qquad(\delta\downarrow0).
\end{equation*}

Let
\[
 \mathcal C_-:=
 \left\{f\in C([0,1]):\max_{0\leq t\leq1}f(t)\leq0\right\}.
\]
The maximum functional is continuous for the uniform norm, so
\(\mathcal C_-\) is closed.  Moreover,
\begin{equation*}
 \PP\{\{B_t\}_{0\leq t\leq1}\in\mathcal C_-\}=0.
\end{equation*}
Indeed, on this event \(B_0=0\) makes zero a maximiser on \([0,1]\), whereas
Lemma~\ref{lem:unique-interior-maximizer} says that the almost surely unique
maximiser belongs to \((0,1)\).

Put
\[
 A_u:=\left\{X(u)=\max_{r\in I}X(r)\right\}.
\]
For every \(0<\delta\leq v-u\), the inclusion
\(A_u\subseteq\{Y_{u,\delta}^{+1}\in\mathcal C_-\}\) holds.  Hence
closed-set Portmanteau gives
\begin{align*}
 \PP(A_u)
 &\leq\liminf_{\delta\downarrow0}
       \PP\{Y_{u,\delta}^{+1}\in\mathcal C_-\}\\
 &\leq\limsup_{\delta\downarrow0}
       \PP\{Y_{u,\delta}^{+1}\in\mathcal C_-\}\\
 &\leq\PP\{\{B_t\}_{0\leq t\leq1}\in\mathcal C_-\}=0.
\end{align*}
Likewise, if
\[
 A_v:=\left\{X(v)=\max_{r\in I}X(r)\right\},
\]
then \(A_v\subseteq\{Y_{v,\delta}^{-1}\in\mathcal C_-\}\) for every
\(0<\delta\leq v-u\), and the same Portmanteau chain gives
\(\PP(A_v)=0\).

Proposition~\ref{prop:rough-unique-window-maximizer} supplies an almost surely
unique maximiser on \(I\).  Since neither endpoint can be that maximiser, it
belongs to \((u,v)\) almost surely.
\end{proof}

\subsubsection{Assembly of the second zoom}
\label{subsec:rough-window-rerooting-assembly}

\begin{proof}[Proof of Theorem~\ref{thm:rough-window-rerooting}]
We first prove the assertion for the fixed positive interval
\(I=[u,v]\) in \eqref{eq:rough-fixed-positive-window}.  Uniqueness and
interiority are
Proposition~\ref{prop:rough-unique-window-maximizer} and
Corollary~\ref{cor:rough-interior-window-maximizer}.  It remains to prove
\eqref{eq:rough-window-rerooting-limit}.

Let \(b_n\to\infty\) be arbitrary.  Under the canonical finite-grid hard-wall
law \(\mu_m=\mathcal L(B\mid B_t\leq0:t\in E_m)\), let \(S_m\) be the
maximiser on \(D_m=E_m\cap I\).  For each fixed \(0<b<\infty\), the path-space
hard-wall limit theorem, the deterministic grid selection in
Proposition~\ref{prop:rough-unique-window-maximizer}, and the extended
continuous mapping theorem give
\begin{equation}
 \mathcal L_{\mu_m}\left(
  b^H\{B(S_m+\,\cdot\,/b)-B(S_m)\}\right)
 \Longrightarrow \mathcal L_Q(\mathcal R_b^IX),
 \qquad m\to\infty.
 \label{eq:rough-fixed-zoom-wall-approximation}
\end{equation}
Here is the required deterministic continuity check.  If \(x_m\to x\)
locally uniformly, \(x\) has a unique maximiser \(s=S_I(x)\), and
\(s_m=S_m(x_m)\), then \(s_m\to s\).  For each \(R<\infty\), with
\(K_R=[u-R/b,v+R/b]\),
\[
\begin{split}
 &\sup_{|t|\leq R}\left|
  b^H\{x_m(s_m+t/b)-x_m(s_m)\}
  -b^H\{x(s+t/b)-x(s)\}
 \right|\\
 &\qquad\leq
  2b^H\|x_m-x\|_{C(K_R)}
  +2b^H\omega_{x,K_R}(|s_m-s|)\longrightarrow0,
\end{split}
\]
where \(\omega_{x,K_R}\) is the uniform modulus of \(x\) on \(K_R\).
The finite-grid selectors are Borel, and the unique-maximiser set has full
\(Q\)-measure, so this is exactly the almost-sure continuity set required by
the extended mapping theorem.  The same argument without the path increment
also gives
\begin{equation}
  S_m\Longrightarrow S_I(X)
  \qquad\text{under }\mu_m.
  \label{eq:rough-grid-window-maximizer-limit}
\end{equation}

The endpoint cutoff must shrink slowly enough to remain inside the proved
geometric scope of
Proposition~\ref{prop:rough-selected-germ-invisibility}.  Put
\(\eta_k=2^{-k-3}(v-u)\), and denote by \(e_{I,k}(b)\to0\) the supremum in
\eqref{eq:rough-selected-germ-invisibility} for the fixed interior window
\([u+\eta_k,v-\eta_k]\).  Recursively choose integers
\(N_k>N_{k-1}\) so large that, for every \(n\geq N_k\),
\[
  b_n\eta_k\geq k,
  \qquad e_{I,k}(b_n)\leq k^{-1}.
\]
Taking \(k(n)=\max\{k:N_k\leq n\}\), with an arbitrary value before
\(N_1\), and \(\eta_n=\eta_{k(n)}\), gives \(k(n)\uparrow\infty\) and
\begin{equation}
  b_n\eta_n\longrightarrow\infty,\qquad
  e_{I,k(n)}(b_n)\longrightarrow0.
  \label{eq:rough-slow-boundary-diagonal}
\end{equation}
Interiority and continuity from above give
\begin{equation}
  Q\{\operatorname{dist}(S_I(X),\{u,v\})\leq2\eta_n\}
  \longrightarrow0.
  \label{eq:rough-continuum-boundary-layer}
\end{equation}

Choose a strictly increasing \(m(n)\) so large that
\begin{enumerate}[label=\textup{(\roman*)}]
\item the bounded-Lipschitz error in
  \eqref{eq:rough-fixed-zoom-wall-approximation}, with \(b=b_n\), is at most
  \(n^{-1}\);
\item \(b_n2^{-m(n)}\leq n^{-1}\);
\item the \(\mu_{m(n)}\)-probability that \(S_{m(n)}\) lies within
  \(\eta_n\) of an endpoint is at most the probability in
  \eqref{eq:rough-continuum-boundary-layer} plus \(n^{-1}\);
\item \(m(n)\geq m_I\), and hence \(s_0\in D_{m(n)}\).
\end{enumerate}
Indeed, for item~\textup{(iii)} choose a continuous
\(\chi_n:I\to[0,1]\) equal to one on the \(\eta_n\)-boundary layer and zero
outside the \(2\eta_n\)-boundary layer.  Equation
\eqref{eq:rough-grid-window-maximizer-limit} gives
\[
\begin{aligned}
 \limsup_{m\to\infty}\mu_m\{
   \operatorname{dist}(S_m,\{u,v\})\leq\eta_n\}
 &\leq \EE_Q\chi_n(S_I(X))\\
 &\leq Q\{\operatorname{dist}(S_I(X),\{u,v\})\leq2\eta_n\}.
\end{aligned}
\]
Thus all four requirements can be imposed after \(b_n\) and \(\eta_n\) have
been fixed.  Moreover, item~\textup{(ii)} and
\eqref{eq:rough-slow-boundary-diagonal} give
\[
  \frac{2^{-m(n)}}{\eta_n}
  \leq\frac{1}{n b_n\eta_n}\longrightarrow0,
\]
so the mesh restriction in
Proposition~\ref{prop:rough-selected-germ-invisibility} holds eventually.

Disintegrate the finite experiment by \(S_{m(n)}\).  Equations
\eqref{eq:reroot-two-root-event}--\eqref{eq:reroot-old-root-height} give the
exact selected two-root mixture.  Abbreviate
\[
 \mathsf P_n^{\mathrm{sel}}
 :=\mathsf P_{b_n,m(n),\eta_n}^{\mathrm{sel}},
 \qquad
 \mathsf P_n^{\mathrm{out}}
 :=\mathsf P_{b_n,m(n),\eta_n}^{\mathrm{out}},
\]
and let their common mass be
\begin{equation}
 w_n:=\sum_{s\in D_{m(n)}\cap I_{\eta_n}}p_{m(n),s}
     =\mu_{m(n)}\{S_{m(n)}\in I_{\eta_n}\}.
 \label{eq:rough-good-selected-mass}
\end{equation}
Items~\textup{(iii)} and \eqref{eq:rough-continuum-boundary-layer} give
\(w_n\to1\).  If
\begin{equation}
 \nu_n:=\sum_{s\in D_{m(n)}}p_{m(n),s}
       \mathcal L(G^{b_n,s}\mid A_{b_n,s}),
 \label{eq:rough-full-reflected-germ-mixture}
\end{equation}
then \(\nu_n-\mathsf P_n^{\mathrm{sel}}\) is the bad selected submeasure and
has mass \(1-w_n\).  Proposition
\ref{prop:rough-selected-germ-invisibility} and
\eqref{eq:rough-slow-boundary-diagonal} give
\begin{equation}
 d_{\mathrm{BL}}(\mathsf P_n^{\mathrm{sel}},
                 \mathsf P_n^{\mathrm{out}})
 \leq e_{I,k(n)}(b_n)\longrightarrow0.
 \label{eq:rough-good-selected-outer-replacement}
\end{equation}
This is an equal-mass finite-measure comparison, not a componentwise
probability-law assertion, and invokes no uniformity over unbounded inverted
geometric ratios.

Every outer component now has mesh
\(\widetilde h_n=b_n2^{-m(n)}\to0\).  For its retained centre \(s\), put
\[
 \widetilde L^-_{n,s}:=b_n(s-\min D_{m(n)}),
 \qquad
 \widetilde L^+_{n,s}:=b_n(\max D_{m(n)}-s).
\]
Since the first and last grid points lie within \(2^{-m(n)}\) of \(u\) and
\(v\), both horizons are at least
\(b_n\eta_n-b_n2^{-m(n)}\to\infty\), uniformly in the retained centre.
Sequential uniformity in
Theorem~\ref{thm:path-space-entrance-law} is stated for dyadic meshes, whereas
\(\widetilde h_n\) need not be dyadic.  We therefore make the following exact
reduction.  Write \(\widetilde Q_{h,L^-,L^+}\) for the conditional law in
\eqref{eq:pure-wall-law} with an arbitrary positive grid-aligned mesh \(h\),
and let
\[
  (\mathcal S_cf)(t):=c^Hf(t/c),\qquad c>0.
\]
Choose the dyadic \(\widehat h_n\in[\widetilde h_n,2\widetilde h_n)\) and
put \(c_n=\widehat h_n/\widetilde h_n\in[1,2)\).  Fractional-Brownian
self-similarity gives, for every outer component with horizons
\(\widetilde L^-_{n,s},\widetilde L^+_{n,s}\),
\[
  (\mathcal S_{c_n})_*
  \widetilde Q_{\widetilde h_n,\widetilde L^-_{n,s},\widetilde L^+_{n,s}}
  =Q_{\widehat h_n,c_n\widetilde L^-_{n,s},c_n\widetilde L^+_{n,s}}.
\]
The mesh on the right is dyadic and tends to zero, its horizons diverge
uniformly over the retained components, and grid alignment is preserved.
The sequentially uniform form of
Theorem~\ref{thm:path-space-entrance-law} therefore makes these scaled laws
converge uniformly to \(Q\).  If the corresponding unscaled convergence
failed, select violating retained centres \(s_n\) and a subsequence on which
\(c_n\to c\in[1,2]\).  Continuity of
\((c,f)\mapsto\mathcal S_{1/c}f\) in the local-uniform topology and
\eqref{eq:limit-selfsimilarity} would give
\[
  (\mathcal S_{1/c_n})_*
  Q_{\widehat h_n,c_n\widetilde L^-_{n,s_n},c_n\widetilde L^+_{n,s_n}}
  \Longrightarrow (\mathcal S_{1/c})_*Q=Q,
\]
a contradiction.  Thus the original arbitrary-mesh hard-wall components
also converge uniformly to \(Q\).  An outer component of
\(\mathsf P_n^{\mathrm{out}}\) is the law of \(-Y\), where
\(Y\sim\widetilde Q_{\widetilde h_n,
 \widetilde L^-_{n,s},\widetilde L^+_{n,s}}\).
Consequently there is \(\delta_n\to0\) such that, for \(X\sim Q\),
\[
 d_{\mathrm{BL}}\left(
   \mathsf P_n^{\mathrm{out}},w_n\mathcal L(-X)
 \right)\leq w_n\delta_n.
\]
Testing against functions bounded by one and Lipschitz with constant at most
one, \eqref{eq:rough-good-selected-outer-replacement} and the bad mass give
\[
 d_{\mathrm{BL}}\bigl(\nu_n,\mathcal L(-X)\bigr)
 \leq2(1-w_n)+e_{I,k(n)}(b_n)+w_n\delta_n\longrightarrow0.
\]
The rerooted path is \(-G^{b_n,s}\) by
\eqref{eq:reroot-reflected-zoom}.  Applying this continuous global sign
change to the last convergence gives
\begin{equation}
 \mathcal L_{\mu_{m(n)}}\left(
  b_n^H\{B(S_{m(n)}+\,\cdot\,/b_n)-B(S_{m(n)})\}\right)
 \Longrightarrow Q.
 \label{eq:rough-diagonal-rerooting-limit}
\end{equation}
Item~\textup{(i)} and the triangle inequality transfer
\eqref{eq:rough-diagonal-rerooting-limit} to
\(\mathcal L_Q(\mathcal R^I_{b_n}X)\).  Since \((b_n)\) was arbitrary, the
sequential characterisation of weak convergence proves
\eqref{eq:rough-window-rerooting-limit}.

We now remove the restriction to positive windows.  Let
\((\Theta f)(t):=f(-t)\).  By
\eqref{eq:rough-tangent-reflection-invariance}, pathwise
\begin{equation}
  S_{-I}(\Theta f)=-S_I(f),\qquad
  \mathcal R_b^{-I}(\Theta f)=\Theta(\mathcal R_b^If).
  \label{eq:rough-window-reflection-covariance}
\end{equation}
The positive-window result and reflection invariance therefore prove the
unique-interior-maximiser and rerooting assertions for every fixed negative
window.

Finally, \eqref{eq:limit-support} shows that if a fixed nondegenerate compact
interval \(J\) contains zero, its unique maximiser is zero.  In that case
\begin{equation}
  (\mathcal R_b^JX)(t)=b^HX(t/b),
\end{equation}
which has law \(Q\) for every \(b>0\) by
\eqref{eq:limit-selfsimilarity}.  If \(0\in\operatorname{int}J\), this proves
Definition~\ref{def:asymptotic-window-rerooting-invariance} exactly.  If
\(0\in\partial J\), the scaling identity remains exact but the definition
does not apply because the unique maximiser is not interior, as stated in
Remark~\ref{rem:rough-window-rerooting-scope}.  This completes the theorem.
\end{proof}

\section{Sharp persistence asymptotics: proof of Theorem~\ref{thm:rough-sharp-persistence}}
\label{sec:rough-persistence}

The proof separates the scalar persistence asymptotic from the construction
of the tilted path limit.  An exact factorisation reduces the sharp
asymptotic to convergence of the reciprocal exponential functional
\(\widehat{\EE}_a[J_a^{-1}]\).  Once this convergence is known, Molchan's
limit for the uncentred exponential functional determines the asymptotic
of \(Z_a\), and a Tauberian argument converts that asymptotic into the stated
persistence probability.

The main task is therefore to identify the path seen from its maximum under
the tilt \(\widehat{\PP}_a\).  We first represent its finite-grid
approximations as mixtures of endpoint-tilted hard-wall laws and establish
compact containment of the rescaled maximum time and height.  We then
construct the finite-left-horizon entrance components and identify the
unique limiting mixture \(\mu_H\).  Finally, block repulsion controls the
growing right-arm contribution to \(J_a\), while a separate reciprocal
moment bound gives uniform integrability; together these estimates yield
the required convergence of \(\widehat{\EE}_a[J_a^{-1}]\).

\subsection{Scalar reduction}

\label{subsec:rough-persistence-scalar-reduction}

Recall the definition of \(J_a\) in
\eqref{eq:rough-persistence-functional}.
For the scalar comparison, let
\begin{equation}
  I(a):=\int_0^a e^{B_s}\,\mathrm ds.
  \label{eq:molchan-exponential-functional}
\end{equation}

\begin{proposition}
\label{prop:exact-persistence-compiler}
For every \(a>0\),
\begin{equation}
  Z_a\,\widehat{\EE}_a[J_a^{-1}]
  =\EE[I(a)^{-1}].
  \label{eq:exact-persistence-factorization}
\end{equation}
Furthermore,
\begin{equation}
  a^{1-H}\EE[I(a)^{-1}]
  \longrightarrow H\,\EE[M].
  \label{eq:molchan-functional-limit}
\end{equation}
Consequently, if
\(\widehat{\EE}_a[J_a^{-1}]\longrightarrow D_H\in(0,\infty)\), then
\begin{equation}
  a^{1-H}Z_a\longrightarrow\frac{H\,\EE[M]}{D_H},
  \label{eq:rough-normalizer-sharp-limit}
\end{equation}
and then \eqref{eq:rough-sharp-persistence-asymptotic}--
\eqref{eq:rough-persistence-constant} follow.
\end{proposition}

\begin{proof}
The substitutions \(u=\tau-s/a\) and \(u=\tau+s/a\) in the two terms of
\eqref{eq:rough-persistence-functional} give the first identity below;
self-similarity gives the second:
\[
  J_a=a e^{-a^HM}\int_0^1 e^{a^HB_u}\,\mathrm du,
  \qquad
  I(a)\deq a\int_0^1e^{a^HB_u}\,\mathrm du.
\]
Therefore
\[
  \EE[I(a)^{-1}]
  =\EE\bigl[e^{-a^HM}J_a^{-1}\bigr]
  =Z_a\widehat{\EE}_a[J_a^{-1}],
\]
which is \eqref{eq:exact-persistence-factorization}.

In Molchan's notation the covariance exponent is \(\gamma=2H\).
Fractional Brownian motion satisfies his affine-invariance relation~(2), and
\(\sup_{0\leq t\leq1}|B_t|\) has every linear exponential moment.  Thus
\textcite[Statement~1, equation~(8)]{molchan1999maximum} applies and gives
\eqref{eq:molchan-functional-limit}.  Moreover,
\(0<\EE[M]<\infty\), since \(M\geq\max\{0,B_1\}\) and the preceding
exponential moment bound controls \(M\).  Combining Molchan's limit with the
assumed reciprocal convergence proves
\eqref{eq:rough-normalizer-sharp-limit}.

It remains to pass from the Laplace transform to the lower tail of \(M\).
Put
\begin{equation}
  \rho:=\frac{1-H}{H}=\frac1H-1
  \quad\text{and}\quad
  \lambda:=a^H.
  \label{eq:rough-tauberian-exponent}
\end{equation}
Then \eqref{eq:rough-normalizer-sharp-limit} becomes
\begin{equation}
  \lambda^\rho\EE e^{-\lambda M}
  \longrightarrow\frac{H\,\EE[M]}{D_H}.
  \label{eq:rough-laplace-sharp-limit}
\end{equation}
Karamata's Laplace--Stieltjes Tauberian theorem
\parencite[Theorem~XIII.5.3]{feller1971probabilityVol2} gives
\begin{equation}
  \PP(M\leq\varepsilon)
  \sim
  \frac{H\,\EE[M]}{\Gamma(\rho+1)D_H}\,
  \varepsilon^\rho
  =\frac{H\,\EE[M]}{\Gamma(1/H)D_H}\,
  \varepsilon^{1/H-1}.
  \label{eq:rough-maximum-lower-tail}
\end{equation}
Finally, self-similarity gives
\begin{equation}
  \PP\bigl(B_t\leq1\text{ for }0\leq t\leq T\bigr)
  =\PP(M\leq T^{-H}).
  \label{eq:rough-persistence-selfsimilarity}
\end{equation}
Substitution into \eqref{eq:rough-maximum-lower-tail} proves
\eqref{eq:rough-sharp-persistence-asymptotic} and
\eqref{eq:rough-persistence-constant}.
\end{proof}

\subsection{Endpoint-tilted finite-grid hard-wall laws and scalar compact containment}
\label{subsec:rough-tilted-source}

The maximum-functional change of measure has Brownian penalisation
precedents \parencite[Theorem~3.6]{roynetteValloisYor2006maximum}, including
its interaction with scaling \parencite{panzo2019scaled}.  For fractional
Brownian motion, \textcite{aurzadaBuckKilian2020penalizing} instead penalise
one-sided negativity through Molchan's reciprocal exponential functional.
These works motivate the penalisation viewpoint, which we extend by the dependent finite-grid disintegration below.

Fix an arbitrary positive sequence \(a_n\to\infty\).  Let \(r_n\in\mathbb N\) be
auxiliary integers, to be chosen later by diagonal refinement, and set the
dyadic tangent mesh \(h_n:=2^{-r_n}\).  The corresponding physical grid is
\begin{equation}
  \Delta_n:=\frac{h_n}{a_n},
  \qquad
  m_n:=\left\lfloor\frac{a_n}{h_n}\right\rfloor,
  \qquad
  \mathcal G_n:=\{j\Delta_n:0\leq j\leq m_n\}.
  \label{eq:tilted-physical-grid}
\end{equation}
Let \(K_n\) be the almost surely unique maximising grid index and put
\begin{equation}
  \tau_n:=K_n\Delta_n,
  \qquad
  M_n^*:=B_{\tau_n}.
  \label{eq:tilted-grid-maximizer}
\end{equation}
The centred grid zoom is
\begin{equation}
  \Xi_n^{\mathrm{grid}}(t)
  :=a_n^H\bigl(B_{\tau_n+t/a_n}-B_{\tau_n}\bigr).
  \label{eq:tilted-grid-zoom}
\end{equation}
For \(0\leq k\leq m_n\), define
\begin{equation}
\begin{aligned}
  L_{n,k}^-&:=h_nk,
  &L_{n,k}^+&:=h_n(m_n-k),\\
  \Lambda_{n,k}&:=\{h_n(j-k):0\leq j\leq m_n\},
  &C_{n,k}&:=\{f:f(t)\leq0\text{ for }t\in\Lambda_{n,k}\}.
\end{aligned}
  \label{eq:tilted-component-geometry}
\end{equation}
If \(W\) is a generic standard two-sided fractional Brownian motion, let
\begin{equation}
  Q_{n,k}:=\mathcal L(W\mid W\in C_{n,k}),
  \qquad
  p_{n,k}:=\PP(W\in C_{n,k}),
  \label{eq:untilted-grid-components}
\end{equation}
and define
\begin{equation}
\begin{split}
  z_{n,k}
  &:=\EE_{Q_{n,k}}e^{W(-L_{n,k}^-)},\\
  \frac{\mathrm d\widehat Q_{n,k}}{\mathrm dQ_{n,k}}(f)
  &:=\frac{e^{f(-L_{n,k}^-)}}{z_{n,k}}.
\end{split}
  \label{eq:endpoint-tilted-components}
\end{equation}
The endpoint belongs to the constraint grid, so \(0<z_{n,k}\leq1\).

Using the path space \(\mathcal X\) and the marked space \(\mathcal Y\) from
\eqref{eq:rough-marked-space}, define for \(u,L\geq0\) the continuous map
\begin{equation}
\begin{aligned}
  \mathscr M_{u,L}:
  \{f\in\mathcal X:f(-u)\leq0\}
  &\longrightarrow\mathcal Y,\\
  \mathscr M_{u,L}(f)
  &:=\left(
    u,-f(-u),
    \bigl(-f(-(s\wedge u))\bigr)_{s\geq0},
    \bigl(-f(s\wedge L)\bigr)_{s\geq0}
  \right).
\end{aligned}
  \label{eq:tilted-grid-mark-map}
\end{equation}

\begin{proposition}
\label{prop:exact-tilted-wall-mixture}
Let \(Z_n^{\mathrm{grid}}\) and \(\widehat{\PP}_n^{\mathrm{grid}}\) be given by
\begin{equation}
  Z_n^{\mathrm{grid}}:=\EE e^{-a_n^HM_n^*},
  \qquad
  \frac{\mathrm d\widehat{\PP}_n^{\mathrm{grid}}}{\mathrm d\PP}
  :=\frac{e^{-a_n^HM_n^*}}{Z_n^{\mathrm{grid}}}.
  \label{eq:grid-tilt-normalizer}
\end{equation}
Then the global normaliser and component weights satisfy
\begin{equation}
  Z_n^{\mathrm{grid}}
  =\sum_{k=0}^{m_n}p_{n,k}z_{n,k},
  \qquad
  q_{n,k}:=\frac{p_{n,k}z_{n,k}}{Z_n^{\mathrm{grid}}},
  \label{eq:tilted-mixture-normalizer}
\end{equation}
and \((q_{n,k})_{0\leq k\leq m_n}\) is a probability vector.  Set
\begin{equation}
  Y_n^{\mathrm{grid}}
  :=\mathscr M_{L_{n,K_n}^-,L_{n,K_n}^+}
      (\Xi_n^{\mathrm{grid}})
  =\bigl(a_n\tau_n,a_n^HM_n^*,
      \overline L_n^{\mathrm{grid}},
      \overline R_n^{\mathrm{grid}}\bigr),
  \label{eq:tilted-grid-mark}
\end{equation}
where the last equality names the two constant-extended deficit arms.  The
path and marked laws have the exact disintegrations
\begin{equation}
\begin{aligned}
  \mathcal L_{\widehat{\PP}_n^{\mathrm{grid}}}
    (\Xi_n^{\mathrm{grid}})
  &=\sum_{k=0}^{m_n}q_{n,k}\widehat Q_{n,k},\\
  \mathcal L_{\widehat{\PP}_n^{\mathrm{grid}}}
    (Y_n^{\mathrm{grid}})
  &=\sum_{k=0}^{m_n}q_{n,k}
    (\mathscr M_{L_{n,k}^-,L_{n,k}^+})_*
    \widehat Q_{n,k}.
\end{aligned}
  \label{eq:exact-tilted-mixture}
\end{equation}
\end{proposition}

\begin{proof}
The rowwise calculation in
Proposition~\ref{prop:exact-pure-wall-mixture} applies to this deterministic
grid, including its endpoint components, and gives
\[
  \PP(K_n=k)=p_{n,k}>0,
  \qquad
  \mathcal L(\Xi_n^{\mathrm{grid}}\mid K_n=k)=Q_{n,k}.
\]
On \(\{K_n=k\}\), time zero has relative coordinate
\(-L_{n,k}^-=-h_nk\), so
\begin{equation}
  \Xi_n^{\mathrm{grid}}(-L_{n,k}^-)
  =a_n^H(B_0-B_{\tau_n})
  =-a_n^HM_n^*
  \label{eq:height-left-endpoint}
\end{equation}
and \eqref{eq:tilted-grid-mark} follows.  Therefore, for every bounded
Borel \(F:\mathcal X\to\mathbb R\),
\begin{align*}
  &\EE\left[
    e^{-a_n^HM_n^*}F(\Xi_n^{\mathrm{grid}});K_n=k
  \right]
  \\
  &\qquad
  =p_{n,k}\EE_{Q_{n,k}}\left[
       e^{W(-L_{n,k}^-)}F(W)
     \right]
  =p_{n,k}z_{n,k}\EE_{\widehat Q_{n,k}}F(W).
\end{align*}
Taking \(F=1\) and summing proves the normaliser identity.  The same
calculation gives
\(\widehat{\PP}_n^{\mathrm{grid}}(K_n=k)=q_{n,k}\) and conditional path law
\(\widehat Q_{n,k}\), proving the probability-vector and path-mixture
claims.  Applying the component-dependent map
\(\mathscr M_{L_{n,k}^-,L_{n,k}^+}\) gives the marked disintegration.

At \(k=0\), the endpoint is the pin, so \(z_{n,0}=1\) and the height is
zero.  At \(k=m_n\), the right horizon is zero and the same calculation
uses the direct one-sided endpoint component of
Proposition~\ref{prop:exact-pure-wall-mixture}.  If \(m_n=0\), these are the
same sole component and all identities reduce to one.
\end{proof}

Write
\(Y_{a_n}:=(U_{a_n},V_{a_n},\overline L_{a_n},\overline R_{a_n})\) for the
continuous-maximiser counterpart of \eqref{eq:tilted-grid-mark}.
Fix a bounded metric \(d_{\mathcal Y}\leq1\) generating the product topology
of \(\mathcal Y\), and let \(d_{\mathrm{BL}}\) be the associated
bounded-Lipschitz metric on probability laws.

\begin{proposition}
\label{prop:relative-tilted-grid-diagonal}
The integers \(r_n\) can be chosen strictly increasing, with
\(h_n\leq2^{-n}\), so that
\begin{equation}
  d_{\mathrm{BL}}\left(
    \mathcal L_{\widehat{\PP}_n^{\mathrm{grid}}}(Y_n^{\mathrm{grid}}),
    \mathcal L_{\widehat{\PP}_{a_n}}(Y_{a_n})
  \right)\longrightarrow0.
  \label{eq:relative-tilted-diagonal}
\end{equation}
The refinement is chosen after the scale sequence; no assertion is made for
an unrelated mesh sequence.
\end{proposition}

\begin{proof}
Fix \(n\).  For \(r\in\mathbb N\), put
\[
  h_{n,r}:=2^{-r},
  \qquad
  m_{n,r}:=\lfloor a_n2^r\rfloor,
  \qquad
  T_{n,r}:=\frac{m_{n,r}2^{-r}}{a_n},
\]
instantiate the grid variables with \(h=h_{n,r}\), and denote the resulting
maximising time, maximum, and mark by
\(\tau_{n,r}\), \(M_{n,r}^*\), and \(Y_{n,r}^{\mathrm{grid}}\).
Then \(T_{n,r}\to1\), and the fixed-row tangent-grid convergence
\eqref{eq:fixed-row-grid-argmax} and continuity of \(B\) give almost surely
\[
  \tau_{n,r}\to\tau,
  \qquad
  M_{n,r}^*\to M,
  \qquad
  Y_{n,r}^{\mathrm{grid}}\to Y_{a_n}
  \quad\text{in }\mathcal Y.
\]
Indeed, the constant-extended grid arms have the explicit forms
\[
\begin{aligned}
  \overline L_{n,r}^{\mathrm{grid}}(s)
  &=a_n^H\left[
      M_{n,r}^*-B_{(\tau_{n,r}-s/a_n)\vee0}
    \right],\\
  \overline R_{n,r}^{\mathrm{grid}}(s)
  &=a_n^H\left[
      M_{n,r}^*-B_{(\tau_{n,r}+s/a_n)\wedge T_{n,r}}
    \right].
\end{aligned}
\]
The continuous arms have the same formulas with
\((\tau_{n,r},M_{n,r}^*,T_{n,r})\) replaced by \((\tau,M,1)\).
The stopped arguments converge uniformly in \(s\geq0\), so uniform
continuity of \(B\) on \([0,1]\) proves the marked convergence in the
display above.

Put
\[
  w_{n,r}:=e^{-a_n^HM_{n,r}^*},
  \qquad
  w_n:=e^{-a_n^HM},
  \qquad
  Z_{n,r}^{\mathrm{grid}}:=\EE w_{n,r}.
\]
These variables lie in \((0,1]\), and \(Z_{a_n}=\EE w_n>0\).  Dominated
convergence gives
\[
  e_{n,r}:=\EE|w_{n,r}-w_n|\to0,
  \qquad
  \delta_{n,r}:=\EE\left[
    w_n d_{\mathcal Y}(Y_{n,r}^{\mathrm{grid}},Y_{a_n})
  \right]\to0
\]
and
\(|Z_{n,r}^{\mathrm{grid}}-Z_{a_n}|\leq e_{n,r}\).

If \(F:\mathcal Y\to[-1,1]\) is one-Lipschitz for \(d_{\mathcal Y}\),
then adding and subtracting unnormalised expectations gives
\[
\begin{aligned}
  &\left|
    \frac{\EE[w_{n,r}F(Y_{n,r}^{\mathrm{grid}})]}
         {Z_{n,r}^{\mathrm{grid}}}
    -\frac{\EE[w_nF(Y_{a_n})]}{Z_{a_n}}
  \right|\\
  &\qquad\leq
  \frac{|Z_{n,r}^{\mathrm{grid}}-Z_{a_n}|+e_{n,r}+\delta_{n,r}}
       {Z_{a_n}}
  \leq\frac{2e_{n,r}+\delta_{n,r}}{Z_{a_n}}.
\end{aligned}
\]
Starting with \(r_0:=0\), choose \(r_n\) recursively so large that
\[
  r_n>r_{n-1},
  \qquad
  2^{-r_n}\leq2^{-n},
  \qquad
  e_{n,r_n}\vee\delta_{n,r_n}\leq2^{-n}Z_{a_n}.
\]
For this choice, \(Y_{n,r_n}^{\mathrm{grid}}=Y_n^{\mathrm{grid}}\) and
\(Z_{n,r_n}^{\mathrm{grid}}=Z_n^{\mathrm{grid}}\).  Taking the supremum in
the preceding bound yields
\(d_{\mathrm{BL}}\leq3\cdot2^{-n}\), proving
\eqref{eq:relative-tilted-diagonal}.
\end{proof}

\label{subsubsec:rough-scalar-compact-containment}

For a finite dyadic constraint set \(F\), put
\begin{equation}
  Q_F:=\mathcal L(W\mid W_t\leq0,\ t\in F),
  \qquad G:=-W,
  \label{eq:generic-wall-law}
\end{equation}
and, for \(e\in F\cup\{0\}\), define
\begin{equation}
  \frac{\mathrm d\widehat Q_{F,e}}{\mathrm dQ_F}
  :=\frac{e^{-G_e}}{\EE_{Q_F}e^{-G_e}}.
  \label{eq:generic-endpoint-tilt}
\end{equation}
At \(e=0\) this is the unweighted hard-wall law.

\begin{lemma}
\label{lem:endpoint-tilted-unit-arm-moment}
For every \(\lambda>0\), there is \(C_{H,\lambda}<\infty\) such that, for
every finite \(F\subset\mathbb D\), every \(e\in F\cup\{0\}\), and either
\(I=[0,1]\) or \(I=[-1,0]\),
\begin{equation}
  \EE_{\widehat Q_{F,e}}
  \exp\left(\lambda\sup_{t\in I}G_t\right)
  \leq C_{H,\lambda},
  \label{eq:endpoint-tilted-unit-arm-moment}
\end{equation}
uniformly in \(F\), \(e\), and the side.
\end{lemma}

\begin{proof}
Fix \(\lambda>0\) and a side \(I\), and put
\[
  T_j:=2^{-j}\mathbb Z\cap I,
  \qquad j\geq0.
\]
Write
\[
  g_F(t):=\EE_{Q_F}G_t,
  \qquad
  Z_F(t):=G_t-g_F(t).
\]
Apply Lemma~\ref{lem:harge-convex-domination} to the joint Gaussian vector on
\(F\cup T_j\) and the wall rectangle \(\{G_t\geq0:t\in F\}\).  Its singular
formulation permits the deterministic coordinate at zero.  The test
\(x\mapsto\exp(\lambda\max_{t\in T_j}|x_t|)\) is convex because the maximum
of absolute coordinates is convex and the exponential is convex and
increasing.  Hence
\[
  \EE_{Q_F}\exp\left(\lambda\max_{t\in T_j}|Z_F(t)|\right)
  \leq
  \EE\exp\left(\lambda\sup_{t\in I}|W_t|\right).
\]
Equation~\eqref{eq:positive-covariance-mixture} makes \(g_F=-m_F\)
continuous, with the empty-wall case immediate.  Thus continuity and monotone
convergence let \(j\to\infty\).  The right-hand side is finite by Fernique's
theorem
\parencite[Theorem~2.8.5]{bogachev1998gaussianMeasures} and the inequality
\(\lambda x\leq\alpha x^2+\lambda^2/(4\alpha)\), for a suitable
Fernique constant \(\alpha>0\).  Since \eqref{eq:pin-mean-bound} gives
\(\sup_{t\in I}g_F(t)\leq C_H\), time reflection now yields, uniformly in
\(F\) and the side,
\[
  \EE_{Q_F}\exp\left(\lambda\sup_{t\in I}G_t\right)
  \leq
  e^{\lambda C_H}
  \EE\exp\left(\lambda\sup_{t\in[0,1]}|W_t|\right)
  <\infty.
\]

If \(e=0\), then \(\widehat Q_{F,e}=Q_F\), so the preceding bound finishes
the proof.  Suppose \(e\in F\).  Applying
Lemma~\ref{lem:finite-wall-mtp2-closure} to the union of \(F\), \(T_j\), and
\(\{e\}\), then restricting to the wall and marginalising unused coordinates,
shows that the tested coordinates and \(G_e\) are associated; the deterministic
zero coordinate is deleted.  Put
\[
  A_{j,N}:=
  \exp\left(\lambda\max_{t\in T_j}G_t\right)\wedge N.
\]
Both \(A_{j,N}\) and \(1-e^{-(G_e)_+}\) are bounded and increasing.  For
the present \(e\), the latter equals \(1-e^{-G_e}\) because \(G_e\geq0\).
Expanding \eqref{eq:association} and rearranging gives
\[
  \EE_{Q_F}\left[
    e^{-G_e}A_{j,N}
  \right]
  \leq
  \EE_{Q_F}e^{-G_e}\,
  \EE_{Q_F}A_{j,N}.
\]
Let \(N\to\infty\), divide by the strictly positive endpoint normaliser,
and then let \(j\to\infty\).  Monotone convergence and the unweighted bound
prove \eqref{eq:endpoint-tilted-unit-arm-moment}.
\end{proof}

The next argument obtains the compact containment needed for the marked
tilted laws directly from a fixed-factor normaliser comparison.
For \(T>0\), write
\begin{equation}
  M_T:=\max_{0\leq s\leq T}B_s,
  \qquad
  \sigma_T:=\operatorname*{argmax}_{0\leq s\leq T}B_s,
  \qquad
  \mathcal Z(T):=\EE e^{-M_T}.
  \label{eq:long-horizon-notation}
\end{equation}
Self-similarity identifies
\begin{equation}
  (V_a,U_a)\deq(M_a,\sigma_a),
  \qquad
  J_a\deq\int_0^ae^{B_s-M_a}\,\mathrm ds,
  \label{eq:long-horizon-scaling}
\end{equation}
under the corresponding likelihood \(e^{-M_a}/\mathcal Z(a)\).
We denote that likelihood by
\begin{equation}
  \frac{\mathrm d\widehat{\PP}_T}{\mathrm d\PP}
  :=\frac{e^{-M_T}}{\mathcal Z(T)},
  \qquad T>0.
  \label{eq:long-horizon-tilted-law}
\end{equation}

Set
\begin{equation}
  p_*:=\PP\left(\sup_{1\leq s\leq2}B_s\leq0\right)>0.
  \label{eq:fixed-block-positive-baseline}
\end{equation}
Strict positivity follows from the sign-reversed assertion of
Lemma~\ref{lem:fbm-compact-positivity-support} on \([1,2]\).

\begin{corollary}
\label{cor:tilted-height-tightness}
Let \(m_H:=\lceil1/H\rceil\).  Then, for every \(T>0\),
\begin{equation}
  \mathcal Z(2^{-1/H}T)
  \leq p_*^{-m_H}\mathcal Z(T)
  \label{eq:normalizer-fixed-factor}
\end{equation}
and, for \(A\geq0\),
\begin{equation}
  \sup_{a>0}\widehat{\PP}_a(V_a>A)
  \leq p_*^{-m_H}e^{-A/2}.
  \label{eq:tilted-height-tail}
\end{equation}
\end{corollary}

\begin{proof}
Put \(F_T(x):=\PP(M_T\leq x)\).  The fractional Brownian level process has
entrywise nonnegative covariance on nonnegative times and is therefore
associated by \textcite{pitt1982associatedGaussian}.  Association applies to
two decreasing indicators by negating both.  For \(x\geq0\), delete the
deterministic coordinate at zero, apply finite-dimensional association on
increasing finite grids, both containing \(T\) and with dense unions in
\([0,T]\) and \([T,2T]\), and then use path continuity.  Self-similarity gives
\begin{equation*}
  F_{2T}(x)
  \geq F_T(x)
  \PP\left(\sup_{T\leq s\leq2T}B_s\leq x\right)
  \geq p_*F_T(x).
\end{equation*}
Since \(M_T\geq0\), Tonelli's theorem gives
\begin{equation*}
  \mathcal Z(T)=\int_0^\infty e^{-x}F_T(x)\,\mathrm dx.
\end{equation*}
Consequently,
\begin{equation*}
  \mathcal Z(2T)\geq p_*\mathcal Z(T),
  \qquad T>0.
\end{equation*}
Put \(S=2^{-1/H}T\).  Since
\(T=2^{1/H}S\leq2^{m_H}S\), iteration and the fact that \(\mathcal Z\) is
nonincreasing give
\begin{equation*}
  p_*^{m_H}\mathcal Z(S)
  \leq\mathcal Z(2^{m_H}S)
  \leq\mathcal Z(T),
\end{equation*}
which is \eqref{eq:normalizer-fixed-factor}.

Finally, self-similarity gives \(Z_a=\mathcal Z(a)\) and
\(\EE e^{-V_a/2}=\mathcal Z(2^{-1/H}a)\).  Therefore, for every \(a>0\),
\begin{equation*}
\begin{split}
  \widehat{\PP}_a(V_a>A)
  &=\frac{\EE[e^{-V_a};\,V_a>A]}{Z_a}\\
  &\leq e^{-A/2}
       \frac{\mathcal Z(2^{-1/H}a)}{\mathcal Z(a)}
  \leq p_*^{-m_H}e^{-A/2}.
\end{split}
\end{equation*}
Taking the supremum over \(a\) proves \eqref{eq:tilted-height-tail}.
\end{proof}

\begin{lemma}
\label{lem:exponential-barrier-disintegration}
Let \(E\sim\operatorname{Exp}(1)\) be independent of \(B\).  For \(T>0\),
let \(\mathbb Q_T\) be the product law of \((B,E)\) conditioned on
\(\{M_T\leq E\}\).  Equivalently,
\begin{equation}
  \mathrm d\mathbb Q_T(B,x)
  :=\mathcal Z(T)^{-1}
    e^{-x}\mathds1_{\{M_T\leq x\}}\,\mathrm d\PP(B)\,\mathrm dx.
  \label{eq:exponential-barrier-law}
\end{equation}
This is a probability law whose path marginal is \(\widehat{\PP}_T\).  Put
\begin{equation}
  F_T(x):=\PP(M_T\leq x),
  \qquad
  \beta_T(\mathrm dx)
  :=\frac{e^{-x}F_T(x)}{\mathcal Z(T)}\,\mathrm dx,
  \qquad x\geq0.
  \label{eq:exponential-barrier-marginal}
\end{equation}
Then \(\beta_T\) is the barrier marginal.  For every \(x>0\), a regular
conditional law of the path given \(E=x\) is
\begin{equation}
  \mathbb P_{T,x}:=\mathcal L(B\mid M_T\leq x).
  \label{eq:hard-wall-law}
\end{equation}
Under \(\mathbb Q_T\), the overshoot \(E-M_T\) is
\(\operatorname{Exp}(1)\) and is independent of \(B\).  Moreover, for
\(A\geq0\),
\begin{equation}
  \beta_T([A,\infty))=\mathbb Q_T(E\geq A)
  \leq2p_*^{-m_H}e^{-A/2}.
  \label{eq:exponential-barrier-tail}
\end{equation}
\end{lemma}

\begin{proof}
Since \(B_0=0\), one has \(M_T\geq0\), and independence and Tonelli give
\begin{equation*}
  \PP(M_T\leq E)
  =\EE\left[\int_{M_T}^\infty e^{-x}\,\mathrm dx\right]
  =\EE e^{-M_T}
  =\mathcal Z(T)>0.
\end{equation*}
Thus \eqref{eq:exponential-barrier-law} has total mass one.  Integrating it
first in \(x\) gives the path marginal
\(\mathcal Z(T)^{-1}e^{-M_T}\,\mathrm d\PP=\mathrm d\widehat{\PP}_T\),
whereas integrating first in the path gives \(\beta_T\).

For \(x>0\), the Gaussian support theorem used in
Lemma~\ref{lem:fbm-compact-positivity-support} gives
\(F_T(x)\geq\PP(\|B\|_{C([0,T])}<x)>0\).  For every Borel path event \(C\),
the map \(x\mapsto\PP(C\cap\{M_T\leq x\})/F_T(x)\) is Borel on
\((0,\infty)\); division by the barrier marginal therefore proves
\eqref{eq:hard-wall-law}.  Its value at zero is immaterial because
\(\beta_T(\{0\})=0\).  Division instead by the path marginal gives
\begin{equation*}
  \mathbb Q_T(E\in\mathrm dx\mid B)
  =e^{-(x-M_T)}\mathds1_{\{x\geq M_T\}}\,\mathrm dx,
\end{equation*}
which proves the overshoot assertion.

Finally, for \(A\geq0\),
\begin{equation}
\begin{split}
  \beta_T([A,\infty))
  &=
  \frac{\int_A^\infty e^{-x}F_T(x)\,\mathrm dx}{\mathcal Z(T)}\\
  &\leq
  \frac{e^{-A/2}}{\mathcal Z(T)}
  \EE\left[\int_{M_T}^\infty e^{-x/2}\,\mathrm dx\right]\\
  &=2e^{-A/2}
  \frac{\mathcal Z(2^{-1/H}T)}{\mathcal Z(T)}
  \leq2p_*^{-m_H}e^{-A/2}.
\end{split}
\end{equation}
Here the inequality uses \(e^{-x/2}\leq e^{-A/2}\) for \(x\geq A\) and
Tonelli.  The first equality on the last line evaluates the exponential tail
and uses self-similarity; the final inequality is
\eqref{eq:normalizer-fixed-factor}.
\end{proof}

We repeatedly use the following form of the Gaussian supremum
anti-concentration theorem of
\textcite[Theorem~2.1]{chernozhukovChetverikovKato2014anticoncentration}.  If
\((X_u)_{u\in K}\) is a separable centred Gaussian process with unit
variances and an almost surely finite supremum, then
\begin{equation}
  \sup_{x\in\mathbb R}
  \PP\left(\left|\sup_{u\in K}X_u-x\right|\leq\varepsilon\right)
  \leq4\varepsilon\left(\EE\sup_{u\in K}X_u+1\right).
  \label{eq:gaussian-supremum-anticoncentration}
\end{equation}

\begin{proposition}
\label{prop:uniform-hard-wall-localization}
There is \(C_H<\infty\) such that, for \(T>R\geq1\) and \(x>0\),
\begin{equation}
  \mathbb P_{T,x}(\sigma_T>R)
  \leq C_HxR^{-H}.
  \label{eq:uniform-hard-wall-localization}
\end{equation}
\end{proposition}

\begin{proof}
For \(1\leq L\leq2\), put
\begin{equation}
  S_L:=\sup_{1\leq u\leq L}B_u,
  \qquad
  X_u:=B_u/u^H.
\end{equation}
The process \(X\) is continuous and hence separable, centred, and of unit
variance.  Since \(u^H\geq1\), continuity and attainment of the two suprema
give, for \(\varepsilon>0\),
\begin{equation*}
  \{0\leq S_L\leq\varepsilon\}
  \subseteq
  \left\{0\leq\sup_{1\leq u\leq L}X_u\leq\varepsilon\right\}.
\end{equation*}
Moreover, Fernique's theorem
\parencite[Theorem~2.8.5]{bogachev1998gaussianMeasures} gives
\begin{equation*}
  \sup_{1\leq L\leq2}
  \EE\sup_{1\leq u\leq L}X_u
  \leq\EE\sup_{1\leq u\leq2}|B_u|<\infty.
\end{equation*}
The Gaussian anti-concentration estimate
\eqref{eq:gaussian-supremum-anticoncentration} therefore gives
\begin{equation}
  \PP(0\leq S_L\leq\varepsilon)
  \leq C_H\varepsilon,
  \qquad 0<\varepsilon\leq1.
  \label{eq:compact-block-anticoncentration}
\end{equation}
Moreover,
\begin{equation}
  \inf_{1\leq L\leq2}\PP(S_L\leq0)
  \geq p_*.
  \label{eq:compact-block-baseline}
\end{equation}

We next compare hard-wall events first on a finite grid.  Fix \(r>0\) and
\(1\leq L\leq2\) such that the block \(I:=r[1,L]\) is contained in
\((0,T]\).  Let \(F\subset(0,T]\) be finite, with
\(F_I:=F\cap I\neq\varnothing\), and put \(F_o:=F\setminus F_I\).  In the
shifted reflected coordinates \(G_t=x-B_t\), define
\begin{equation}
  A_I:=\{G_t\geq0:t\in F_I\},
  \qquad
  A_o:=\{G_t\geq0:t\in F_o\}.
\end{equation}
The finite Gaussian vector \(G_F\) is nondegenerate and has an everywhere
positive density, so \(\PP(A_I)>0\) and
\(\PP(A_I\cap A_o)>0\).  Lemma~\ref{lem:finite-wall-mtp2-closure} makes
\(\mathcal L(G_F\mid A_I)\) associated.  Under this law \(A_o\) is
increasing, while
\begin{equation}
  D_I:=\{\min_{t\in F_I}G_t\leq x\}
\end{equation}
is decreasing.  Association applied to the two increasing events \(D_I^c\)
and \(A_o\), followed by division by the positive probability of \(A_o\),
yields
\begin{equation}
  \PP(D_I\mid A_I\cap A_o)
  \leq\PP(D_I\mid A_I).
  \label{eq:block-wall-comparison-grid}
\end{equation}
In the original coordinates, \(A_I=\{\max_{t\in F_I}B_t\leq x\}\) and
\(D_I=\{\max_{t\in F_I}B_t\geq0\}\).

Choose increasing finite grids \(F_n\subset(0,T]\) whose union is dense in
\([0,T]\) and such that \(\bigcup_n(F_n\cap I)\) is dense in \(I\).
Apply \eqref{eq:block-wall-comparison-grid} with \(F=F_n\).  Continuity makes
the finite upper-wall events decrease to their continuum counterparts, and
the grid suprema on \(I\) increase to \(\sup_{t\in I}B_t\).  After scaling
\(I=r[1,L]\), \eqref{eq:compact-block-anticoncentration} gives
\(\PP(\sup_{t\in I}B_t=0)=0\).  This is the only boundary event needed in
the limit.  Bounded convergence, the positivity of
\(\PP(M_T\leq x)\) established in
Lemma~\ref{lem:exponential-barrier-disintegration}, and
\(\PP(\sup_{t\in I}B_t\leq x)\geq p_*\) now give
\begin{equation}
  \PP\left(\sup_{t\in I}B_t\geq0\,\middle|\,M_T\leq x\right)
  \leq
  \frac{\PP(0\leq\sup_{t\in I}B_t\leq x)}
       {\PP(\sup_{t\in I}B_t\leq x)}.
  \label{eq:block-wall-comparison-continuum}
\end{equation}

For the dyadic cover, put
\begin{equation}
\begin{aligned}
  J_T&:=\{j\in\mathbb Z_{\geq0}:r_j:=2^jR<T\},\\
  I_j&=[r_j,\min(2r_j,T)],
  &\qquad
  L_j&:=\frac{\min(2r_j,T)}{r_j}\in(1,2].
\end{aligned}
\end{equation}
Then \(J_T\) is finite, \(I_j=r_j[1,L_j]\), and the blocks cover \([R,T]\),
including the terminal truncated block.  Scaling
Lemma~\ref{lem:unique-interior-maximizer} to \([0,T]\) shows that the unique
maximiser is interior and, since \(B_0=0\), that \(M_T>0\) almost surely.
The same holds under \(\mathbb P_{T,x}\), which is absolutely continuous
with respect to \(\PP\).  Hence \(\{\sigma_T>R\}\) is contained, up to a
null set, in the union over \(j\in J_T\) of the events
\(\{\sup_{t\in I_j}B_t\geq0\}\).

Put \(\varepsilon_j:=xr_j^{-H}\).  Self-similarity,
\eqref{eq:compact-block-anticoncentration}--\eqref{eq:compact-block-baseline},
and \eqref{eq:block-wall-comparison-continuum} give, when
\(\varepsilon_j\leq1\),
\begin{equation}
  \PP\left(\sup_{t\in I_j}B_t\geq0\,\middle|\,M_T\leq x\right)
  \leq C_Hxr_j^{-H};
\end{equation}
the denominator in \eqref{eq:block-wall-comparison-continuum} is at least
\(p_*\), which has been absorbed into \(C_H\).  If
\(\varepsilon_j>1\), the same bound follows after enlarging \(C_H\), since
the left-hand side is at most one.  The union bound now gives
\begin{equation*}
\begin{split}
  \mathbb P_{T,x}(\sigma_T>R)
  &\leq C_Hx\sum_{j\in J_T}r_j^{-H}\\
  &\leq C_HxR^{-H}\sum_{j=0}^{\infty}2^{-jH}
   =\frac{C_H}{1-2^{-H}}xR^{-H}.
\end{split}
\end{equation*}
Absorbing the last fixed factor proves
\eqref{eq:uniform-hard-wall-localization}.
\end{proof}

\begin{theorem}
\label{thm:tilted-scalar-compact-containment}
For \(T>R\geq e\),
\begin{equation}
  \widehat{\PP}_T(\sigma_T>R)
  \leq C_H(1+\log R)R^{-H}.
  \label{eq:tilted-left-horizon-tail}
\end{equation}
Consequently, under \(\widehat{\PP}_a\), the families \((U_a)_{a>0}\) and
\((V_a)_{a>0}\) are tight, and, as \(a\to\infty\),
\begin{equation}
  a-U_a\longrightarrow\infty
  \quad\text{in probability}.
  \label{eq:tilted-right-horizon-divergence}
\end{equation}
\end{theorem}

\begin{proof}
Disintegrate by Lemma~\ref{lem:exponential-barrier-disintegration}.  For
\(A>0\), the barrier marginal has no atom at zero, and
Proposition~\ref{prop:uniform-hard-wall-localization} gives
\begin{equation}
\begin{split}
  \widehat{\PP}_T(\sigma_T>R)
  &\leq\mathbb Q_T(E>A)
    +\int_{(0,A]}\mathbb P_{T,x}(\sigma_T>R)
       \,\beta_T(\mathrm dx)\\
  &\leq2p_*^{-m_H}e^{-A/2}
    +C_HR^{-H}\int_{(0,A]}x\,\beta_T(\mathrm dx)\\
  &\leq2p_*^{-m_H}e^{-A/2}+C_HAR^{-H}.
\end{split}
  \label{eq:soft-location-from-hard-wall}
\end{equation}
Since \(R\geq e\), the choice \(A=2H\log R\) is positive and satisfies
\(e^{-A/2}=R^{-H}\).  Thus \eqref{eq:soft-location-from-hard-wall} is at
most
\begin{equation*}
  \bigl(2p_*^{-m_H}+2HC_H\log R\bigr)R^{-H},
\end{equation*}
which proves \eqref{eq:tilted-left-horizon-tail}.  This is exactly the
normalised left-horizon estimate; no unnormalised estimate of order
\(T^{-(1-H)}\) is asserted or needed.

By \eqref{eq:long-horizon-scaling}, if \(a>R\geq e\), then
\begin{equation*}
  \widehat{\PP}_a(U_a>R)
  =\widehat{\PP}_a(\sigma_a>R)
  \leq C_H(1+\log R)R^{-H}.
\end{equation*}
If \(a\leq R\), then \(U_a\leq a\leq R\).  Hence the preceding bound holds
after taking the supremum over every \(a>0\), and its right-hand side tends
to zero as \(R\to\infty\).  This proves tightness of \((U_a)_{a>0}\).
Tightness of \((V_a)_{a>0}\) is
Corollary~\ref{cor:tilted-height-tightness}.

Finally, fix \(0\leq K<\infty\) and \(\varepsilon>0\).  Tightness gives
\(M<\infty\) such that
\(\sup_{a>0}\widehat{\PP}_a(U_a>M)<\varepsilon\).  For
\(a>M+K\),
\begin{equation*}
  \{a-U_a\leq K\}\subseteq\{U_a\geq a-K>M\}.
\end{equation*}
Therefore
\(\limsup_{a\to\infty}\widehat{\PP}_a(a-U_a\leq K)\leq\varepsilon\).
Letting \(\varepsilon\downarrow0\), and then allowing arbitrary fixed
\(K\), proves \eqref{eq:tilted-right-horizon-divergence}.
\end{proof}

\subsection{The tilted finite-left-horizon mixture}
\label{subsec:rough-boundary-entrance-mixture}

We first isolate the zero-germ estimate that controls a left wall collapsing
into the deterministic pin.
\label{subsubsec:zero-germ}

Let \(P\subset(0,\infty)\cap\mathbb D\) and
\(N\subset[-\delta,0)\cap\mathbb D\) be finite, and write
\begin{equation}
  A_P:=\{G_s\geq0:s\in P\},
  \qquad
  A_N:=\{G_s\geq0:s\in N\}.
  \label{eq:zero-germ-wall-events}
\end{equation}

\begin{lemma}
\label{lem:zero-germ-scalar-harnack}
Fix \(t\in(0,\infty)\cap\mathbb D\).  There is \(\delta_t>0\) such that the
following holds.  If \(t\in P\) and \(0<\delta\leq\delta_t\), define
\begin{equation}
  w_{P,N,t}(x)
  :=\PP(A_N\mid A_{P\setminus\{t\}},G_t=x),
  \qquad x\geq0,
  \label{eq:zero-germ-continuation-weight}
\end{equation}
where conditioning at the prescribed value is the normalised Gaussian
density section.  Then, for every \(0\leq A<\infty\),
\begin{equation}
  1\leq\frac{w_{P,N,t}(y)}{w_{P,N,t}(x)}
  \leq\exp\{C_{H,t,A}\delta^H\},
  \qquad0\leq x\leq y\leq A,
  \label{eq:zero-germ-harnack}
\end{equation}
uniformly in \(P,N\), and their cardinalities.
\end{lemma}

\begin{proof}
If \(N=\varnothing\), then \(w_{P,N,t}\equiv1\), so assume that \(N\) is
nonempty.  Lemma~\ref{lem:finite-wall-mtp2-closure} makes \(G_{P\cup N}\)
nondegenerate.  Its Gaussian density is strictly positive, as are the
normalisers obtained by fixing \(G_t=x\) and imposing the finite orthant
\(A_{P\setminus\{t\}}\).  Thus \eqref{eq:zero-germ-continuation-weight}
means the corresponding normalised density section for every \(x\geq0\),
and
\begin{equation*}
  0<w_{P,N,t}(x)\leq1.
\end{equation*}

Condition on \(G_t=x\) and write
\begin{equation}
  G_s=m_t(s)x+\xi_s,
  \qquad
  m_t(s):=\frac{K_H(s,t)}{t^{2H}},
  \qquad
  d_t(s)^2:=\operatorname{Var}(\xi_s).
  \label{eq:zero-germ-regression}
\end{equation}
The centred residual law \(\nu\) is independent of \(x\).  On the coordinates
\((P\setminus\{t\})\cup N\), it is the density section of \(G_{P\cup N}\) at
\(G_t=0\).  Lemma~\ref{lem:finite-wall-mtp2-closure} therefore makes this
residual vector MTP\(_2\) and associated.  Its covariance is the positive-
definite Schur complement of the \(G_t\) entry, so \(d_t(s)>0\) at every
remaining coordinate.  Put
\begin{equation}
\begin{aligned}
  P_x&:=\{\xi_s\geq-m_t(s)x:s\in P\setminus\{t\}\},\\
  N_x&:=\{\xi_s\geq-m_t(s)x:s\in N\}.
\end{aligned}
  \label{eq:zero-germ-moving-walls}
\end{equation}
Then \(w_{P,N,t}(x)=\nu(N_x\mid P_x)\).

We prove the required monotone regression directly.  Restrict the joint law
to \(A_{P\setminus\{t\}}\), marginalise to \((G_t,G_N)\), and denote its
strictly positive MTP\(_2\) density by \(f(a,z)\).  If \(x\leq y\) and
\(z\leq z'\) coordinatewise, the lattice inequality gives
\begin{equation*}
  f(x,z)f(y,z')\geq f(x,z')f(y,z).
\end{equation*}
Hence \(r_{x,y}(z):=f(y,z)/f(x,z)\) is coordinatewise nondecreasing.
The density section at \(G_t=x\) is again MTP\(_2\) and associated.
Writing \(\EE_x\) for expectation under that section, apply association to
the increasing orthant indicator
\(\mathbf1_{A_N}\) and \(r_{x,y}\wedge M\), then letting \(M\to\infty\),
gives
\begin{equation*}
  \frac{\EE_x[\mathbf1_{A_N}r_{x,y}]}
       {\EE_x[r_{x,y}]}
  \geq\EE_x\mathbf1_{A_N}.
\end{equation*}
The two sides are \(w_{P,N,t}(y)\) and \(w_{P,N,t}(x)\), respectively.
This proves the first inequality in \eqref{eq:zero-germ-harnack} for every
prescribed \(0\leq x\leq y\).

Take \(\delta_t\leq t/2\), put \(\beta=2H\), and, for
\(-\delta\leq s<0\), write \(v=|s|/t\leq1/2\).  Then
\[
  0\leq m_t(s)
  =\frac12\bigl(v^\beta+1-(1+v)^\beta\bigr)
  \leq\frac12v^\beta,
\]
by subadditivity and monotonicity of \(x\mapsto x^\beta\).  Moreover,
\[
  d_t(s)^2=t^\beta\bigl(v^\beta-m_t(s)^2\bigr)
  \geq\frac34t^\beta v^\beta.
\]
Consequently,
\begin{equation}
  0\leq\frac{m_t(s)}{d_t(s)}
  \leq C_H\frac{|s|^H}{t^{2H}},
  \qquad -\delta\leq s<0.
  \label{eq:zero-germ-standardized-shift}
\end{equation}
Partition \(N\) into its nonempty multiplicative shells, where
\(\rho_j:=2^jt\):
\begin{equation}
  S_j:=\{s\in N:\rho_j\leq|s|<2\rho_j\},
  \qquad
  \eta_j:=C_HA\rho_j^Ht^{-2H}.
  \label{eq:zero-germ-shells}
\end{equation}
For \(Z_s:=\xi_s/d_t(s)\) and \(M_j:=\min_{s\in S_j}Z_s\), apply the uniform
shell estimate of \textcite[Lemma~4.9]{moench2026horizonProblem}.  The
interval stipulated there has relative radius smaller than \(1/2\), so it
lies in \((0,\infty)\), whereas \(S_j\subset(-\infty,0)\); hence every
\(S_j\) lies in the required shell domain.  The lemma gives
constants independent of the shell and its cardinality such that
\begin{equation}
  \PP(M_j\geq0)\geq p_t>0,
  \qquad
  \EE\sup_{s\in S_j}(-Z_s)\leq B_t.
  \label{eq:zero-germ-shell-baseline}
\end{equation}
The centred anti-concentration estimate
\eqref{eq:gaussian-supremum-anticoncentration} therefore gives
\begin{equation}
  \PP(-\eta_j\leq M_j<0)\leq C_t\eta_j.
  \label{eq:zero-germ-shell-anticoncentration}
\end{equation}
Only centred variables occur here: for \(0\leq x\leq y\leq A\), every
normalised threshold in \eqref{eq:zero-germ-moving-walls} lies between zero
and \(-\eta_j\).  If \(N_x^j,N_y^j\) are the corresponding hard-wall events on
\(S_j\), then
\begin{equation}
  \frac{\nu(N_y^j)}{\nu(N_x^j)}
  \leq\frac{\PP(M_j\geq-\eta_j)}{\PP(M_j\geq0)}
  =1+\frac{\PP(-\eta_j\leq M_j<0)}{\PP(M_j\geq0)}
  \leq e^{C_t\eta_j}.
  \label{eq:zero-germ-unconditional-shell-ratio}
\end{equation}

The same ratio holds in the presence of any fixed increasing event \(B\) on
the other residual coordinates.  Indeed, \(N_x^j\subset N_y^j\), and
association under \(\nu(\,\cdot\mid N_y^j)\) yields
\begin{equation}
  \frac{\nu(N_y^j\cap B)}{\nu(N_x^j\cap B)}
  \leq\frac{\nu(N_y^j)}{\nu(N_x^j)}.
  \label{eq:zero-germ-conditioned-shell-ratio}
\end{equation}
All denominators here are positive because every finite residual vector is
nondegenerate and has an everywhere positive Gaussian density.
Telescope \eqref{eq:zero-germ-conditioned-shell-ratio} over all negative
shells, keeping \(P_x\) and every other threshold fixed.  There are only
finitely many nonempty shells, and since
\begin{equation}
  \sum_{\rho_j\leq\delta}\rho_j^H\leq C_H\delta^H,
  \label{eq:zero-germ-shell-sum}
\end{equation}
we obtain
\begin{equation}
  \frac{\nu(N_y\mid P_x)}{\nu(N_x\mid P_x)}
  \leq e^{C_{H,t,A}\delta^H}.
  \label{eq:zero-germ-fixed-positive-wall-ratio}
\end{equation}
Finally \(K_H\geq0\) in the rough range, by subadditivity of
\(x\mapsto x^{2H}\), so \(P_x\subset P_y\).  Under the associated law
\(\nu(\,\cdot\mid P_y)\), imposing the additional increasing event \(P_x\)
raises the probability of \(N_y\).  Hence
\begin{equation}
  w_{P,N,t}(y)=\nu(N_y\mid P_y)
  \leq\nu(N_y\mid P_x).
\end{equation}
Together with
\eqref{eq:zero-germ-fixed-positive-wall-ratio}, this proves the upper bound
in \eqref{eq:zero-germ-harnack}.
\end{proof}

\begin{lemma}
\label{lem:zero-germ-scalar-invisibility}
For every fixed \(t\in(0,\infty)\cap\mathbb D\),
\begin{equation}
  \sup_{P,N}d_{\mathrm{TV}}\left(
    \mathcal L(G_t\mid A_P\cap A_N),
    \mathcal L(G_t\mid A_P)
  \right)\longrightarrow0
  \label{eq:zero-germ-scalar-invisibility}
\end{equation}
as \(\delta\downarrow0\), where the supremum is over finite \(P\) containing
\(t\) and finite \(N\subset[-\delta,0)\cap\mathbb D\).
\end{lemma}

\begin{proof}
If \(N=\varnothing\), then the two laws agree, so assume that \(N\) is
nonempty.  Put
\(\mu_0=\mathcal L(G_t\mid A_P)\) and
\(\mu_1=\mathcal L(G_t\mid A_P\cap A_N)\).  Disintegration over \(G_t\)
under the first law and the definition
\eqref{eq:zero-germ-continuation-weight} give
\begin{equation*}
  z_{P,N,t}
  :=\int w_{P,N,t}\,\mathrm d\mu_0
  =\PP(A_N\mid A_P)>0.
\end{equation*}
Positivity follows also directly from the everywhere positive Gaussian
density of \(G_{P\cup N}\).  Bayes' formula therefore gives
\begin{equation}
  \frac{\mathrm d\mu_1}{\mathrm d\mu_0}(x)
  =\frac{w_{P,N,t}(x)}{z_{P,N,t}},
  \qquad x\geq0.
  \label{eq:zero-germ-scalar-density}
\end{equation}

The two measures are the one-coordinate members of the full hard-wall
family in Theorem~\ref{thm:finite-dimensional-entrance-law}, with constraint
sets \(P\) and \(P\cup N\), respectively.  Fix
\(0<\varepsilon<1\).  Full-family tightness supplies \(A<\infty\), independent
of \(P,N,\delta\), such that
\begin{equation}
  \mu_0([0,A]^c)+\mu_1([0,A]^c)\leq\varepsilon.
  \label{eq:zero-germ-compact-truncation}
\end{equation}
Set \(C_A=[0,A]\), \(p_i=\mu_i(C_A)>0\), and
\(\bar\mu_i=\mu_i(\,\cdot\mid C_A)\).  For
\(0<\delta\leq\delta_t\), put
\begin{equation*}
  R_\delta:=e^{C_{H,t,A}\delta^H}.
\end{equation*}
Lemma~\ref{lem:zero-germ-scalar-harnack} gives
\begin{equation*}
  w_{P,N,t}(0)
  \leq w_{P,N,t}(x)
  \leq R_\delta w_{P,N,t}(0),
  \qquad x\in C_A.
\end{equation*}
The global normaliser in \eqref{eq:zero-germ-scalar-density} cancels after
the two restrictions are normalised:
\begin{equation*}
  \frac{\mathrm d\bar\mu_1}{\mathrm d\bar\mu_0}(x)
  =\frac{w_{P,N,t}(x)}
         {\int w_{P,N,t}\,\mathrm d\bar\mu_0}
  \in[R_\delta^{-1},R_\delta].
\end{equation*}
Consequently
\begin{equation*}
  d_{\mathrm{TV}}(\bar\mu_1,\bar\mu_0)
  \leq R_\delta-1.
\end{equation*}
Since
\(d_{\mathrm{TV}}(\mu_i,\bar\mu_i)=\mu_i(C_A^c)\), the triangle inequality
and \eqref{eq:zero-germ-compact-truncation} yield the uniform bound
\begin{equation*}
  d_{\mathrm{TV}}(\mu_1,\mu_0)
  \leq\varepsilon+e^{C_{H,t,A}\delta^H}-1.
\end{equation*}
For the fixed \(A\), let \(\delta\downarrow0\), and then let
\(\varepsilon\downarrow0\).  This proves
\eqref{eq:zero-germ-scalar-invisibility}.  The possibly small quantity
\(\PP(A_N\mid A_P)=z_{P,N,t}\) cancels; no uniform lower bound for it is
used.
\end{proof}

\begin{theorem}
\label{thm:zero-germ}
For every finite \(T\subset(0,\infty)\cap\mathbb D\), there is a function
\(\omega_T(\delta)\downarrow0\) such that, uniformly over finite positive
constraint sets \(P\supset T\) and finite
\(N\subset[-\delta,0)\cap\mathbb D\),
\begin{equation}
  d_{\mathrm{BL}}\left(
    \mathcal L(G_T\mid A_P\cap A_N),
    \mathcal L(G_T\mid A_P)
  \right)
  \leq\omega_T(\delta).
  \label{eq:zero-germ-vector-invisibility}
\end{equation}
Here \(d_{\mathrm{BL}}\) is formed from the bounded metric
\begin{equation*}
  d_T(x,y):=1\wedge\sum_{t\in T}|x_t-y_t|
\end{equation*}
on \(\mathbb R^T\).
\end{theorem}

\begin{proof}
If \(T=\varnothing\), the claim holds with \(\omega_T\equiv0\), so assume
that \(T\) is nonempty.  Since \(P\subset P\cup N\),
Proposition~\ref{prop:wall-monotonicity} gives
\begin{equation*}
  \mathcal L(G_T\mid A_P)
  \leq_{\mathrm{st}}
  \mathcal L(G_T\mid A_P\cap A_N).
\end{equation*}
The coordinatewise order is closed on the Polish space \(\mathbb R^T\).
The monotone-coupling theorem of \textcite{strassen1965marginals} therefore
gives a coupling \((X,Y)\) with
\begin{equation*}
  X\sim\mathcal L(G_T\mid A_P),
  \qquad
  Y\sim\mathcal L(G_T\mid A_P\cap A_N),
\end{equation*}
such that
\begin{equation}
  X_t\leq Y_t,
  \qquad t\in T,
  \quad\text{almost surely}.
  \label{eq:zero-germ-monotone-coupling}
\end{equation}

We make the uniform-integrability step explicit.  Fix \(t\in T\) and
\(\lambda>0\).  For every finite dyadic wall \(F\) containing \(t\),
\eqref{eq:pin-mean-bound} and
\eqref{eq:harge-mgf-bound}, applied to the \(t\)-coordinate of the convex
hard-wall restriction, give
\begin{equation}
\begin{split}
  \EE(e^{\lambda G_t}\mid A_F)
  &\leq
  \exp\left\{\lambda C_Ht^H
              +\frac{\lambda^2t^{2H}}2\right\}
  =:M_{\lambda,t}<\infty.
\end{split}
  \label{eq:zero-germ-coordinate-exponential-moment}
\end{equation}
The constant is independent of \(F\), and therefore applies to the
\(t\)-marginals of both \(X\) and \(Y\).

Let
\begin{equation*}
  \eta_t(\delta):=
  \sup_{P,N}d_{\mathrm{TV}}\left(
    \mathcal L(G_t\mid A_P\cap A_N),
    \mathcal L(G_t\mid A_P)
  \right),
\end{equation*}
with the same admissible walls as in the theorem.  Lemma
\ref{lem:zero-germ-scalar-invisibility} gives
\(\eta_t(\delta)\to0\).  Since \(X_t,Y_t\geq0\), truncation at any
\(R<\infty\), followed by
\eqref{eq:zero-germ-coordinate-exponential-moment}, yields
\begin{equation}
\begin{split}
  0\leq\EE Y_t-\EE X_t
  &\leq 2R\eta_t(\delta)
    +2M_{\lambda,t}\sup_{x>R}xe^{-\lambda x}.
\end{split}
  \label{eq:zero-germ-coordinate-mean-truncation}
\end{equation}
First let \(\delta\downarrow0\) with \(R\) fixed, and then let
\(R\to\infty\).  Thus, with
\begin{equation*}
  \rho_t(\delta):=
  \sup_{P,N}\bigl(
    \EE(G_t\mid A_P\cap A_N)-\EE(G_t\mid A_P)
  \bigr),
\end{equation*}
stochastic order makes every term nonnegative, and the admissible wall
classes are nested in \(\delta\).  Hence
\(\rho_t(\delta)\downarrow0\), and in particular
\begin{equation}
  \sup_{P,N}\bigl(\EE Y_t-\EE X_t\bigr)\longrightarrow0,
  \qquad t\in T.
  \label{eq:zero-germ-coordinate-mean-gap}
\end{equation}
By \eqref{eq:zero-germ-monotone-coupling},
\begin{equation}
  \EE\sum_{t\in T}|Y_t-X_t|
  =\sum_{t\in T}(\EE Y_t-\EE X_t)
  \leq\sum_{t\in T}\rho_t(\delta).
  \label{eq:zero-germ-vector-coupling-cost}
\end{equation}
The bounded-Lipschitz distance formed from \(d_T\) is bounded by
\(\EE d_T(X,Y)\), and hence by the left side of
\eqref{eq:zero-germ-vector-coupling-cost}.  Therefore
\eqref{eq:zero-germ-vector-invisibility} holds with
\begin{equation*}
  \omega_T(\delta):=\sum_{t\in T}\rho_t(\delta)\downarrow0.
\end{equation*}
\end{proof}

We next construct the normalised finite-left-horizon components.  Recall that
\(G=-W\).  For \(u>0\), put
\begin{equation}
  \mathbb D_u:=\{t\in\mathbb D:t>-u\}.
  \label{eq:finite-left-directed-set}
\end{equation}
For \(u,v>0\) and \(f\in C(\mathbb R)\), define the constant extension from
the moving left endpoint by
\begin{equation}
  (\mathsf E_{v\to u}f)(t):=f(t\vee(-v)),
  \qquad t\in[-u,\infty).
  \label{eq:moving-left-extension}
\end{equation}
If \(v\geq u\), this is just restriction to \([-u,\infty)\).  For
\(f\in C([-u,\infty))\), also put
\begin{equation}
  \mathscr A_u(f):=
  \left(u,-f(-u),
    \bigl(-f(-(s\wedge u))\bigr)_{s\geq0},
    \bigl(-f(s)\bigr)_{s\geq0}\right)\in\mathcal Y.
  \label{eq:finite-left-arm-map}
\end{equation}
Thus the left arm is extended constantly after its lifetime \(u\).

\begin{proposition}
\label{prop:finite-left-entrance-law}
For every \(u>0\) there is a unique probability law \(Q^u\) on
\(C([-u,\infty))\) with the following property.  Suppose
\begin{equation}
\begin{gathered}
  h_n=2^{-r_n}\downarrow0,\qquad
  u_n\in h_n\mathbb N,\qquad u_n\to u,\\
  L_n\in h_n\mathbb N,\qquad L_n\to\infty,
\end{gathered}
  \label{eq:finite-left-entrance-parameters}
\end{equation}
and let
\begin{equation}
  Q_n:=\mathcal L\bigl(
    W\,\big|\,W_t\leq0,
    \ t\in h_n\mathbb Z\cap[-u_n,L_n]
  \bigr).
  \label{eq:finite-left-wall-law}
\end{equation}
Then
\begin{equation}
  (\mathsf E_{u_n\to u})_*Q_n\Longrightarrow Q^u
  \quad\text{in }C_{\mathrm{loc}}([-u,\infty)).
  \label{eq:finite-left-moving-limit}
\end{equation}
Moreover,
\begin{equation}
  Q^u(W_t\leq0\text{ for every }t\geq-u)=1,
  \qquad Q^u(W_{-u}<0)=1.
  \label{eq:finite-left-wall-and-repulsion}
\end{equation}
The map
\begin{equation}
  u\longmapsto(\mathscr A_u)_*Q^u
  \quad\text{from }(0,\infty)\text{ to }\mathcal P(\mathcal Y)
  \label{eq:positive-entrance-arm-continuity}
\end{equation}
is weakly continuous.

There is also a unique directed one-sided hard-wall law \(Q^0\) on
\(C([0,\infty))\), obtained by using only positive dyadic constraint sites.
It satisfies
\begin{equation}
  Q^0(W_0=0,\ W_t\leq0\text{ for every }t\geq0)=1.
  \label{eq:one-sided-wall-support}
\end{equation}
\end{proposition}

\begin{proof}
Fix \(u>0\) and a finite nonempty set \(T\subset\mathbb D_u\).  Let
\begin{equation*}
  \mathcal I_{u,T}:=
  \{F\subset\mathbb D_u:F\text{ finite and }T\subseteq F\}.
\end{equation*}
The family \((\nu_F^T)_{F\in\mathcal I_{u,T}}\) is a subfamily of the full
tight family in Theorem~\ref{thm:finite-dimensional-entrance-law}.  For every
bounded continuous increasing \(\varphi\), the net
\(\int\varphi\,\mathrm d\nu_F^T\) is increasing by
Proposition~\ref{prop:wall-monotonicity}.  The increasing determining-class
argument in Lemma~\ref{lem:cofinal-tightness} therefore gives a unique limit
\(\nu_u^T\).  Projection along a common cofinal tail proves that the family
\((\nu_u^T)_T\) is projectively consistent.  In particular, every sequence
of finite subsets of \(\mathbb D_u\) which eventually contains each such
finite subset has these same limiting marginals.

For a finite \(F\subset\mathbb D_u\), write
\(\mathbf Q_F:=\mathcal L(W\mid W_t\leq0,\ t\in F)\).  The proof of
\eqref{eq:centered-modulus} is uniform over every positive-probability finite
convex conditioning, and
\eqref{eq:conditional-mean-modulus}--\eqref{eq:pin-mean-bound} are uniform
over every finite dyadic wall.  Since \(W_0=0\), these estimates give
tightness of \((\mathbf Q_F)\) on every \(C([-u,R])\), \(R>0\).  Take the
nested cofinal sets \(E_N\cap\mathbb D_u\).  An exhaustion in \(R\) gives
subsequential limits in \(C_{\mathrm{loc}}([-u,\infty))\), and every such
limit has the marginals \((-\operatorname{id})_*\nu_u^T\) at finite
\(T\subset\mathbb D_u\).  Continuous paths are determined by their values on
the dense set \(\mathbb D_u\).  Hence the limit is unique; call it \(Q^u\).
The same tightness and marginal argument shows convergence to \(Q^u\) along
every cofinal sequence in \(\mathbb D_u\).

For each \(t\in\mathbb D_u\), the closed constraint \(W_t\leq0\) is present
eventually.  A countable intersection and continuity therefore give
\begin{equation*}
  Q^u(W_t\leq0\text{ for every }t>-u)=1.
\end{equation*}
Right continuity at \(-u\) includes the endpoint and proves the first
assertion of \eqref{eq:finite-left-wall-and-repulsion} without yet using
endpoint repulsion.

We record the estimate that prevents degeneracy at that point.  If a finite
constraint set contains \(e<0\), then the restriction to \(\{G_e\geq0\}\) is
associated by Lemma~\ref{lem:finite-wall-mtp2-closure}.  All remaining
constraints and \(\{G_e>y\}\) are increasing events.  Association applied to
these two events, followed by complementation in the first coordinate,
implies, for \(y\geq0\),
\begin{equation}
\begin{split}
  \PP(G_e\leq y\mid G_t\geq0,\ t\in F)
  &\leq\PP(0\leq G_e\leq y\mid G_e\geq0)\\
  &\leq \sqrt{\frac2\pi}\,\frac{y}{|e|^H}.
\end{split}
  \label{eq:finite-left-endpoint-repulsion}
\end{equation}
At every continuity point of the limiting one-coordinate distribution, the
same estimate passes along a cofinal wall sequence.  Approximating an
arbitrary \(y\) from above by continuity points proves, for every negative
\(e\in\mathbb D_u\),
\begin{equation*}
  Q^u(0\leq G_e\leq y)
  \leq \sqrt{\frac2\pi}\,\frac{y}{|e|^H}.
\end{equation*}
Choose negative \(e_j\in\mathbb D_u\) with \(e_j\downarrow-u\).  Continuity
of the coordinate path and then \(j\to\infty\) and \(\eta\downarrow0\) give
\begin{equation}
  Q^u(0\leq G_{-u}\leq y)
  \leq \sqrt{\frac2\pi}\,yu^{-H}.
  \label{eq:entrance-endpoint-repulsion}
\end{equation}
Indeed,
\(\{0\leq G_{-u}\leq y\}\subseteq
\liminf_j\{0\leq G_{e_j}\leq y+\eta\}\).
Fatou's lemma for the indicators gives the asserted bound.
Taking \(y=0\) proves the strict inequality in
\eqref{eq:finite-left-wall-and-repulsion}.

We now prove the moving-boundary assertion.  Discarding finitely many terms,
define the grid-aligned reference horizon
\begin{equation*}
  v_n:=h_n\left(\left\lceil\frac{u}{h_n}\right\rceil-1\right),
  \qquad 0<v_n<u,\qquad 0<u-v_n\leq h_n,
\end{equation*}
and put
\begin{equation*}
  F_n^\circ:=
  \bigl(h_n\mathbb Z\cap[-v_n,L_n]\bigr)\setminus\{0\}.
\end{equation*}
These sets lie in \(\mathbb D_u\).  If \(t=k2^{-q}\in\mathbb D_u\), then
\(r_n\geq q\), \(-v_n<t\), and \(L_n>|t|\) eventually; hence
\(t\in F_n^\circ\) eventually.  Thus \((F_n^\circ)\) is cofinal in
\(\mathbb D_u\), and its path laws converge to \(Q^u\).  Replacing a path by
\(\mathsf E_{v_n\to u}W\) changes it on a fixed compact by at most its
oscillation over \([-u,-v_n]\), which tends to zero in probability by the
uniform centred and mean moduli.  Thus the constant-extended reference laws
also converge to \(Q^u\).

It remains to compare the given horizon \(u_n\) with \(v_n\).  Set
\begin{equation*}
  s_n:=u_n\wedge v_n,\qquad
  \ell_n:=u_n\vee v_n,\qquad
  e_n:=-s_n,\qquad
  \delta_n:=\ell_n-s_n.
\end{equation*}
Then \(s_n\to u\) and \(\delta_n\to0\).  The law with left horizon
\(\ell_n\) is the law with left horizon \(s_n\), conditioned additionally on
the new strip constraints.  Under the shorter-wall law, failure of those
constraints is contained, for every \(\varepsilon>0\), in
\begin{equation}
  \{G_{e_n}\leq2\varepsilon\}
  \cup
  \left\{
    \operatorname{osc}_{[e_n-\delta_n,e_n]}G>\varepsilon
  \right\}.
  \label{eq:moving-boundary-failure}
\end{equation}
The finite-grid estimate \eqref{eq:finite-left-endpoint-repulsion} bounds the
first probability by \(C_H\varepsilon s_n^{-H}\).  For each fixed
\(\varepsilon\), the second tends to zero by
\eqref{eq:centered-modulus} and
\eqref{eq:conditional-mean-modulus}; the constants are uniform because
\((s_n)\) remains in a compact subset of \((0,\infty)\).  Consequently the
limsup of the failure probability is at most \(C_H\varepsilon u^{-H}\), and
letting \(\varepsilon\downarrow0\) makes it vanish.

At the finite-grid level, conditioning the shorter-wall law on the new strip
gives the longer-wall law.  Their total-variation distance is therefore
bounded by this failure probability.  The two extension maps
\(\mathsf E_{u_n\to u}\) and \(\mathsf E_{v_n\to u}\) differ on a fixed
compact by at most the oscillation over the same vanishing strip.  Total
variation contracts under a common pushforward, and the uniform modulus
controls the change of extension map.  More explicitly, if \(\Phi\) is
bounded by one and one-Lipschitz for the uniform norm on \([-u,R]\), then the
difference between its expectations under the two extended laws is at most
twice the preceding total-variation distance plus
\begin{equation*}
  \EE\left[
    1\wedge\operatorname{osc}_{[-\ell_n,-s_n]}W
  \right]
\end{equation*}
under either one of the two finite-wall laws.  Both terms tend to zero.
Comparison with the reference laws, followed by an exhaustion in \(R\), now
proves \eqref{eq:finite-left-moving-limit}.

Let \(u_j\to u>0\).  For each \(j\), the convergence just proved allows a
grid law \(P_j\), with its dyadic meshes chosen strictly decreasing and at
most \(j^{-1}\), right horizon at least \(j\), and grid-aligned left horizon
\(w_j\) satisfying
\(|w_j-u_j|\leq j^{-1}\), such that
\begin{equation*}
  d_{\mathrm{BL}}\left(
    (\mathscr A_{u_j})_*(\mathsf E_{w_j\to u_j})_*P_j,
    (\mathscr A_{u_j})_*Q^{u_j}
  \right)\leq j^{-1}.
\end{equation*}
Here one may choose the term far enough along a defining sequence for
\(Q^{u_j}\) to meet all four bounds.  Since \(w_j\to u\), another application
of \eqref{eq:finite-left-moving-limit} gives
\((\mathsf E_{w_j\to u})_*P_j\Rightarrow Q^u\).  Local uniform convergence,
\(u_j\to u\), and the uniform modulus show that replacing
\(\mathscr A_{u_j}\circ\mathsf E_{w_j\to u_j}\) by
\(\mathscr A_u\circ\mathsf E_{w_j\to u}\) is asymptotically negligible.
The continuous mapping theorem and the triangle inequality prove
\eqref{eq:positive-entrance-arm-continuity}.

Finally put \(\mathbb D_+:=\mathbb D\cap(0,\infty)\).  Restricting the same
finite-dimensional directed construction to finite subsets of
\(\mathbb D_+\) again inherits tightness from the full family and retains
monotonicity.  The uniform moduli and the deterministic pin \(W_0=0\)
reconstruct a unique law \(Q^0\) on \(C([0,\infty))\), and density of the
  positive dyadics gives \eqref{eq:one-sided-wall-support}.
\end{proof}

Under \(Q^u\), define
\begin{equation}
\begin{aligned}
  V^u&:=G_{-u},&
  L^u(s)&:=G_{-s}, &&0\leq s\leq u,\\
  R^u(s)&:=G_s,&&&s\geq0,
\end{aligned}
  \label{eq:entrance-arms}
\end{equation}
and extend \(L^u\) constantly after \(u\).  For \(u=0\), set
\(V^0=0\), \(L^0\equiv0\), and use \(Q^0\) for \(R^0\).  Put
\begin{equation}
  z(u):=\EE_{Q^u}e^{-V^u}
  \label{eq:entrance-endpoint-normalizer}
\end{equation}
so in particular \(z(0)=1\), and define the endpoint-tilted kernel \(K_u\)
on \(\mathcal Y\), for bounded Borel test functions \(F\), by
\begin{equation}
  K_u(F)
  :=\frac{1}{z(u)}\EE_{Q^u}\left[
    e^{-V^u}F(u,V^u,L^u,R^u)
  \right].
  \label{eq:entrance-kernel}
\end{equation}

\begin{lemma}
\label{lem:entrance-endpoint-tilt}
For every \(0\leq R<\infty\),
\begin{equation}
  \inf_{0\leq u\leq R}z(u)
  \geq e^{-C_HR^H}>0.
  \label{eq:entrance-normalizer-lower}
\end{equation}
The map \(u\mapsto K_u\) is weakly continuous on \((0,\infty)\).
\end{lemma}

\begin{proof}
Fix \(u>0\), and choose negative points
\(e_j\in\mathbb D_u\) with \(e_j\downarrow-u\).  For every member \(F\) of a
cofinal tail in \(\mathbb D_u\) containing \(e_j\), that coordinate is
constrained and hence nonnegative under the corresponding \(G=-W\) law.
Moreover, the natural-scale conditional-mean estimate
\eqref{eq:pin-mean-bound} gives
\begin{equation*}
  0\leq\EE_{Q_F}G_{e_j}=-m_F(e_j)\leq C_H|e_j|^H.
\end{equation*}
For \(M<\infty\), the function
\(x\mapsto (x^+)\wedge M\) is bounded and continuous.  The finite-dimensional convergence used to construct \(Q^u\) therefore gives
\begin{equation*}
  \EE_{Q^u}\bigl[(G_{e_j}^+)\wedge M\bigr]
  \leq C_H|e_j|^H.
\end{equation*}
Letting \(M\to\infty\), using the wall support, and then applying Fatou's
lemma to \(G_{e_j}\to G_{-u}\) under the continuous-path law \(Q^u\), gives
\begin{equation}
  \EE_{Q^u}V^u
  \leq\liminf_{j\to\infty}\EE_{Q^u}G_{e_j}
  \leq C_Hu^H.
  \label{eq:entrance-endpoint-mean}
\end{equation}
At \(u=0\), the same bound is the identity \(V^0=0\).  Since
\(x\mapsto e^{-x}\) is convex, Jensen's inequality now gives, for
\(0\leq u\leq R\),
\begin{equation*}
  z(u)\geq\exp\{-\EE_{Q^u}V^u\}
  \geq e^{-C_Hu^H}\geq e^{-C_HR^H},
\end{equation*}
which proves \eqref{eq:entrance-normalizer-lower}.

For \(u>0\), write \(P_u=(\mathscr A_u)_*Q^u\).  By
\eqref{eq:positive-entrance-arm-continuity}, the map \(u\mapsto P_u\) is
weakly continuous.  If \(y=(a,v,\ell,r)\in\mathcal Y\), then
\(w(y):=e^{-v}\) belongs to \(C_b(\mathcal Y)\), and
\begin{equation*}
  \int_{\mathcal Y}F\,\mathrm dK_u
  =\frac{\int_{\mathcal Y}Fw\,\mathrm dP_u}
         {\int_{\mathcal Y}w\,\mathrm dP_u},
  \qquad F\in C_b(\mathcal Y).
\end{equation*}
Both integrals vary continuously with \(u\), and the denominator is positive
by \eqref{eq:entrance-normalizer-lower}.  This proves weak continuity on
  \((0,\infty)\).
\end{proof}

For a dyadic mesh \(h=2^{-r}\) and \(u,L\in h\mathbb Z_{\geq0}\), write
\begin{equation}
  Q_{h,u,L}:=\mathcal L\bigl(
    W\,\big|\,W_t\leq0,
    \ t\in h\mathbb Z\cap[-u,L]
  \bigr).
  \label{eq:finite-entrance-component}
\end{equation}
Set
\begin{equation}
\begin{split}
  z_{h,u,L}&:=\EE_{Q_{h,u,L}}e^{-G_{-u}},\\
  \frac{\mathrm d\widehat Q_{h,u,L}}{\mathrm dQ_{h,u,L}}
  &:=\frac{e^{-G_{-u}}}{z_{h,u,L}},
  \qquad
  K_{h,u,L}:=(\mathscr M_{u,L})_*\widehat Q_{h,u,L},
\end{split}
  \label{eq:finite-entrance-endpoint-tilt}
\end{equation}
where \(\mathscr M_{u,L}\) is the constant-extension map in
\eqref{eq:tilted-grid-mark-map}.  Equivalently, for every bounded Borel
function \(F\) on \(\mathcal Y\),
\begin{equation*}
  K_{h,u,L}(F)
  =\frac{\EE_{Q_{h,u,L}}\!\left[
      e^{-G_{-u}}F(\mathscr M_{u,L}(W))
    \right]}{z_{h,u,L}}.
\end{equation*}

\begin{corollary}
\label{cor:zero-germ-identification}
Let
\begin{equation*}
  h_n=2^{-r_n}\downarrow0,
  \qquad
  u_n,L_n\in h_n\mathbb Z_{\geq0},
  \qquad
  u_n\to0,
  \qquad
  L_n\to\infty.
\end{equation*}
Put
\begin{equation*}
  P_{h,u,L}:=(\mathscr M_{u,L})_*Q_{h,u,L},
  \qquad
  P_u:=(\mathscr A_u)_*Q^u\quad(u>0),
\end{equation*}
and let \(P_0\) be the law on \(\mathcal Y\) with zero lifetime, height, and
left arm, and right arm \((G_s)_{s\geq0}\) under \(Q^0\).  Thus
\(K_0=P_0\).  Then
\begin{equation}
\begin{gathered}
  \mathcal L_{Q_{h_n,u_n,L_n}}
    \bigl(W|_{[0,\infty)}\bigr)
  \Longrightarrow Q^0
  \quad\text{in }C_{\mathrm{loc}}([0,\infty)),\\
  P_{h_n,u_n,L_n}\Longrightarrow P_0,
  \qquad
  K_{h_n,u_n,L_n}\Longrightarrow K_0
  \quad\text{in }\mathcal Y.
\end{gathered}
  \label{eq:zero-germ-finite-component-limit}
\end{equation}
Consequently, the arm-valued map \(u\mapsto P_u\) and the endpoint-tilted map
\(u\mapsto K_u\) are weakly continuous on \([0,\infty)\).
\end{corollary}

\begin{proof}
Write
\begin{equation*}
  F_n^+:=h_n\mathbb Z\cap(0,L_n],
  \qquad
  F_n^-:=h_n\mathbb Z\cap[-u_n,0),
  \qquad
  Q_n:=Q_{h_n,u_n,L_n}.
\end{equation*}
Fix a finite positive dyadic set \(T\).  Eventually \(T\subset F_n^+\), and
Theorem~\ref{thm:zero-germ}, with \(\delta=u_n\), gives
\begin{equation*}
  d_{\mathrm{BL}}\left(
    \mathcal L_{Q_n}(G_T),
    \mathcal L(G_T\mid A_{F_n^+})
  \right)\longrightarrow0.
\end{equation*}
The walls \(F_n^+\) are cofinal in the positive dyadics, so the second law
converges to the \(T\)-marginal of \(Q^0\) by
Proposition~\ref{prop:finite-left-entrance-law}.

We now upgrade these marginals to paths.  The full-family form of
\eqref{eq:centered-modulus}, together with
\eqref{eq:conditional-mean-modulus} and the deterministic pin at zero, makes
the restrictions of \((Q_n)\) tight on every \(C([0,R])\).  Every
subsequential limit has the \(Q^0\) marginals at every finite positive dyadic
set.  Since continuous paths are determined by their values on that dense
set and at the pin, the limit is \(Q^0|_{[0,R]}\).  Exhaustion in \(R\)
proves the first convergence in
\eqref{eq:zero-germ-finite-component-limit}.

We next control the collapsing marks before tilting.  Put
\begin{equation*}
  \mathcal F_n:=
  \bigl(h_n\mathbb Z\cap[-u_n,L_n]\bigr)\setminus\{0\},
  \qquad
  m_n:=m_{\mathcal F_n},
  \qquad
  Z_n:=W-m_n.
\end{equation*}
If \(u_n>0\), the endpoint is constrained and
\eqref{eq:pin-mean-bound} gives
\begin{equation}
  0\leq\EE_{Q_n}G_{-u_n}
  =-m_n(-u_n)\leq C_Hu_n^H;
  \label{eq:zero-germ-finite-endpoint-collapse}
\end{equation}
the same display is the identity zero when \(u_n=0\).  Moreover,
\begin{equation}
\begin{split}
  \sup_{s\geq0}|G_{-(s\wedge u_n)}|
  &=\sup_{t\in[-u_n,0]}|G_t|\\
  &\leq
  \sup_{t\in[-u_n,0]}|Z_n(t)-Z_n(0)|+C_Hu_n^H
  \longrightarrow0
\end{split}
  \label{eq:zero-germ-finite-left-arm-collapse}
\end{equation}
in probability, by \eqref{eq:centered-modulus}.  The lifetime tends to zero,
and the right-arm stopping at \(L_n\) is absent on every fixed compact for
large \(n\).  The positive-path convergence just proved and Slutsky's lemma
therefore give
\begin{equation*}
  P_{h_n,u_n,L_n}\Longrightarrow P_0.
\end{equation*}

The finite and continuum normalisers must be kept separate.  From
\eqref{eq:zero-germ-finite-endpoint-collapse},
\begin{equation}
  0\leq1-z_{h_n,u_n,L_n}
  \leq\EE_{Q_n}G_{-u_n}
  \leq C_Hu_n^H,
  \label{eq:zero-germ-finite-normalizer-limit}
\end{equation}
so \(z_{h_n,u_n,L_n}\to1\).  If
\(w(a,v,\ell,r)=e^{-v}\), then, for \(\|F\|_\infty\leq1\),
\begin{equation*}
\begin{split}
  \left|K_{h_n,u_n,L_n}(F)-P_{h_n,u_n,L_n}(F)\right|
  &\leq
  \frac{\int|w-z_{h_n,u_n,L_n}|\,
        \mathrm dP_{h_n,u_n,L_n}}
       {z_{h_n,u_n,L_n}}\\
  &\leq
  \frac{2(1-z_{h_n,u_n,L_n})}
       {z_{h_n,u_n,L_n}}
  \longrightarrow0.
\end{split}
\end{equation*}
This proves all of \eqref{eq:zero-germ-finite-component-limit}.

We transfer the finite statement to the continuum family by a diagonal.  Let
\(u_j\to0\), and discard the terms equal to zero.  For each remaining \(j\),
the moving-horizon convergence in
Proposition~\ref{prop:finite-left-entrance-law} also gives convergence of the
unweighted marks: the lifetime converges, the moving-endpoint discrepancy of
the constant left arm is controlled by the same uniform modulus, and the
right-arm truncation is eventually absent on every fixed compact.  We may
therefore choose
\begin{equation*}
  h_j=2^{-r_j}\leq j^{-1},
  \qquad
  v_j,L_j\in h_j\mathbb Z_{\geq0},
  \qquad
  |v_j-u_j|\leq j^{-1},
  \qquad
  L_j\geq j,
\end{equation*}
such that
\begin{equation*}
  d_{\mathrm{BL}}(P_{h_j,v_j,L_j},P_{u_j})\leq j^{-1}.
\end{equation*}
Here the terms are chosen recursively so that \((r_j)\) is strictly
increasing; a term in a defining sequence for \(P_{u_j}\) is simply taken far
enough to meet this requirement and all four displayed bounds.  Since
\(v_j\to0\), the finite statement and
the triangle inequality give \(P_{u_j}\Rightarrow P_0\).  Together with
\eqref{eq:positive-entrance-arm-continuity}, this proves continuity of
\(u\mapsto P_u\) on \([0,\infty)\).

The bounded continuous weight \(w\) now gives
\begin{equation*}
  z(u_j)=\int w\,\mathrm dP_{u_j}\longrightarrow
  \int w\,\mathrm dP_0=1,
\end{equation*}
and, for every \(F\in C_b(\mathcal Y)\),
\begin{equation*}
  K_{u_j}(F)
  =\frac{\int Fw\,\mathrm dP_{u_j}}
         {\int w\,\mathrm dP_{u_j}}
  \longrightarrow\int F\,\mathrm dP_0=K_0(F).
\end{equation*}
The positive-horizon continuity in
Lemma~\ref{lem:entrance-endpoint-tilt} proves continuity of
\(u\mapsto K_u\) on \([0,\infty)\).
\end{proof}

\begin{proposition}
\label{prop:uniform-entrance-component-compiler}
For every \(R<\infty\),
\begin{equation}
\begin{split}
  \lim_{r,S\to\infty}\ \sup\bigl\{
  &d_{\mathrm{BL}}(K_{2^{-r},u,L},K_u):
  0\leq u\leq R,\\
  &u,L\in2^{-r}\mathbb Z_{\geq0}, L\geq S
  \bigr\}=0.
\end{split}
  \label{eq:uniform-entrance-compiler-closed}
\end{equation}
\end{proposition}

\begin{proof}
We first record the finite-component normaliser estimate.  Suppose
\(h=2^{-r}\) and \(0\leq u\leq R\).  If \(u=0\), then \(G_{-u}=G_0=0\) and
\(z_{h,0,L}=1\).  If \(u>0\), put
\begin{equation*}
  \mathcal F_{h,u,L}
  :=\bigl(h\mathbb Z\cap[-u,L]\bigr)\setminus\{0\}.
\end{equation*}
Since \(W_0=0\), this finite-wall law is \(Q_{\mathcal F_{h,u,L}}\), and the
endpoint \(-u\) belongs to its dyadic wall.  Hence
\eqref{eq:pin-mean-bound} and the wall constraint give
\begin{equation*}
  0\leq\EE_{Q_{h,u,L}}G_{-u}
  =-m_{\mathcal F_{h,u,L}}(-u)\leq C_Hu^H.
\end{equation*}
Jensen's inequality therefore yields, in both cases,
\begin{equation}
  z_{h,u,L}\geq e^{-C_Hu^H}\geq e^{-C_HR^H}>0.
  \label{eq:finite-entrance-normalizer-lower}
\end{equation}
In particular, on \(u\leq R\) the density in
\eqref{eq:finite-entrance-endpoint-tilt} is bounded above by
\(e^{C_HR^H}\).

We next isolate the sequential statement.  Let
\begin{equation*}
  h_n=2^{-r_n}\downarrow0,\qquad
  u_n\in h_n\mathbb Z_{\geq0},\quad u_n\to u>0,\qquad
  L_n\in h_n\mathbb Z_{\geq0},\quad L_n\to\infty,
\end{equation*}
and set
\begin{equation*}
  Q_n:=Q_{h_n,u_n,L_n},\qquad
  P_n:=(\mathscr M_{u_n,L_n})_*Q_n,
  \qquad P_u:=(\mathscr A_u)_*Q^u.
\end{equation*}
Under \(Q_n\), put \(\widetilde W_n=\mathsf E_{u_n\to u}W\).
Proposition~\ref{prop:finite-left-entrance-law} gives
\(\mathcal L_{Q_n}(\widetilde W_n)\Rightarrow Q^u\) in
\(C_{\mathrm{loc}}([-u,\infty))\), and hence
\(\mathscr A_u(\widetilde W_n)\Rightarrow P_u\).  If \(u_n\leq u\), its
endpoint and constant-extended left arm agree with those in
\(\mathscr M_{u_n,L_n}(W)\); if \(u_n>u\), their discrepancy is bounded by
the oscillation of \(W\) on \([-u_n,-u]\).  This oscillation tends to zero in
probability by the uniform centred and mean moduli used in the moving-endpoint
proof of Proposition~\ref{prop:finite-left-entrance-law}.  The lifetime marks
satisfy \(u_n\to u\), and on every fixed right-arm compact the truncation at
\(L_n\) is absent eventually.  It follows that \(P_n\Rightarrow P_u\) in
\(\mathcal Y\).

For \(y=(a,v,\ell,r)\in\mathcal Y\), put \(w(y)=e^{-v}\).  This function is
bounded and continuous.  In particular,
\begin{equation*}
  z_{h_n,u_n,L_n}=\int w\,\mathrm dP_n
  \longrightarrow\int w\,\mathrm dP_u=z(u)>0,
\end{equation*}
where positivity is \eqref{eq:entrance-normalizer-lower}.  Thus, for every
\(F\in C_b(\mathcal Y)\),
\begin{equation*}
  \int F\,\mathrm dK_{h_n,u_n,L_n}
  =\frac{\int Fw\,\mathrm dP_n}{\int w\,\mathrm dP_n}
  \longrightarrow
  \frac{\int Fw\,\mathrm dP_u}{\int w\,\mathrm dP_u}
  =\int F\,\mathrm dK_u.
\end{equation*}
Consequently \(K_{h_n,u_n,L_n}\Rightarrow K_u\), and the fixed
bounded-Lipschitz metric
metrises this convergence.  Lemma~\ref{lem:entrance-endpoint-tilt} also gives
\(K_{u_n}\Rightarrow K_u\), so the triangle inequality proves
\begin{equation}
  d_{\mathrm{BL}}(K_{h_n,u_n,L_n},K_{u_n})\longrightarrow0.
  \label{eq:sequential-entrance-compiler-positive}
\end{equation}

Suppose now that \eqref{eq:uniform-entrance-compiler-closed} failed.  Then
there would be \(\eta>0\), sequences \(r_n,S_n\to\infty\), and admissible
\(u_n,L_n\), with \(0\leq u_n\leq R\) and \(L_n\geq S_n\), for which the
displayed distance is at least \(\eta\).  Passing to a subsequence, assume
that \(r_n\) is strictly increasing and \(u_n\to u\in[0,R]\).  Then
\(2^{-r_n}\downarrow0\) and \(L_n\to\infty\).  If \(u>0\),
\eqref{eq:sequential-entrance-compiler-positive} is a contradiction.  If
\(u=0\), Corollary~\ref{cor:zero-germ-identification} gives
\begin{equation*}
  K_{2^{-r_n},u_n,L_n}\Longrightarrow K_0,
  \qquad
  K_{u_n}\Longrightarrow K_0,
\end{equation*}
so their bounded-Lipschitz distance again tends to zero.  This proves the
joint limit in \(r\) and \(S\).

\end{proof}

For the meshes \(h_n=2^{-r_n}\) chosen in
Proposition~\ref{prop:relative-tilted-grid-diagonal}, so that the tilted
grid laws approximate the corresponding tilted laws of the
continuous-maximiser variables, write
\begin{equation}
\begin{split}
  \mu_n
  &:=\mathcal L_{\widehat{\PP}_n^{\mathrm{grid}}}
       (Y_n^{\mathrm{grid}})\\
  &=\sum_{k=0}^{m_n}q_{n,k}
    K_{h_n,L_{n,k}^-,L_{n,k}^+},
\end{split}
  \label{eq:tilted-marked-component-mixture}
\end{equation}
and
\begin{equation}
  \pi_n:=\sum_{k=0}^{m_n}q_{n,k}\delta_{L_{n,k}^-}
  \label{eq:tilted-left-horizon-mixture}
\end{equation}
for its left-horizon marginal.  Also set
\begin{equation}
\begin{aligned}
  U_n^{\mathrm{grid}}&:=L_{n,K_n}^-,
  &V_n^{\mathrm{grid}}&:=a_n^HM_n^*,\\
  D_n^{\mathrm{grid}}&:=L_{n,K_n}^+
  =h_nm_n-U_n^{\mathrm{grid}}.
\end{aligned}
  \label{eq:grid-left-height-right-horizons}
\end{equation}
The fixed-row convergence and bounded-likelihood calculation in the proof of
Proposition~\ref{prop:relative-tilted-grid-diagonal} allow its recursive
diagonal to be taken far enough that, in addition,
\begin{equation}
  d_{\mathrm{BL}}^{(2)}\left(
    \mathcal L_{\widehat{\PP}_n^{\mathrm{grid}}}
      (U_n^{\mathrm{grid}},V_n^{\mathrm{grid}}),
    \mathcal L_{\widehat{\PP}_{a_n}}(U_{a_n},V_{a_n})
  \right)\leq n^{-1},
  \label{eq:scalar-tilted-grid-diagonal}
\end{equation}
where \(d_{\mathrm{BL}}^{(2)}\) is a bounded-Lipschitz metric for the usual
topology of \(\mathbb R_+^2\).  We make this harmless additional choice.

\begin{theorem}
\label{thm:subsequential-boundary-entrance-mixture}
The measures \((\pi_n)\) and the corresponding marked tilted grid laws are
tight.  From every subsequence one may extract a further subsequence, a
probability measure \(\pi\) on \(\mathbb R_+\), and a marked law
\(\Lambda\) such that
\begin{equation}
  \pi_n\Longrightarrow\pi
  \label{eq:subsequential-mixing-limit}
\end{equation}
and the marked grid laws converge to \(\Lambda\), where
\begin{equation}
  \Lambda=\int_{[0,\infty)}K_u\,\pi(\mathrm du).
  \label{eq:full-subsequential-entrance-mixture}
\end{equation}
In particular,
\begin{equation}
  \Lambda(\mathrm dy;U>0)
  =\int_{(0,\infty)}K_u(\mathrm dy)\,\pi(\mathrm du).
  \label{eq:positive-horizon-mixture}
\end{equation}
Along the same further subsequence,
\begin{equation}
  \mathcal L_{\widehat{\PP}_{a_n}}(Y_{a_n})
  \Longrightarrow\Lambda,
  \label{eq:continuous-subsequential-marked-limit}
\end{equation}
and its \(U\)-marginal converges to \(\pi\).  The grid right horizons
\(D_n^{\mathrm{grid}}\) and the continuous right horizons \(a_n-U_{a_n}\)
both tend to infinity in probability.
\end{theorem}

\begin{proof}
Theorem~\ref{thm:tilted-scalar-compact-containment} and
\eqref{eq:scalar-tilted-grid-diagonal} give tightness of \((\pi_n)\) and of
the grid height coordinates.  Moreover,
\begin{equation*}
  0\leq a_n-h_nm_n<h_n,
\end{equation*}
so \(h_nm_n\to\infty\).  Since
\(D_n^{\mathrm{grid}}=h_nm_n-U_n^{\mathrm{grid}}\) and the grid left
horizons are tight,
\begin{equation}
  D_n^{\mathrm{grid}}\longrightarrow\infty
  \quad\text{in probability}.
  \label{eq:grid-right-horizon-divergence}
\end{equation}
The continuous right-horizon assertion is
\eqref{eq:tilted-right-horizon-divergence}.

We next prove tightness of \((\mu_n)\).  Fix \(\eta>0\), and use tightness of
\((\pi_n)\) to choose \(R<\infty\) with
\begin{equation*}
  \sup_n\pi_n((R,\infty))<\eta.
\end{equation*}
On every component with left horizon \(u\leq R\),
\eqref{eq:finite-entrance-normalizer-lower} bounds the endpoint-tilted
density by \(e^{C_HR^H}\).  The uniform centred modulus
\eqref{eq:centered-modulus} and conditional-mean bounds
\eqref{eq:conditional-mean-modulus}--\eqref{eq:pin-mean-bound}, applied on
\([-R,0]\) and on each fixed positive compact, therefore give uniform
compact containment and equicontinuity for the two stopped, constant-extended
arms; the paths are anchored by \(W_0=m_F(0)=0\).  The stopping maps do not
increase their moduli.  Arzelà--Ascoli on
each compact, followed by a diagonal exhaustion, gives tightness of the arm
pair outside mass at most \(\eta\).  Together with tightness of the lifetime
and height coordinates, this proves tightness of \((\mu_n)\).

Starting from an arbitrary subsequence, take a common further subsequence,
not relabelled, such that
\begin{equation*}
  \pi_n\Longrightarrow\pi,
  \qquad
  \mu_n\Longrightarrow\Lambda.
\end{equation*}
The first-coordinate marginal of \(\mu_n\) is \(\pi_n\) by
\eqref{eq:tilted-marked-component-mixture}.  Since projection onto \(U\) is
continuous, the \(U\)-marginal of \(\Lambda\) is \(\pi\).

We now construct the right-horizon cutoff.  By
\eqref{eq:grid-right-horizon-divergence}, choose strictly increasing integers
\(N_j\) such that
\begin{equation*}
  \widehat{\PP}_n^{\mathrm{grid}}
    (D_n^{\mathrm{grid}}<j)\leq2^{-j},
  \qquad n\geq N_j.
\end{equation*}
After changing a finite prefix, define \(S_n=j\) for
\(N_j\leq n<N_{j+1}\).  Then \(S_n\to\infty\), and
\begin{equation}
  b_n:=\sum_{k=0}^{m_n}q_{n,k}
       \mathbf1_{\{L_{n,k}^+<S_n\}}
  =\widehat{\PP}_n^{\mathrm{grid}}
       (D_n^{\mathrm{grid}}<S_n)
  \longrightarrow0.
  \label{eq:short-right-component-mass}
\end{equation}

Let \(F\) be bounded and Lipschitz on \(\mathcal Y\), and let
\(\chi\in C_c(\mathbb R_+)\), with support contained in \([0,R]\).  Put
\begin{equation*}
  \Delta_n:=
  \sup\left\{
    d_{\mathrm{BL}}(K_{h_n,u,L},K_u):
    0\leq u\leq R,\
    u,L\in h_n\mathbb Z_{\geq0},\ L\geq S_n
  \right\}.
\end{equation*}
Proposition~\ref{prop:uniform-entrance-component-compiler} gives
\(\Delta_n\to0\).  The exact mixture
\eqref{eq:tilted-marked-component-mixture} and
\eqref{eq:short-right-component-mass} imply
\begin{equation}
\begin{split}
  \left|
    \int_{\mathcal Y}\chi(U)F\,\mathrm d\mu_n
    -\int_{\mathbb R_+}\chi(u)K_u(F)\,\pi_n(\mathrm du)
  \right|
  \leq{}&
  \|\chi\|_\infty\|F\|_{\mathrm{BL}}\Delta_n\\
  &+2\|\chi\|_\infty\|F\|_\infty b_n
  \longrightarrow0.
\end{split}
  \label{eq:component-kernel-approximation}
\end{equation}

Corollary~\ref{cor:zero-germ-identification} makes
\(u\mapsto K_u(F)\) continuous on \(\mathbb R_+\).  Hence
\(\chi(u)K_u(F)\) belongs to \(C_b(\mathbb R_+)\).
Passing to the two weak limits in
\eqref{eq:component-kernel-approximation} gives
\begin{equation}
  \int_{\mathcal Y}\chi(U)F\,\mathrm d\Lambda
  =\int_{\mathbb R_+}\chi(u)K_u(F)\,\pi(\mathrm du).
  \label{eq:positive-cutoff-mixture}
\end{equation}
Choose \(0\leq\chi_j\leq1\) in \(C_c(\mathbb R_+)\) with
\(\chi_j(u)\to1\) for every \(u\geq0\).  Dominated convergence in
\eqref{eq:positive-cutoff-mixture}, followed by the fact that bounded
Lipschitz functions determine finite Borel measures on \(\mathcal Y\),
proves \eqref{eq:full-subsequential-entrance-mixture}; restriction to
\(\{U>0\}\) gives \eqref{eq:positive-horizon-mixture}.

Equation~\eqref{eq:relative-tilted-diagonal} and the triangle inequality now
give \eqref{eq:continuous-subsequential-marked-limit}; continuous projection
gives convergence of its \(U\)-marginal to \(\pi\).
\end{proof}

It remains to prove that the mixing measure in
\eqref{eq:full-subsequential-entrance-mixture} is independent of the chosen
subsequence.

\begin{lemma}
\label{lem:submarkov-fbm-precision}
Let \(F\subset\mathbb R\setminus\{0\}\) be nonempty and finite, let \(C_F\) be the
covariance matrix of \(B_F\), and put \(\mathsf P_F=C_F^{-1}\).  Then
\begin{equation}
  (\mathsf P_F)_{s,t}\leq0\quad(s\neq t),
  \qquad
  \mathsf P_F\mathbf1_F\geq0
  \quad\text{coordinatewise}.
  \label{eq:submarkov-precision-signs}
\end{equation}
For every nonempty \(I\subseteq F\), the conditional precision
\((\mathsf P_F)_{II}\) and the marginal precision \(C_{II}^{-1}\) have the
same two properties.  If \(J:=F\setminus I\) is nonempty, the latter is the
Schur complement of \((\mathsf P_F)_{JJ}\) in \(\mathsf P_F\).
\end{lemma}

\begin{proof}
The Green-kernel identity \eqref{eq:fbm-stable-green-identification} and
\textcite[Definition~3.5 and
Theorem~3.6(ii)]{eisenbaumKaspi2006squaredGaussian} show directly that
\(\mathsf P_F\) is a diagonally dominant \(M\)-matrix.  This is exactly the
two assertions in \eqref{eq:submarkov-precision-signs}: diagonal dominance
is defined there as nonnegativity of the row sums.

Fix a nonempty \(I\subseteq F\), put \(J:=F\setminus I\), and write
\begin{equation}
  \mathsf P=
  \begin{pmatrix}
    \mathsf P_{II}&\mathsf P_{IJ}\\
    \mathsf P_{JI}&\mathsf P_{JJ}
  \end{pmatrix},
  \qquad
  k:=\mathsf P\mathbf1_F\geq0.
\end{equation}
The conditional precision on \(I\) is \(\mathsf P_{II}\).  It is positive
definite, retains the off-diagonal signs, and
\begin{equation}
  \mathsf P_{II}\mathbf1_I
  =k_I-\mathsf P_{IJ}\mathbf1_J
  \geq k_I\geq0.
  \label{eq:principal-block-row-sums}
\end{equation}
If \(J=\varnothing\), this is the original precision and there is nothing
more to prove, so assume that \(J\) is nonempty.

The block \(\mathsf P_{JJ}\) is positive definite with nonpositive
off-diagonal entries, and hence its inverse is entrywise nonnegative.  For
completeness, choose \(a>\lambda_{\max}(\mathsf P_{JJ})\) and put
\(N=a\,\mathrm{Id}_J-\mathsf P_{JJ}\), which is entrywise nonnegative.
Its spectral radius satisfies
\(\rho(N)=a-\lambda_{\min}(\mathsf P_{JJ})<a\), and therefore
\begin{equation*}
  \mathsf P_{JJ}^{-1}
  =a^{-1}\sum_{m=0}^{\infty}(N/a)^m\geq0
  \quad\text{entrywise}.
\end{equation*}

The block inverse formula gives the marginal precision on \(I\):
\begin{equation}
  \mathsf S_I
  :=\mathsf P_{II}
    -\mathsf P_{IJ}\mathsf P_{JJ}^{-1}\mathsf P_{JI}
  =C_{II}^{-1},
  \label{eq:marginal-precision-schur}
\end{equation}
Since \(\mathsf P_{JJ}^{-1}\geq0\), \(\mathsf P_{IJ}\leq0\), and
\(\mathsf P_{JI}\mathbf1_I+\mathsf P_{JJ}\mathbf1_J=k_J\),
\begin{equation}
  \mathsf S_I\mathbf1_I
  =k_I-\mathsf P_{IJ}\mathsf P_{JJ}^{-1}k_J
  \geq0.
  \label{eq:schur-row-sums}
\end{equation}
Since \(\mathsf P\) is positive definite, its Schur complement
\(\mathsf S_I\) is positive definite.  Moreover, for distinct \(r,s\in I\),
\begin{equation}
  (\mathsf S_I)_{r,s}
  =\mathsf P_{r,s}
   -\mathsf P_{r,J}\mathsf P_{JJ}^{-1}\mathsf P_{J,s}
  \leq0,
  \label{eq:schur-off-diagonal-signs}
\end{equation}
because \(\mathsf P_{r,s}\leq0\), the entries of
\(\mathsf P_{r,J}\) and \(\mathsf P_{J,s}\) are nonpositive, and
\(\mathsf P_{JJ}^{-1}\geq0\) entrywise.  Thus \(\mathsf S_I\) is Stieltjes,
completing the proof.
\end{proof}

\begin{lemma}
\label{lem:nested-diagonal-orthant-ratios}
Let \(I,J\subset\mathbb R\setminus\{0\}\) be disjoint and finite, and define
\begin{equation}
  F_I(x):=\PP(B_t\leq x:t\in I),
  \qquad
  F_{I\cup J}(x):=\PP(B_t\leq x:t\in I\cup J).
  \label{eq:nested-diagonal-orthants}
\end{equation}
We use the convention \(F_\varnothing(x)=1\).  Then
\begin{equation}
  x\longmapsto\frac{F_{I\cup J}(x)}{F_I(x)}
  \label{eq:nested-orthant-ratio}
\end{equation}
is nondecreasing wherever its denominator is positive.
\end{lemma}

\begin{proof}
If \(J=\varnothing\), the ratio is identically one.  If \(I=\varnothing\)
and \(J\neq\varnothing\), it equals \(F_J(x)\), which is nondecreasing by
event inclusion.  Hence assume that \(I\) and \(J\) are both nonempty.  The
nondegenerate Gaussian vector \(B_I\) has an everywhere positive density, so
\(F_I(x)>0\) for every \(x\in\mathbb R\).

At barrier level \(x\), put
\begin{equation}
  X^{(x)}:=x\mathbf1_{I\cup J}-B_{I\cup J}.
\end{equation}
This Gaussian vector has mean \(x\mathbf1_{I\cup J}\), covariance
\(C_{I\cup J}\),
and a Stieltjes precision \(\mathsf P\) satisfying
\(k:=\mathsf P\mathbf1_{I\cup J}\geq0\) by
Lemma~\ref{lem:submarkov-fbm-precision}.  The marginal precision
\(\mathsf S_I=C_{II}^{-1}\) satisfies
\(\mathsf S_I\mathbf1_I\geq0\) by
\eqref{eq:schur-row-sums}.  Write
\begin{equation*}
  \mu_x:=\mathcal L\big(X_I^{(x)}\mid X_I^{(x)}\geq0\big).
\end{equation*}
This law is MTP\(_2\), hence associated, by
Lemma~\ref{lem:finite-wall-mtp2-closure}.  If \(y>x\), the
Radon--Nikodym derivative of \(\mu_y\) with respect to \(\mu_x\) is
\begin{equation}
  r_{y,x}(z):=\frac{\mathrm d\mu_y}{\mathrm d\mu_x}(z)
  =c_{y,x}\exp\left\{
    (y-x)(\mathsf S_I\mathbf1_I)^{\mathsf T}z
  \right\},
  \qquad z\geq0,
  \label{eq:common-shift-likelihood-ratio}
\end{equation}
for a constant \(c_{y,x}>0\).  Thus \(r_{y,x}\) is coordinatewise
nondecreasing.  For every bounded nonnegative coordinatewise nondecreasing
test \(\phi\), association applied under \(\mu_x\) to \(\phi\) and
\(r_{y,x}\wedge m\) gives
\begin{equation*}
  \int\phi(r_{y,x}\wedge m)\,\mathrm d\mu_x
  \geq
  \int\phi\,\mathrm d\mu_x
  \int(r_{y,x}\wedge m)\,\mathrm d\mu_x.
\end{equation*}
Letting \(m\to\infty\) and using \(\int r_{y,x}\,\mathrm d\mu_x=1\)
shows that \(\int\phi\,\mathrm d\mu_y\geq\int\phi\,\mathrm d\mu_x\).
Adding a constant covers every bounded increasing test.  Consequently
\begin{equation}
  \mu_x\leq_{\mathrm{st}}\mu_y.
  \label{eq:truncated-common-shift-order}
\end{equation}

The inverse-positivity argument in
Lemma~\ref{lem:submarkov-fbm-precision} gives
\(\mathsf P_{JJ}^{-1}\geq0\) entrywise.  Conditional on
\(X_I^{(x)}=z\), the vector \(X_J^{(x)}\) has covariance
\(\mathsf P_{JJ}^{-1}\) and mean
\begin{equation}
  Rz+x h,
  \qquad
  R:=-\mathsf P_{JJ}^{-1}\mathsf P_{JI}\geq0,
  \qquad
  h:=\mathbf1_J-R\mathbf1_I
     =\mathsf P_{JJ}^{-1}k_J\geq0.
  \label{eq:submarkov-conditional-mean}
\end{equation}
Let \(W\sim N(0,\mathsf P_{JJ}^{-1})\).  The canonical conditional version
\begin{equation}
  g_x(z):=\PP(W+Rz+xh\geq0)
  \label{eq:nested-orthant-conditional}
\end{equation}
is coordinatewise nondecreasing in \(z\) and nondecreasing in \(x\), by a
same-\(W\) coupling.  Therefore, if \(y>x\),
\begin{align*}
  \frac{F_{I\cup J}(y)}{F_I(y)}
  =\int g_y\,\mathrm d\mu_y
  &\geq\int g_x\,\mathrm d\mu_y \\
  &\geq\int g_x\,\mathrm d\mu_x
   =\frac{F_{I\cup J}(x)}{F_I(x)},
\end{align*}
where the second inequality is \eqref{eq:truncated-common-shift-order}.
Equivalently,
\begin{equation}
  \PP(X_J^{(x)}\geq0\mid X_I^{(x)}\geq0)
  =\frac{F_{I\cup J}(x)}{F_I(x)}
\end{equation}
is nondecreasing in \(x\), as claimed.
\end{proof}

For \(T>0\), recall the distribution function \(F_T\) and barrier law
\(\beta_T\) from \eqref{eq:exponential-barrier-marginal}.

\begin{corollary}
\label{cor:monotone-exponential-barriers}
If \(T'>T\), then \(\beta_T\leq_{\mathrm{st}}\beta_{T'}\).  The family is
tight and has a unique weak limit \(\beta_\infty\) as \(T\to\infty\).
\end{corollary}

\begin{proof}
For \(n\geq1\), put
\begin{equation*}
  D_n:=2^{-n}\mathbb Z\cap(0,T],
  \qquad
  D_n':=2^{-n}\mathbb Z\cap(0,T'].
\end{equation*}
These are increasing finite sets with \(D_n\subset D_n'\), and their unions
are dense in the respective intervals.  Put
\begin{equation*}
  F_{T,n}(x):=\PP(B_t\leq x:t\in D_n),
  \qquad
  F_{T',n}(x):=\PP(B_t\leq x:t\in D_n').
\end{equation*}
Lemma~\ref{lem:nested-diagonal-orthant-ratios} makes
\(F_{T',n}(x)/F_{T,n}(x)\) nondecreasing in \(x\) for every \(n\).
By path continuity, the finite-wall events decrease to the corresponding
continuum-wall events, and hence
\begin{equation*}
  F_{T,n}(x)\downarrow F_T(x),
  \qquad
  F_{T',n}(x)\downarrow F_{T'}(x).
\end{equation*}
For \(x>0\), the Gaussian support argument used in
Lemma~\ref{lem:fbm-compact-positivity-support} gives
\begin{equation*}
  F_T(x)
  \geq\PP\bigl(\lVert B\rVert_{C([0,T])}<x\bigr)>0.
\end{equation*}
Thus the finite-grid ratios converge to the quotient of the two limits.
Their pointwise limit is again nondecreasing, so
\begin{equation}
  x\longmapsto\frac{F_{T'}(x)}{F_T(x)}
  \quad\text{is nondecreasing on }(0,\infty).
  \label{eq:continuous-nested-wall-ratio}
\end{equation}
Both barrier laws are atomless at zero, and \(F_T(x)>0\) for \(x>0\).
Consequently \(\beta_{T'}\ll\beta_T\), and, \(\beta_T\)-almost everywhere,
\begin{equation}
  r_{T',T}(x)
  :=\frac{\mathrm d\beta_{T'}}{\mathrm d\beta_T}(x)
  =\frac{\mathcal Z(T)}{\mathcal Z(T')}
    \frac{F_{T'}(x)}{F_T(x)}
  \label{eq:barrier-monotone-likelihood-ratio}
\end{equation}
is nondecreasing.  To check the orientation locally, let \(X,\widetilde X\)
be independent with law \(\beta_T\).  For every bounded nondecreasing Borel
function \(\varphi\), integrability of \(r_{T',T}\) and
\(\EE r_{T',T}(X)=1\) give
\begin{equation*}
\begin{split}
  \int\varphi\,\mathrm d\beta_{T'}
  -\int\varphi\,\mathrm d\beta_T
  &=\operatorname{Cov}\bigl(\varphi(X),r_{T',T}(X)\bigr)\\
  &=\frac12\EE\left[
      \bigl(\varphi(X)-\varphi(\widetilde X)\bigr)
      \bigl(r_{T',T}(X)-r_{T',T}(\widetilde X)\bigr)
    \right]\geq0.
\end{split}
\end{equation*}
This proves \(\beta_T\leq_{\mathrm{st}}\beta_{T'}\).

The law \(\beta_T\) is the barrier marginal in
Lemma~\ref{lem:exponential-barrier-disintegration}, and
\eqref{eq:exponential-barrier-tail} gives
\begin{equation}
  \sup_{T>0}\beta_T([A,\infty))
  \leq C_He^{-A/2}.
  \label{eq:uniform-exponential-barrier-tightness}
\end{equation}
Thus the family is tight.  To verify convergence of the whole family, let
\begin{equation*}
  G_T(a):=\beta_T([0,a]),
  \qquad
  q_T(u):=\inf\{a\geq0:G_T(a)\geq u\},
  \quad 0<u<1.
\end{equation*}
The stochastic order just proved gives \(q_T(u)\leq q_{T'}(u)\) whenever
\(T<T'\).  For each fixed \(u<1\), the uniform tail bound supplies
\(A_u<\infty\) such that \(q_T(u)\leq A_u\) for every \(T\).  Hence
\begin{equation*}
  q_\infty(u):=\sup_{m\geq1}q_m(u)<\infty.
\end{equation*}
If \(U\) is uniform on \((0,1)\), then \(q_T(U)\) has law \(\beta_T\).
Monotonicity, with the integer horizons bracketing every sufficiently large
\(T\), gives
\begin{equation*}
  q_T(U)\longrightarrow q_\infty(U)
  \quad\text{almost surely as }T\to\infty.
\end{equation*}
Therefore \(\beta_T\) converges weakly to the unique probability law
\(\beta_\infty:=\mathcal L(q_\infty(U))\).
\end{proof}

\begin{lemma}
\label{lem:shifted-hard-wall-gate2}
Let \(F\subset(0,\infty)\cap\mathbb D\) be finite and nonempty, let
\(x\geq0\), and put
\begin{equation}
  \begin{aligned}
    Q_{F,x}
    &:=\mathcal L(G\mid G_s\geq-x,\ s\in F),\\
    g_{F,x}(t)&:=\EE_{Q_{F,x}}G_t,
    &\qquad
    Z_{F,x}(t)&:=G_t-g_{F,x}(t).
  \end{aligned}
  \label{eq:shifted-finite-wall-law}
\end{equation}
There are coefficients \(\lambda_{F,x}(s)\geq0\), \(s\in F\), such that
\begin{equation}
  g_{F,x}(t)
  =\sum_{s\in F}\lambda_{F,x}(s)K_H(t,s).
  \label{eq:shifted-wall-positive-mixture}
\end{equation}
Uniformly in \(F\) and \(x\geq0\),
\begin{equation}
  0\leq g_{F,x}(t)\leq g_{F,0}(t)\leq C_H|t|^H,
  \qquad t\in\mathbb R.
  \label{eq:shifted-wall-pin-bound}
\end{equation}
For every \(0<a<R<\infty\), there is \(C_{H,a,R}<\infty\) such that
\begin{equation}
  |g_{F,x}(t)-g_{F,x}(s)|
  \leq C_{H,a,R}|t-s|^{2H},
  \qquad s,t\in[a,R].
  \label{eq:shifted-wall-annular-mean-modulus}
\end{equation}
Consequently, for every \(R<\infty\),
\begin{equation}
  \lim_{\delta\downarrow0}\sup_{F,x}
  \sup_{\substack{s,t\in[0,R]\\|s-t|\leq\delta}}
  |g_{F,x}(t)-g_{F,x}(s)|=0,
  \label{eq:shifted-wall-full-mean-modulus}
\end{equation}
and, for every \(\varepsilon>0\),
\begin{equation}
  \lim_{\delta\downarrow0}\sup_{F,x}
  Q_{F,x}\left(
    \sup_{\substack{s,t\in[0,R]\\|s-t|\leq\delta}}
    |Z_{F,x}(t)-Z_{F,x}(s)|>\varepsilon
  \right)=0.
  \label{eq:shifted-wall-centered-modulus}
\end{equation}

For \(T>0\) and \(x>0\), these bounds pass to
\begin{equation}
  \mathbb P_{T,x}
  =\mathcal L(B\mid M_T\leq x)
  \label{eq:shifted-wall-continuum-law}
\end{equation}
by dense finite-grid hard-wall approximation.  In particular, for every
\(A,R<\infty\), the laws of \(x-B\) on \(C([0,R])\), under
\(\mathbb P_{T,x}\) with \(T\geq R\) and \(0<x\leq A\), are tight uniformly
in \(T\) and \(x\).
\end{lemma}

\begin{proof}
Let \(X=G_F\), let \(C_F\) and \(P_F=C_F^{-1}\) be its covariance and
precision, and put
\begin{equation}
  \bar x:=\EE(X\mid X\geq-x\mathbf1).
\end{equation}
Nondegeneracy makes the centred Gaussian density \(p_F\) strictly positive,
so the shifted-orthant normaliser is finite and strictly positive.  Integration
in coordinate \(i\) over \([-x,\infty)\), with the empty complementary
integral interpreted in the usual way when \(|F|=1\), and using
\(\partial_i p_F(y)=-(P_Fy)_ip_F(y)\), gives
\begin{equation}
\begin{split}
  (P_F\bar x)_i
  &=
  \frac{1}{\PP(X\geq-x\mathbf1)}
  \int_{[-x,\infty)^F}(P_Fy)_ip_F(y)\,\mathrm dy\\
  &=
  \frac{1}{\PP(X\geq-x\mathbf1)}
  \int_{[-x,\infty)^{F\setminus\{i\}}}
    p_F(-x,y_{-i})\,\mathrm dy_{-i}\geq0.
\end{split}
  \label{eq:shifted-wall-boundary-flux}
\end{equation}
Set \(\lambda_{F,x}=P_F\bar x\).  Gaussian regression gives
\(\EE(G_t\mid X)=K_H(t,F)P_FX\); averaging over the shifted orthant proves
\eqref{eq:shifted-wall-positive-mixture}.  If \(s>0\), the triangle
inequality and subadditivity of the power \(2H\in(0,1)\) give
\[
  |t-s|^{2H}\leq(|t|+s)^{2H}\leq|t|^{2H}+s^{2H},
\]
so \(K_H(t,s)\geq0\).  Thus \eqref{eq:shifted-wall-positive-mixture} also
proves \(g_{F,x}\geq0\).

The comparison with the zero wall is immediate at \(t=0\).  If \(t\neq0\),
form the vector on the set \(F\cup\{t\}\), so a target already in \(F\) is
not duplicated, and restrict it to \(\{G_F\geq-x\mathbf1\}\).  Its law is
MTP\(_2\) and associated by
Lemma~\ref{lem:finite-wall-mtp2-closure}.  The increasing event
\(C=\{G_F\geq0\}\) has conditional probability
\[
  \frac{\PP(G_F\geq0)}{\PP(G_F\geq-x\mathbf1)}>0,
\]
and conditioning further on \(C\) gives the zero-wall law.  Apply association
to \(\mathds1_C\) and the bounded increasing clipping
\(y\mapsto\max\{-N,\min\{y_t,N\}\}\).  Conditional Gaussian first-moment
integrability under both shifted walls permits dominated convergence as
\(N\to\infty\), and gives
\begin{equation}
  g_{F,x}(t)\leq g_{F,0}(t).
  \label{eq:shifted-wall-zero-wall-comparison}
\end{equation}
In the zero-wall notation of Section~\ref{sec:path-tightness},
\(g_{F,0}=-m_F\).  The last bound in
\eqref{eq:shifted-wall-pin-bound} is therefore
\eqref{eq:pin-mean-bound}.

Fix \(0<a<R\) and choose a positive dyadic \(r<a\).  From
\eqref{eq:shifted-wall-positive-mixture} and
Lemma~\ref{lem:covariance-kernel-ratio},
\begin{equation}
\begin{split}
  |g_{F,x}(t)-g_{F,x}(s)|
  &\leq
  C|t-s|^{2H}
  \sum_{u\in F}\lambda_{F,x}(u)K_H(r,u)\\
  &=C|t-s|^{2H}g_{F,x}(r)
  \leq C_{H,a,R}|t-s|^{2H},
\end{split}
  \label{eq:shifted-wall-mean-ratio-proof}
\end{equation}
This first proves the bound when \(|t-s|\leq1\).  Partitioning the segment
joining \(s\) and \(t\) into a bounded number of unit-length subintervals,
as in the proof of Proposition~\ref{prop:conditional-mean-modulus}, and
enlarging the constant proves
\eqref{eq:shifted-wall-annular-mean-modulus} for every \(s,t\in[a,R]\).
To obtain the full modulus, the case \(R=0\) is immediate.  Otherwise fix
\(\varepsilon>0\), choose \(a\in(0,\min\{R,1\})\) so that
\(2C_H(2a)^H<\varepsilon/2\), and then choose \(\delta_0<a\) so that
\(C_{H,a,R}\delta_0^{2H}<\varepsilon/2\).  If
\(s,t\in[0,R]\), \(|s-t|\leq\delta_0\), and
\(\min\{s,t\}\leq a\), then both points lie in \([0,2a]\), and the pin
bound makes the mean difference smaller than \(\varepsilon/2\).  Otherwise
both points lie in \([a,R]\), and the annular bound gives the same conclusion.
This proves \eqref{eq:shifted-wall-full-mean-modulus} with every choice
independent of \(F,x\).

For fixed \(s,t\), form the possibly singular Gaussian vector
\((G_u)_{u\in F\cup\{s,t\}}\), again without repeated coordinates.  The
coordinate set imposing \(G_u\geq-x\) for \(u\in F\) is convex and has
positive probability, and its conditional barycentre under the increment
functional is \(g_{F,x}(t)-g_{F,x}(s)\).  Hence
Lemma~\ref{lem:harge-convex-domination} gives
\begin{equation}
  \EE_{Q_{F,x}}
  e^{\lambda(Z_{F,x}(t)-Z_{F,x}(s))}
  \leq e^{\lambda^2|t-s|^{2H}/2}.
  \label{eq:shifted-wall-centered-increment-mgf}
\end{equation}
The dyadic chaining proof of Lemma~\ref{lem:uniform-centered-modulus} is
dimension-free and uses only this bound and continuity.  The latter holds
because \eqref{eq:shifted-wall-positive-mixture} is a finite sum of continuous
covariance kernels.  The same chaining proof therefore proves
\eqref{eq:shifted-wall-centered-modulus}, uniformly in \(F,x\).

For the continuum passage, take
\(F_n=2^{-n}\mathbb Z\cap(0,T]\), after discarding the finitely many empty
initial sets.  These sets are nested and their union is dense in \((0,T]\).
By continuity and \(G_0=0\),
\begin{equation}
  \{G_s\geq-x:s\in F_n\}\downarrow
  \{G_s\geq-x:0\leq s\leq T\}.
  \label{eq:shifted-wall-dense-event-limit}
\end{equation}
For \(x>0\), the limiting event has positive probability: the uniform ball
of radius \(x\) around the zero path has positive Gaussian measure by the
support theorem.  Write \(Q_{T,x}\) for the conditional law of \(G\) on this
limiting event.  Decreasing-event convergence and positivity of the limit
show directly that \(Q_{F_n,x}\to Q_{T,x}\) in total variation.

Total variation alone does not pass unbounded means, so fix \(t\).  Taking
\(s=0\) in \eqref{eq:shifted-wall-centered-increment-mgf}, using \(G_0=0\),
and integrating the resulting sub-Gaussian tail gives
\[
  \sup_n\EE_{Q_{F_n,x}}|Z_{F_n,x}(t)|^2\leq C|t|^{2H}.
\]
Together with \eqref{eq:shifted-wall-pin-bound}, this gives
\(\sup_n\EE_{Q_{F_n,x}}|G_t|^2<\infty\).  Thus the coordinate family is
uniformly integrable, and total variation yields
\[
  g_{F_n,x}(t)\longrightarrow
  g_{T,x}(t):=\EE_{Q_{T,x}}G_t.
\]
The common deterministic modulus
\eqref{eq:shifted-wall-full-mean-modulus}, first passed pointwise to
\(g_{T,x}\) and then applied on a finite net, upgrades this to uniform
convergence on every compact interval.

Put \(Z_{T,x}=G-g_{T,x}\), and set
\[
  w_{R,\delta}(f)
  :=\sup\{|f(t)-f(s)|:s,t\in[0,R],\ |t-s|\leq\delta\}.
\]
For each fixed \(T,x,R\), compact-uniform mean convergence
and total variation give
\[
\begin{split}
  Q_{T,x}\bigl(w_{R,\delta}(Z_{T,x})>\varepsilon\bigr)
  &=\lim_{n\to\infty}
    Q_{F_n,x}\bigl(w_{R,\delta}(Z_{T,x})>\varepsilon\bigr)\\
  &\leq\limsup_{n\to\infty}
    Q_{F_n,x}\bigl(w_{R,\delta}(Z_{F_n,x})>\varepsilon/2\bigr).
\end{split}
\]
The finite bound on the right is uniform in \(F_n,x\).  Consequently its
limit as \(\delta\downarrow0\) is zero uniformly in all continuum pairs
\(T>0,x>0\), and the same pointwise passage gives the full mean modulus.

Under \(\mathbb P_{T,x}\), the reflected process \(G=-B\) has law
\(Q_{T,x}\), and \(x-B=x+G\).  Its deterministic mean is \(x+g_{T,x}\),
the added constant does not affect a modulus, and its centred process is
\(Z_{T,x}\).
Moreover \((x-B_0)=x\leq A\).  The two uniform moduli and this one-point
bound give the asserted uniform \(C([0,R])\)-tightness by
Arzelà--Ascoli.
\end{proof}

On an extension carrying \(D\sim\operatorname{Exp}(1)\), independent of
\(B\), define under \(\widehat{\PP}_T\)
\begin{equation}
  E_T:=M_T+D,
  \qquad
  X_T(t):=E_T-B_t,\quad0\leq t\leq T,
  \label{eq:barrier-deficit-process}
\end{equation}
and extend \(X_T\) constantly after \(T\).

\begin{proposition}
\label{prop:infinite-barrier-deficit-law}
There is a unique probability law \(\Gamma_H\) on
\(\mathbb R_+\times C_{\mathrm{loc}}(\mathbb R_+)\) such that
\begin{equation}
  \mathcal L_{\widehat{\PP}_T}(E_T,X_T)
  \Longrightarrow\Gamma_H,
  \qquad T\to\infty.
  \label{eq:barrier-deficit-limit}
\end{equation}
\end{proposition}

\begin{proof}
The path-marginal and overshoot assertions of
Lemma~\ref{lem:exponential-barrier-disintegration} identify the joint law of
\((B,E_T)\) in the present construction with \(\mathbb Q_T\), and hence
\begin{equation}
  \mathrm d\mathbb Q_T(B,x)
  =\frac{e^{-x}\mathds1_{\{M_T\leq x\}}}
         {\mathcal Z(T)}\,\PP(\mathrm dB)\,\mathrm dx.
  \label{eq:barrier-deficit-disintegration}
\end{equation}
Thus \(E_T\) has law \(\beta_T\), and, conditional on \(E_T=x\),
\begin{equation}
  X_T=x-B
  \label{eq:conditional-barrier-deficit}
\end{equation}
is Gaussian with mean \(x\mathbf1\), conditioned to be nonnegative on
\([0,T]\).

We first make the two conditional-kernel orders precise.  Fix a finite
\(S\subset\mathbb R_+\), put \(S_+:=S\setminus\{0\}\) and
\(r_S:=\max(S\cup\{0\})\), and consider \(T'>T\geq r_S\).  For all
sufficiently large \(n\), set
\[
  D_n(T):=\bigl(2^{-n}\mathbb Z\cap(0,T]\bigr)\cup S_+,
  \qquad
  D_n(T'):=\bigl(2^{-n}\mathbb Z\cap(0,T']\bigr)\cup S_+.
\]
These finite nonzero sets are nested in both \(n\) and the horizon.  Fix
\(x>0\), put \(Y^{(x)}_t=x-B_t\), and condition the vector on
\(D_n(T')\) first on \(Y^{(x)}\geq0\) at the sites of \(D_n(T)\).  The
old and enlarged orthant normalisers are strictly positive by finite Gaussian
nondegeneracy.  The resulting old-wall law is MTP\(_2\) and associated by
Lemma~\ref{lem:finite-wall-mtp2-closure}.  Both a bounded increasing test of
\(Y^{(x)}_{S_+}\) and the indicator of nonnegativity on the added sites are
increasing.  Association, followed by division by the positive added-event
normaliser, shows that conditioning on the added sites raises the test
expectation.  If there are no added sites the two expectations agree.

As \(n\to\infty\), the two finite wall events decrease to the hard-wall
events on \([0,T]\) and \([0,T']\).  Their probabilities are positive because
\(x>0\), and the corresponding conditional path laws converge in total
variation.  The coordinate at zero is the deterministic value \(x\).  Hence
\[
  \mathcal L\bigl(X_T(S)\mid E_T=x\bigr)
  \leq_{\mathrm{st}}
  \mathcal L\bigl(X_{T'}(S)\mid E_{T'}=x\bigr).
\]

Now fix \(T\geq r_S\) and compare \(0<x<y\).  On the finite set
\(F_n:=D_n(T)\), let \(\mu_{n,x}\) be the law of
\(Y^{(x)}_{F_n}\) restricted to the positive orthant.  If
\(\mathsf P_n\) is the full precision, then
\(\mathsf P_n\mathbf1_{F_n}\geq0\) by
Lemma~\ref{lem:submarkov-fbm-precision}, and
\[
  \frac{\mathrm d\mu_{n,y}}{\mathrm d\mu_{n,x}}(z)
  =c_{n,y,x}\exp\left\{
    (y-x)(\mathsf P_n\mathbf1_{F_n})^{\mathsf T}z
  \right\},
  \qquad z\geq0,
\]
for a positive constant.  This ratio is coordinatewise increasing.  The
law \(\mu_{n,x}\) is associated by
Lemma~\ref{lem:finite-wall-mtp2-closure}; applying association to a bounded
increasing test and the ratio truncated at level \(m\), and then using
monotone convergence exactly as in
\eqref{eq:truncated-common-shift-order}, proves
\(\mu_{n,x}\leq_{\mathrm{st}}\mu_{n,y}\).  Projection to \(S_+\), insertion
of the deterministic coordinates \(Y^{(x)}_0=x\leq y=Y^{(y)}_0\), and the
same total-variation continuum passage give
\[
  \mathcal L\bigl(X_T(S)\mid E_T=x\bigr)
  \leq_{\mathrm{st}}
  \mathcal L\bigl(X_T(S)\mid E_T=y\bigr).
\]

Let \(f\) be a bounded coordinatewise increasing Borel function on
\(\mathbb R_+\times\mathbb R_+^S\), and, for \(x>0\), write
\[
  K_{T,S}(x,\cdot)
  :=\mathcal L\bigl(X_T(S)\mid E_T=x\bigr)
\]
and put
\begin{equation}
  k_{T,S}(x)
  :=\EE\bigl[f(x,X_T(S))\mid E_T=x\bigr].
  \label{eq:barrier-kernel-monotonicity}
\end{equation}
If \(0<x<y\), the barrier order and monotonicity of \(f\) in its first
coordinate give
\[
  \int f(y,z)\,K_{T,S}(y,\mathrm dz)
  \geq \int f(x,z)\,K_{T,S}(y,\mathrm dz)
  \geq \int f(x,z)\,K_{T,S}(x,\mathrm dz),
\]
Hence \(k_{T,S}\) is nondecreasing on \((0,\infty)\).  Define
\(k_{T,S}(0):=\inf_{x>0}k_{T,S}(x)\).  This is a bounded nondecreasing Borel
extension, and its value at zero does not change an integral because
\(\beta_T(\{0\})=0\).  The horizon order also gives
\(k_{T',S}\geq k_{T,S}\) pointwise, including at zero.  Consequently,
Corollary~\ref{cor:monotone-exponential-barriers} gives the explicit
two-stage comparison
\[
\begin{split}
  \EE f(E_{T'},X_{T'}(S))
  &=\int k_{T',S}\,\mathrm d\beta_{T'}
   \geq\int k_{T,S}\,\mathrm d\beta_{T'}\\
  &\geq\int k_{T,S}\,\mathrm d\beta_T
   =\EE f(E_T,X_T(S)).
\end{split}
\]
Thus the finite-dimensional laws are stochastically increasing once the
horizon contains \(S\).

We next prove joint path tightness.  The barriers are tight by
\eqref{eq:uniform-exponential-barrier-tightness}.  Fix \(R<\infty\) and
\(\eta>0\).  Choose \(A\) so that
\(\sup_T\beta_T((A,\infty))<\eta\).  By
Lemma~\ref{lem:shifted-hard-wall-gate2}, there is a compact
\(K_{R,\eta,A}\subset C([0,R])\) such that
\[
  \sup_{T\geq R}\sup_{0<x\leq A}
  \PP_{T,x}\bigl(x-B\notin K_{R,\eta,A}\bigr)<\eta.
\]
The barrier disintegration and \(\beta_T(\{0\})=0\) therefore give
\[
  \sup_{T\geq R}
  \widehat{\PP}_T\bigl(X_T|_{[0,R]}\notin K_{R,\eta,A}\bigr)
  <2\eta.
\]
Thus the restrictions are asymptotically tight on every compact interval.
To make the local-to-global passage precise, fix an arbitrary sequence
\(T_j\to\infty\) and \(\varepsilon>0\).  For each integer \(m\geq1\), the
displayed estimate makes the tail with \(T_j\geq m\) tight on \(C([0,m])\).
The remaining indices are finite in number and their laws are individually
tight, so there is a compact \(K_m\subset C([0,m])\) such that
\[
  \sup_j\widehat{\PP}_{T_j}
  \bigl(X_{T_j}|_{[0,m]}\notin K_m\bigr)
  <\varepsilon 2^{-m-1}.
\]
If \(\pi_m\) denotes restriction to \([0,m]\), then
\[
  K:=\bigcap_{m\geq1}\pi_m^{-1}(K_m)
\]
is compact in \(C_{\mathrm{loc}}(\mathbb R_+)\): every sequence in \(K\)
has, by successive compact subsequences and a diagonal extraction, a limit
uniformly on every \([0,m]\), and \(K\) is closed.  The union bound gives
\(\sup_j\widehat{\PP}_{T_j}(X_{T_j}\notin K)<\varepsilon/2\).  The constant
extension is irrelevant to each tail estimate once \(T_j\geq m\).
Finally, \eqref{eq:uniform-exponential-barrier-tightness} supplies
\(A<\infty\) with \(\sup_j\beta_{T_j}((A,\infty))<\varepsilon/2\).
Thus \([0,A]\times K\) proves tightness of the joint laws along every
sequence \(T_j\to\infty\).

It remains to identify one limit.  For every finite \(S\) and every
nonnegative vector \(\theta=(\theta_0,(\theta_s)_{s\in S})\), stochastic
monotonicity and nonnegativity of every coordinate make
\[
  \EE\exp\left\{-\theta_0E_T-
    \sum_{s\in S}\theta_sX_T(s)\right\}
\]
nonincreasing for all sufficiently large \(T\), so it has a limit.  Tightness
shows that every sequence tending to infinity has a weakly convergent
subsequence, and every such subsequential finite-dimensional limit has these
same Laplace transforms.  Laplace transforms determine probability laws on
\(\mathbb R_+^{1+|S|}\), so the finite-dimensional limit is unique.

Finally, take \(S\) through the finite subsets of
\(\mathbb Q_+\).  Evaluation at each fixed time is continuous on
\(C_{\mathrm{loc}}\), so every subsequential path limit has the unique joint
laws of \((E,(X_s)_{s\in S})\).  The Borel sigma-field of
\(C_{\mathrm{loc}}(\mathbb R_+)\) is generated by evaluations on
\(\mathbb Q_+\), and a continuous path is determined by those values.
Therefore all subsequential limits agree.  Tightness and the usual
contradiction-sequence argument prove the full convergence
\eqref{eq:barrier-deficit-limit}; its common limit is \(\Gamma_H\).
\end{proof}

\begin{lemma}
\label{lem:entrance-unique-zero}
For every \(u\geq0\), under \(K_u\) the centred deficit path
\begin{equation}
  G^{(u)}(t):=-W_t,\qquad -u\leq t<\infty,
  \label{eq:entrance-centered-deficit}
\end{equation}
satisfies
\begin{equation}
  G^{(u)}(0)=0,
  \qquad
  G^{(u)}(t)>0\quad\text{for every }t\neq0
  \quad\text{almost surely}.
  \label{eq:entrance-unique-zero}
\end{equation}
\end{lemma}

\begin{proof}
It is enough to work under \(Q^u\).  Indeed, \(K_u\) is the pushforward of
the endpoint-tilted law with density \(e^{-V^u}/z(u)\) relative to \(Q^u\);
this density is bounded by \(1/z(u)<\infty\) and hence sends \(Q^u\)-null
events to \(K_u\)-null events.  At \(u=0\) the density is identically one.

Let \(I\subset[-u,\infty)\setminus\{0\}\) be a compact interval of positive
length.  Write
\begin{equation*}
  \mathbb D_u^\circ:=
  \begin{cases}
    \mathbb D_u,&u>0,\\
    \mathbb D\cap(0,\infty),&u=0.
  \end{cases}
\end{equation*}
For the unconditioned reflected process \(G=-W\), put
\(Z_t:=G_t/|t|^H\).  This is a centred unit-variance Gaussian process.  If
\(s,t\in I\), then
\begin{equation*}
\begin{split}
  \|Z_s-Z_t\|_2
  &\leq \frac{\|G_s-G_t\|_2}{|s|^H}
    +\|G_t\|_2\left|\frac1{|s|^H}-\frac1{|t|^H}\right|\\
  &\leq C_I|s-t|^H.
\end{split}
\end{equation*}
Here \(I\) is bounded away from zero, so \(x\mapsto |x|^{-H}\) is
Lipschitz on \(I\), and the last bound absorbs the resulting
\(|s-t|\)-term.  The corresponding one-dimensional entropy integral is
finite.  Dudley's Gaussian entropy bound
\parencite[Theorem~1.3.3]{adlerTaylor2007randomFields} therefore gives
\begin{equation}
  \sup_{\substack{S\subset \mathbb D_u^\circ\cap I\\
                  0<|S|<\infty}}
  \EE\sup_{t\in S}(-Z_t)<\infty.
  \label{eq:unique-zero-entropy}
\end{equation}

Fix now a nonempty finite set \(S\subset\mathbb D_u^\circ\cap I\), and put
\begin{equation}
  M_S:=\min_{t\in S}Z_t.
  \label{eq:unique-zero-standardization}
\end{equation}
Lemma~\ref{lem:fbm-compact-positivity-support} gives positive probability
to the continuous strict-positivity event on \(I\).  Therefore, uniformly
in such \(S\),
\begin{equation}
  \PP(M_S\geq0)\geq p_I>0.
  \label{eq:unique-zero-positive-baseline}
\end{equation}
For \(\varepsilon>0\), since \(-M_S=\sup_{t\in S}(-Z_t)\),
\eqref{eq:unique-zero-entropy} and
\eqref{eq:gaussian-supremum-anticoncentration} give
\begin{equation}
  \PP(0\leq M_S\leq\varepsilon)\leq C_I\varepsilon.
  \label{eq:unique-zero-anticoncentration}
\end{equation}

Let \(F\subset\mathbb D_u^\circ\) be any finite set containing \(S\), and
write \(A_T:=\{G_t\geq0:t\in T\}\).  The vector \(G_F\) is nondegenerate,
so both \(A_S\) and \(A_F\) have positive probability.  By
Lemma~\ref{lem:finite-wall-mtp2-closure}, the law of \(G_F\) conditioned on
\(A_S\) is associated.  Under this law, \(A_{F\setminus S}\) and
\(\{M_S>\varepsilon\}\) are increasing events.  Association and
complementation therefore give
\begin{equation}
\begin{split}
  \PP(M_S\leq\varepsilon\mid A_F)
  &\leq \PP(M_S\leq\varepsilon\mid A_S)\\
  &=\frac{\PP(0\leq M_S\leq\varepsilon)}
          {\PP(M_S\geq0)}
  \leq C_I'\varepsilon.
\end{split}
  \label{eq:unique-zero-conditioned-bound}
\end{equation}
The deterministic coordinate at zero is absent throughout this finite
Gaussian calculation.

Keep \(S\) fixed and let \(F\) increase along a cofinal sequence in
\(\mathbb D_u^\circ\) containing \(S\).  The cofinal constructions of
\(Q^u\) and \(Q^0\), respectively, give weak convergence of these finite-wall
path laws to \(Q^u\).  Since \(\{M_S<\varepsilon\}\) is open under the fixed
finite evaluation map, the portmanteau theorem and
\eqref{eq:unique-zero-conditioned-bound} yield
\begin{equation}
  Q^u(M_S<\varepsilon)\leq C_I'\varepsilon.
  \label{eq:unique-zero-fixed-grid-limit}
\end{equation}
Now choose increasing finite sets
\(S_m\subset\mathbb D_u^\circ\cap I\) whose union is dense in \(I\).
Continuity of \(Z\), the wall support of \(Q^u\), and continuity from below
for the events in \eqref{eq:unique-zero-fixed-grid-limit} give
\begin{equation}
  Q^u\left(\min_{t\in I}Z_t=0\right)=0,
  \qquad
  Q^u\left(\min_{t\in I}G^{(u)}(t)=0\right)=0.
  \label{eq:unique-zero-compact}
\end{equation}
Indeed, the probability of \(\{\min_I Z<\varepsilon\}\) is at most
\(C_I'\varepsilon\), and \(|t|^H\) is bounded above and away from zero on
\(I\).

If \(u>0\), the intervals \([-u,-u/m]\), \(m\geq2\), cover the negative
part of the domain, while \([1/m,m]\), \(m\geq1\), cover its positive part.
For \(u=0\), only the latter family is needed.  Applying
\eqref{eq:unique-zero-compact} to this countable cover, and using the
deterministic pin \(G^{(u)}(0)=0\), proves
\eqref{eq:entrance-unique-zero}.
\end{proof}

\begin{theorem}
\label{thm:unique-entrance-mixture}
There is a unique probability measure \(\pi_H\) on \(\mathbb R_+\) such
that, with
\begin{equation}
  \mu_H:=\int_{[0,\infty)}K_u\,\pi_H(\mathrm du),
  \label{eq:unique-entrance-mixture}
\end{equation}
one has
\begin{equation}
  \mathcal L_{\widehat{\PP}_a}
  (U_a,V_a,\overline L_a,\overline R_a)
  \Longrightarrow\mu_H,
  \qquad a\to\infty.
  \label{eq:pb1-marked-convergence}
\end{equation}
Moreover, let \(a_n\to\infty\) be arbitrary and choose the strictly
increasing integers \(r_n\), with \(h_n:=2^{-r_n}\leq2^{-n}\), far enough
along the fixed-row approximations that both
\eqref{eq:relative-tilted-diagonal} and
\eqref{eq:scalar-tilted-grid-diagonal} hold, as in
Theorem~\ref{thm:subsequential-boundary-entrance-mixture}.  Then the
associated left-horizon mixing measures satisfy
\begin{equation}
  \pi_n\Longrightarrow\pi_H.
  \label{eq:unique-grid-mixing-limit}
\end{equation}
\end{theorem}

\begin{proof}
Let \(\eta\) denote the unit exponential law.  We first make the
reconstruction and its topology precise.  Define the closed compatible
subset \(\mathcal C\subset\mathcal Y\) by
\begin{equation*}
\begin{split}
  \mathcal C:=\{(u,v,\ell,r):{}&
    \ell(0)=r(0)=0,\quad v=\ell(u),\\
    &\ell(s)=\ell(s\wedge u)\text{ for every }s\geq0\}.
\end{split}
\end{equation*}
Every finite-horizon marked law and every \(K_u\) is supported on
\(\mathcal C\).  The set is closed because evaluation at a moving time and
the map that stops a path at a moving time are continuous in the product
local-uniform topology.  We therefore regard all these laws as laws on the
closed Polish subspace \(\mathcal C\).  For
\(((u,v,\ell,r),d)\in\mathcal C\times\mathbb R_+\), set
\begin{equation}
\begin{aligned}
  E&:=v+d,\\
  X(t)&:=d+\ell(u-t),&&0\leq t\leq u,\\
  X(t)&:=d+r(t-u),&&t\geq u,
\end{aligned}
\qquad
  \Phi((u,v,\ell,r),d):=(E,X).
  \label{eq:marked-to-barrier-reconstruction}
\end{equation}
Because both arms vanish at zero, the path formula can equivalently be
written
\begin{equation}
  X(t)=d+\ell((u-t)^+)+r((t-u)^+).
  \label{eq:continuous-reconstruction-formula}
\end{equation}
This also proves continuity of \(\Phi\) for the product local-uniform
topologies.  Indeed, suppose
\begin{equation*}
  (u_n,v_n,\ell_n,r_n,d_n)\longrightarrow(u,v,\ell,r,d).
\end{equation*}
On each fixed \([0,A]\), the two stopped arguments in
\eqref{eq:continuous-reconstruction-formula} differ from their limits by at
most \(|u_n-u|\).  Local-uniform convergence of the arms and uniform
continuity of \(\ell,r\) on the resulting compact argument intervals give
\(\sup_{t\leq A}|X_n(t)-X(t)|\to0\), while \(E_n\to E\).

Now take an arbitrary sequence \(a_n\to\infty\), choose the compatible grid
diagonal of Proposition~\ref{prop:relative-tilted-grid-diagonal}, and then
take any further subsequence supplied by
Theorem~\ref{thm:subsequential-boundary-entrance-mixture}.  Along it the
continuous marked laws converge to
\begin{equation}
  \mu_\pi=\int_{[0,\infty)}K_u\,\pi(\mathrm du)
  \label{eq:generic-subsequential-mixture}
\end{equation}
for some probability measure \(\pi\).  Adjoining an independent variable
with law \(\eta\) gives convergence of the product laws to
\(\mu_\pi\otimes\eta\).  In the long-horizon representation, applying
\(\Phi\) to the finite-horizon mark and that independent variable gives
exactly \((E_{a_n},X_{a_n})\) from
\eqref{eq:barrier-deficit-process}.  Indeed, there
\(u=\sigma_{a_n}\), \(v=M_{a_n}\), and
\begin{equation*}
  \ell(u-t)=M_{a_n}-B_t\quad(0\leq t\leq u),
  \qquad
  r(t-u)=M_{a_n}-B_t\quad(u\leq t\leq a_n).
\end{equation*}
The constant extension of \(r\) gives the constant extension of \(X_{a_n}\)
after \(a_n\).  Hence the continuous mapping theorem and
Proposition~\ref{prop:infinite-barrier-deficit-law} give
\begin{equation}
  \Phi_*(\mu_\pi\otimes\eta)=\Gamma_H.
  \label{eq:common-reconstructed-image}
\end{equation}

We next construct a single measurable inverse.  For
\(f\in C_{\mathrm{loc}}(\mathbb R_+)\), put
\begin{equation}
  m(f):=\inf_{n\geq1}\min_{0\leq t\leq n}f(t).
  \label{eq:measurable-global-infimum}
\end{equation}
Each compact minimum is continuous in the local-uniform topology, so
\(m\), with values allowed in \([ -\infty,\infty)\), is Borel.  Also put
\begin{equation}
  \alpha(f)
  :=\inf\left\{
    q\in\mathbb Q_+:
    \min_{0\leq t\leq q}f(t)=m(f)
  \right\},
  \qquad \inf\varnothing:=\infty.
  \label{eq:measurable-unique-argmin}
\end{equation}
This is an extended-valued Borel map because the infimum is over a countable
set of Borel equality events.  Whenever \(f\) has a finite attained global
minimum at first time \(t_*\), every rational \(q<t_*\) has
\(\min_{[0,q]}f>m(f)\), whereas every rational \(q\geq t_*\) has
\(\min_{[0,q]}f=m(f)\).  Thus \(\alpha(f)=t_*\); in particular it is the
minimiser when that minimiser is unique.

On the Borel set where
\(m(f)\in\mathbb R\), \(\alpha(f)<\infty\), and
\(0\leq m(f)\leq e\), define
\begin{equation}
\begin{aligned}
  d_f&:=m(f),& u_f&:=\alpha(f),& v_f&:=e-m(f),\\
  \ell_f(s)&:=f((u_f-s)^+)-m(f),& &s\geq0,\\
  r_f(s)&:=f(u_f+s)-m(f),& &s\geq0,
\end{aligned}
  \label{eq:measurable-arm-recovery}
\end{equation}
and let
\begin{equation*}
  \Psi(e,f):=((u_f,v_f,\ell_f,r_f),d_f).
\end{equation*}
Give \(\Psi\) any fixed value off this set.  The maps
\begin{equation*}
\begin{split}
  (f,u,d)&\longmapsto
    \bigl(f((u-s)^+)-d\bigr)_{s\geq0},\\
  (f,u,d)&\longmapsto
    \bigl(f(u+s)-d\bigr)_{s\geq0}
\end{split}
\end{equation*}
are continuous into \(C_{\mathrm{loc}}(\mathbb R_+)\): on a fixed
\(s\)-compact, local-uniform convergence of \(f\) and its uniform
continuity on one larger compact control the moving arguments.  Thus
\(\Psi\) is Borel.

Lemma~\ref{lem:entrance-unique-zero} now supplies the essential inverse
property.  Under \(K_u\), the arm definitions and
\eqref{eq:marked-to-barrier-reconstruction} give
\begin{equation*}
  X(t)-d=G^{(u)}(t-u),
  \qquad t\geq0.
\end{equation*}
Consequently Lemma~\ref{lem:entrance-unique-zero} gives
\begin{equation}
  X(u)=d,
  \qquad X(t)>d\quad(t\neq u).
  \label{eq:barrier-to-marked-inverse}
\end{equation}
No behaviour at infinity is needed: the value \(d\) is attained at \(u\)
and nowhere else.  Hence \(m(X)=d\) and \(\alpha(X)=u\).  Moreover,
\begin{equation*}
  E-d=v,
  \qquad
  X((u-s)^+)-d=\ell(s),
  \qquad
  X(u+s)-d=r(s).
\end{equation*}
For \(s>u\), the middle identity uses both \(X(0)-d=\ell(u)\) and the
constant extension of \(\ell\); it also covers \(u=0\), where
\(\ell\equiv0\).  All identities in this recovery statement are Borel
events by the maps just constructed.  Lemma~\ref{lem:entrance-unique-zero}
gives them \(K_u\)-almost surely for every \(u\geq0\); integrating against
\(\pi\) and then \(\eta\) therefore gives
\begin{equation}
  \Psi\circ\Phi=\operatorname{id}
  \quad
  (\mu_\pi\otimes\eta)\text{-almost surely}.
  \label{eq:reconstruction-left-inverse}
\end{equation}
Combining \eqref{eq:common-reconstructed-image} and
\eqref{eq:reconstruction-left-inverse} yields
\begin{equation}
  \mu_\pi\otimes\eta
  =\Psi_*\Gamma_H.
  \label{eq:marked-law-from-barrier-law}
\end{equation}
Thus every possible subsequential marked law \(\mu_\pi\) is the same.
Since \(K_u\) has deterministic first coordinate \(u\), the first marginal
of \(\mu_\pi\) is exactly \(\pi\); consequently every possible mixing law
is the same probability measure, which we denote by \(\pi_H\), and the
common marked law is
\(\mu_H=\int K_u\,\pi_H(\mathrm du)\).

It remains to compile the two full convergence statements.  Existence of a
subsequential pair \((\pi_H,\mu_H)\) follows by applying
Theorem~\ref{thm:subsequential-boundary-entrance-mixture}, for example, to
\(a_n=n\).
For any sequence \(a_n\to\infty\), that theorem makes its continuous marked
laws relatively compact, through a compatible grid diagonal, and every
subsequence has a further subsequence whose limit is the just identified
\(\mu_H\).  The subsequence criterion therefore proves
\eqref{eq:pb1-marked-convergence}.

Finally fix an arbitrary sequence and selected diagonal from the second
part of the statement.  The laws \((\pi_n)\) are tight by
Theorem~\ref{thm:subsequential-boundary-entrance-mixture}.  From any
subsequence, that theorem supplies a further subsequence such that
\(\pi_n\Rightarrow\pi\) and the continuous marked laws converge
to \(\int K_u\,\pi(\mathrm du)\).  The convergence already proved forces
this marked limit to be \(\mu_H\), and comparison of first marginals gives
\(\pi=\pi_H\).  Thus every subsequence of \((\pi_n)\) has a further
subsequence converging to \(\pi_H\), which proves
\eqref{eq:unique-grid-mixing-limit}.
\end{proof}

\begin{corollary}
\label{cor:entrance-mixing-scale-fixed-point}
Put \(\theta:=1-H\).  Then
\begin{equation}
  z(u)=\EE_{Q^1}e^{-u^HV^1},
  \label{eq:entrance-normalizer-scaling}
\end{equation}
and, on \((0,\infty)\),
\begin{equation}
  \pi_H(\mathrm du)=c_H^{\mathrm{mix}}z(u)u^{\theta-1}\,\mathrm du,
  \qquad
  (c_H^{\mathrm{mix}})^{-1}
  =\int_0^\infty z(u)u^{\theta-1}\,\mathrm du.
  \label{eq:explicit-entrance-mixing-law}
\end{equation}
Moreover, for every \(c>0\),
\begin{equation}
  \frac{\mathcal Z(ca)}{\mathcal Z(a)}
  \longrightarrow c^{-\theta}.
  \label{eq:normalizer-regular-variation}
\end{equation}
\end{corollary}

\begin{proof}
For \(c>0\), define the marked dilation
\begin{equation}
  \mathscr S_c(u,v,\ell,r)
  :=\left(cu,c^Hv,
     \bigl(c^H\ell(s/c)\bigr)_{s\geq0},
     \bigl(c^Hr(s/c)\bigr)_{s\geq0}\right).
  \label{eq:marked-scale-dilation}
\end{equation}
This map is continuous for the product topology of \(\mathcal Y\).  Write
\(Y_a:=(U_a,V_a,\overline L_a,\overline R_a)\).  For the bounded weak-limit
passage, now fix \(c\geq1\) and put
\begin{equation*}
  w_c(u,v,\ell,r):=e^{-(c^H-1)v}.
\end{equation*}
Thus \(w_c\in C_b(\mathcal Y)\) and \(0<w_c\leq1\).  Self-similarity and
the definition of the tilted law give the exact finite-scale identity
\begin{equation}
  \mathcal L_{\widehat{\PP}_{ca}}(Y_{ca})
  =
  (\mathscr S_c)_*
  \left[
    \frac{w_c\,\mathcal L_{\widehat{\PP}_a}(Y_a)}
         {\widehat{\EE}_a w_c(Y_a)}
  \right].
  \label{eq:finite-scale-tilt-fixed-point}
\end{equation}
Its denominator is exactly
\begin{equation}
  \widehat{\EE}_a w_c(Y_a)
  =\frac{\EE e^{-c^HM_a}}{\mathcal Z(a)}
  =\frac{\mathcal Z(ca)}{\mathcal Z(a)}.
  \label{eq:finite-scale-tilt-denominator}
\end{equation}
By Theorem~\ref{thm:unique-entrance-mixture}, bounded weak convergence in
\eqref{eq:finite-scale-tilt-fixed-point} gives
\begin{equation}
\begin{split}
  h_c(u)&:=\int_{\mathcal Y}w_c\,\mathrm dK_u,\\
  d_c
  &:=\int_{\mathcal Y}w_c\,\mathrm d\mu_H\\
  &=\int_{[0,\infty)}h_c(u)\,\pi_H(\mathrm du)
   =\lim_{a\to\infty}\frac{\mathcal Z(ca)}{\mathcal Z(a)},
  \qquad c\geq1.
\end{split}
  \label{eq:normalizer-ratio-limit-ge-one}
\end{equation}
The function \(h_c\) is continuous by the weak continuity of
\(u\mapsto K_u\), and \(0<h_c\leq1\).  The number \(d_c\) is strictly
positive, so division in the limit is legitimate.  The limiting marked law
satisfies
\begin{equation}
  \mu_H=(\mathscr S_c)_*\left(\frac{w_c\mu_H}{d_c}\right),
  \qquad c\geq1.
  \label{eq:limiting-marked-scale-fixed-point}
\end{equation}
Taking the lifetime marginal shows, for every
\(f\in C_b([0,\infty))\), that
\begin{equation}
  \int f(u)\,\pi_H(\mathrm du)
  =d_c^{-1}\int f(cu)h_c(u)\,\pi_H(\mathrm du),
  \qquad c\geq1.
  \label{eq:mixing-scale-map}
\end{equation}

We next identify \(h_c\) without assuming scale invariance of the dyadic
entrance construction.  The right-hand side of
\eqref{eq:limiting-marked-scale-fixed-point} is the mixture with source
measure
\begin{equation*}
  \widehat\pi_c(\mathrm du):=d_c^{-1}h_c(u)\pi_H(\mathrm du)
\end{equation*}
and component at \(u\)
\begin{equation*}
  \widetilde K_{c,u}
  :=(\mathscr S_c)_*\left(\frac{w_cK_u}{h_c(u)}\right).
\end{equation*}
That component has deterministic lifetime \(cu\), and
\eqref{eq:mixing-scale-map} says that
\((D_c)_*\widehat\pi_c=\pi_H\), where \(D_c(u):=cu\).  Uniqueness of regular
conditional laws given the lifetime coordinate in the two disintegrations
of \(\mu_H\) therefore gives
\begin{equation}
  K_{cu}=(\mathscr S_c)_*\left(\frac{w_cK_u}{h_c(u)}\right)
  \quad\text{for \(\pi_H\)-almost every \(u\)}.
  \label{eq:entrance-kernel-scale-cocycle-ae}
\end{equation}
Here \(\widehat\pi_c\) and \(\pi_H\) have the same null sets because
\(h_c>0\).

The endpoint tilt itself gives the finite moment
\begin{equation*}
  \EE_{K_u}e^V=\frac1{z(u)},
  \qquad u\geq0.
\end{equation*}
Apply this nonnegative test to
\eqref{eq:entrance-kernel-scale-cocycle-ae}.  Since the output height is
\(c^HV\), for \(\pi_H\)-almost every \(u\),
\begin{equation*}
  \frac1{z(cu)}
  =\frac{\EE_{K_u}\!\left[e^{c^HV}e^{-(c^H-1)V}\right]}{h_c(u)}
  =\frac1{h_c(u)z(u)}.
\end{equation*}
Consequently,
\begin{equation}
  h_c(u)=\frac{z(cu)}{z(u)}
  \quad\text{for \(\pi_H\)-almost every \(u\)},
  \qquad c\geq1.
  \label{eq:entrance-tilt-scaling-ae}
\end{equation}

On \((0,\infty)\), define the Borel measure
\begin{equation*}
  \rho(\mathrm du):=z(u)^{-1}\pi_H(\mathrm du).
\end{equation*}
It is Radon because \(z\) is continuous and has a positive minimum on every
compact subset of \((0,\infty)\).  If
\(g\in C_c((0,\infty))\), then \(g/z\), extended by zero at the origin,
belongs to \(C_b([0,\infty))\).  Substituting this function in
\eqref{eq:mixing-scale-map} and using
\eqref{eq:entrance-tilt-scaling-ae} gives
\begin{equation*}
  \int g(cu)\,\rho(\mathrm du)=d_c\int g(u)\,\rho(\mathrm du).
\end{equation*}
Consequently, as an equality of Radon measures,
\begin{equation}
  (D_c)_*\rho=d_c\rho,
  \qquad D_c(u):=cu,\qquad c\geq1.
  \label{eq:homogeneous-mixing-measure}
\end{equation}

Formula \eqref{eq:normalizer-ratio-limit-ge-one} and dominated convergence
applied to \(w_c\) under \(\mu_H\) give
\begin{equation*}
  0<d_c\leq1,
  \qquad c\longmapsto d_c\ \text{continuous on }[1,\infty).
\end{equation*}
For \(c,d\geq1\), factor the finite ratios as
\begin{equation*}
  \frac{\mathcal Z(cda)}{\mathcal Z(a)}
  =\frac{\mathcal Z(cd a)}{\mathcal Z(da)}
   \frac{\mathcal Z(da)}{\mathcal Z(a)}.
\end{equation*}
Letting \(a\to\infty\) proves \(d_{cd}=d_cd_d\).  Therefore the continuous
additive function
\begin{equation*}
  t\longmapsto-\log d_{e^t},
  \qquad t\geq0,
\end{equation*}
equals \(\alpha t\) for some \(\alpha\geq0\).  Define
\(d_c:=d_{1/c}^{-1}\) when \(0<c<1\).  We have proved
\begin{equation}
  d_c=c^{-\alpha},
  \qquad c>0.
  \label{eq:normalizer-power-cocycle}
\end{equation}
Applying \((D_c)_*\) to \eqref{eq:homogeneous-mixing-measure} with
\(1/c\geq1\) extends that measure identity to \(0<c<1\).  The actual
normaliser limit at such a scale follows without an unbounded change of
measure: with \(b=ca\),
\begin{equation*}
  \frac{\mathcal Z(ca)}{\mathcal Z(a)}
  =\left(
      \frac{\mathcal Z((1/c)b)}{\mathcal Z(b)}
    \right)^{-1}
  \longrightarrow d_{1/c}^{-1}=d_c.
\end{equation*}
Thus
\begin{equation}
  \lim_{a\to\infty}\frac{\mathcal Z(ca)}{\mathcal Z(a)}
  =d_c=c^{-\alpha},
  \qquad c>0.
  \label{eq:normalizer-ratio-limit}
\end{equation}

To identify \(\rho\), let \(\nu:=(\log)_*\rho\).  This Radon measure is
locally finite because logarithms of compact intervals are compact.  The now
all-scale identity \eqref{eq:homogeneous-mixing-measure}, with \(c=e^s\), is
equivalent to
\begin{equation}
  \nu(A+s)=e^{\alpha s}\nu(A),
  \qquad s\in\mathbb R.
  \label{eq:log-homogeneous-measure}
\end{equation}
It follows directly that the locally finite Radon measure
\begin{equation*}
  \widetilde\nu(\mathrm dy):=e^{-\alpha y}\nu(\mathrm dy)
\end{equation*}
is translation invariant.  If \(\widetilde\nu=0\), set \(C=0\).  Otherwise,
the uniqueness of Haar measure
\parencite[Theorem~2.20]{folland2016abstractHarmonic} gives
\(\widetilde\nu=C\,\mathrm dy\) for some \(C>0\).  Thus in either case
\(C\geq0\), and \(C=0\) remains possible at this stage.  Since \(u=e^y\)
and \(\mathrm dy=\mathrm du/u\),
transforming back yields
\begin{equation}
  \rho(\mathrm du)=C u^{\alpha-1}\,\mathrm du.
  \label{eq:homogeneous-mixing-density}
\end{equation}

The logarithmic persistence theorem of \textcite{molchan1999maximum} gives
\begin{equation*}
  \frac{\log F_a(1)}{\log a}\longrightarrow-\theta.
\end{equation*}
We transfer this exponent to the Laplace normaliser.  Since \(M_a\geq0\),
Tonelli's theorem gives
\begin{equation}
  \mathcal Z(a)
  =\int_0^\infty e^{-x}F_a(x)\,\mathrm dx.
  \label{eq:normalizer-distribution-integral}
\end{equation}
Therefore
\begin{equation}
  \mathcal Z(a)
  \geq\int_1^2e^{-x}F_a(x)\,\mathrm dx
  \geq(e^{-1}-e^{-2})F_a(1).
  \label{eq:normalizer-log-lower-comparison}
\end{equation}
For the reverse bound, fix \(\kappa>\theta\) and put
\(A_a=\kappa\log a\).  Splitting
\eqref{eq:normalizer-distribution-integral} at \(A_a\) gives
\begin{equation}
  \mathcal Z(a)\leq F_a(A_a)+e^{-A_a}.
  \label{eq:normalizer-log-upper-comparison}
\end{equation}
By self-similarity, if \(b_a:=a/A_a^{1/H}\), then
\begin{equation*}
  F_a(A_a)
  =F_{b_a}(1),
\end{equation*}
and
\begin{equation*}
  \frac{\log b_a}{\log a}
  =1-\frac{\log(\kappa\log a)}{H\log a}
  \longrightarrow1.
\end{equation*}
Applying Molchan's logarithmic theorem at time \(b_a\) yields
\begin{equation*}
  \frac{\log F_a(A_a)}{\log a}
  =
  \frac{\log F_{b_a}(1)}{\log b_a}
  \frac{\log b_a}{\log a}
  \longrightarrow-\theta.
\end{equation*}
Moreover, \(e^{-A_a}=a^{-\kappa}\).  Since \(\kappa>\theta\), combining
\eqref{eq:normalizer-log-lower-comparison} and
\eqref{eq:normalizer-log-upper-comparison} proves
\begin{equation}
  \frac{\log\mathcal Z(a)}{\log a}\longrightarrow-\theta.
  \label{eq:normalizer-log-exponent}
\end{equation}
We next compare this exponent with the cocycle exponent without invoking a
Tauberian theorem.  Fix \(q>1\).  By
\eqref{eq:normalizer-ratio-limit},
\begin{equation*}
  \log\frac{\mathcal Z(q^{k+1})}{\mathcal Z(q^k)}
  \longrightarrow-\alpha\log q.
\end{equation*}
Summing and applying Ces\`aro convergence gives
\begin{equation*}
  \frac{\log\mathcal Z(q^n)}{n\log q}\longrightarrow-\alpha.
\end{equation*}
The function \(\mathcal Z\) is nonincreasing.  If
\(q^n\leq a<q^{n+1}\), its logarithm is therefore trapped between the two
adjacent displayed values, while \(\log a=n\log q+O(1)\).  Hence
\begin{equation*}
  \frac{\log\mathcal Z(a)}{\log a}\longrightarrow-\alpha.
\end{equation*}
Comparison with \eqref{eq:normalizer-log-exponent} proves
\(\alpha=\theta\).

Let \(p_0:=\pi_H(\{0\})\).  Evaluating the lifetime measure identity
\eqref{eq:mixing-scale-map} at the invariant set \(\{0\}\), or equivalently
using its extension from continuous tests to Borel sets, gives
\begin{equation*}
  p_0=d_c^{-1}p_0,
  \qquad c>1.
\end{equation*}
Here \(h_c(0)=1\), because the height under \(K_0\) is zero.  But
\(d_c=c^{-\theta}<1\), so \(p_0=0\).  Consequently \(C>0\) in
\eqref{eq:homogeneous-mixing-density}, because \(\pi_H\) is a probability
measure and now has full mass on \((0,\infty)\).  Finally,
\begin{equation*}
  1=C\int_0^\infty z(u)u^{\theta-1}\,\mathrm du.
\end{equation*}
Thus the integral is finite and positive, and taking
\(c_H^{\mathrm{mix}}=C\) proves
\eqref{eq:explicit-entrance-mixing-law}.  Equations
\eqref{eq:normalizer-ratio-limit} and
\eqref{eq:normalizer-power-cocycle}, with \(\alpha=\theta\), prove
\eqref{eq:normalizer-regular-variation}.

It remains to recover \eqref{eq:entrance-normalizer-scaling}.  The density
just proved is strictly positive on every nonempty open subset of
\((0,\infty)\).  For each fixed \(c\geq1\), both sides of
\eqref{eq:entrance-tilt-scaling-ae} are continuous in \(u>0\), so that
identity holds for every \(u>0\).  Likewise, both sides of
\eqref{eq:entrance-kernel-scale-cocycle-ae} depend weakly continuously on
\(u>0\): on the right this uses the bounded function \(w_c\), continuity of
\(u\mapsto K_u\), and positivity of \(h_c\).  Hence
\begin{equation}
  K_{cu}=(\mathscr S_c)_*\left(
    \frac{e^{-(c^H-1)V}K_u}{h_c(u)}
  \right),
  \qquad u>0,\quad c\geq1.
  \label{eq:entrance-kernel-scale-cocycle}
\end{equation}
Let \(P_u:=(\mathscr A_u)_*Q^u\) be the unweighted marked arm law.  The
definition of \(K_u\) is equivalently
\begin{equation*}
  P_u(F)=z(u)\EE_{K_u}[e^VF],
  \qquad F\geq0\ \text{Borel}.
\end{equation*}
Using \eqref{eq:entrance-kernel-scale-cocycle} and
\(h_c(u)=z(cu)/z(u)\), we obtain
\begin{equation*}
\begin{split}
  P_{cu}(F)
  &=\frac{z(cu)}{h_c(u)}
    \EE_{K_u}\!\left[
      e^{c^HV}e^{-(c^H-1)V}F(\mathscr S_cY)
    \right]\\
  &=z(u)\EE_{K_u}[e^VF(\mathscr S_cY)]
   =P_u(F\circ\mathscr S_c).
\end{split}
\end{equation*}
Thus \(P_{cu}=(\mathscr S_c)_*P_u\) for \(u>0\) and \(c\geq1\).  If
\(u\geq1\), take the base horizon to be one and \(c=u\).  If \(0<u<1\),
apply the same identity at base horizon \(u\) with \(c=1/u\), and then apply
the inverse dilation \(\mathscr S_u\).  In either case,
\begin{equation*}
  P_u=(\mathscr S_u)_*P_1,
  \qquad u>0.
\end{equation*}
The height coordinate therefore satisfies
\(V^u\deq u^HV^1\) under the unweighted laws.  Taking \(\EE e^{-V}\)
proves \eqref{eq:entrance-normalizer-scaling} for \(u>0\); at
\(u=0\), both sides equal one.
\end{proof}

\subsection{The growing exponential functional}
\label{subsec:rough-growing-functional}

\begin{lemma}
\label{lem:uniform-block-minimum-repulsion}
Fix \(K<\infty\).  There is \(C_K<\infty\) with the following property.
Let \(\widehat Q_{h,u,L}\) be the underlying endpoint-tilted path law of a
component \(K_{h,u,L}\) from
Proposition~\ref{prop:uniform-entrance-component-compiler}, where
\begin{equation}
  h=2^{-m}\leq1,\quad m\in\mathbb Z_{\geq0},\qquad
  u,L\in h\mathbb Z_{\geq0},\qquad
  0\leq u\leq K.
  \label{eq:block-repulsion-component-range}
\end{equation}
Thus, if \(Q_{h,u,L}\) denotes the unweighted law in
\eqref{eq:finite-entrance-component}, then
\begin{equation}
  \frac{\mathrm d\widehat Q_{h,u,L}}{\mathrm dQ_{h,u,L}}
  =\frac{e^{-G_{-u}}}{\EE_{Q_{h,u,L}}e^{-G_{-u}}}.
  \label{eq:remote-block-endpoint-tilt}
\end{equation}
Write \(G=-W\) on its right arm.  If \(r\geq2\) and
\begin{equation}
  S:=h\mathbb Z\cap[r,\min(2r,L)]
  \label{eq:remote-right-block-grid}
\end{equation}
is nonempty, then, for \(0<\varepsilon\leq1\),
\begin{equation}
  \widehat Q_{h,u,L}\left(
    \min_{s\in S}\frac{G_s}{s^H}\leq\varepsilon
  \right)
  \leq C_K\varepsilon.
  \label{eq:uniform-block-minimum-repulsion}
\end{equation}
\end{lemma}

\begin{proof}
First work under the unweighted hard-wall law \(Q_{h,u,L}\).  Under the original
centred Gaussian law put
\begin{equation}
  Z_s:=\frac{G_s}{s^H},\qquad
  M_S:=\min_{s\in S}Z_s.
\end{equation}
Self-similarity and Lemma~\ref{lem:fbm-compact-positivity-support} give
\begin{equation}
  \PP(M_S\geq0)
  \geq
  \PP(G_t\geq0\text{ for every }t\in[1,2])
  =p_*>0.
  \label{eq:block-positive-baseline}
\end{equation}
The equality uses \(G=-W\), the fact that \(W\) has the law of \(B\), and
\eqref{eq:fixed-block-positive-baseline}.
For the metric estimate, put \(T:=r^{-1}S\subset[1,2]\) and
\(X_t:=-Z_{rt}\), \(t\in T\).  Self-similarity and symmetry identify this
centred unit-variance vector with \((B_t/t^H)_{t\in T}\).  For
\(s,t\in[1,2]\), Minkowski's inequality and the mean value theorem give
\begin{equation*}
\begin{split}
  \left\|
    \frac{B_s}{s^H}-\frac{B_t}{t^H}
  \right\|_2
  &\leq s^{-H}\|B_s-B_t\|_2
    +\|B_t\|_2\,|s^{-H}-t^{-H}|\\
  &\leq C_H\bigl(|s-t|^H+|s-t|\bigr)
   \leq C_H'|s-t|^H.
\end{split}
\end{equation*}
Thus the canonical covering numbers of every such \(T\) are bounded by
those of \([1,2]\) under \(C_H'|s-t|^H\), whose Dudley entropy integral
is finite.  Dudley's bound
\parencite[Theorem~1.3.3]{adlerTaylor2007randomFields} therefore gives
\begin{equation*}
  \sup_T\EE\sup_{t\in T}X_t<\infty.
\end{equation*}
Since \(-M_S=\sup_{t\in T}X_t\), the event
\(\{0\leq M_S\leq\varepsilon\}\) is
\begin{equation*}
  \left\{-\varepsilon\leq\sup_{t\in T}X_t\leq0\right\}
  \subseteq
  \left\{
    \left|\sup_{t\in T}X_t+\frac{\varepsilon}{2}\right|
    \leq\frac{\varepsilon}{2}
  \right\}.
\end{equation*}
The centred unit-variance Gaussian anti-concentration estimate
\eqref{eq:gaussian-supremum-anticoncentration} now yields
\begin{equation}
  \PP(0\leq M_S\leq\varepsilon)\leq C_H\varepsilon.
  \label{eq:block-minimum-anticoncentration}
\end{equation}
The argument includes a proper final subinterval and a singleton \(S\):
the continuum positivity event still implies the finite one, and the same
continuum metric bound controls every restriction.

We make the finite conditioning exact.  Delete the deterministic zero
coordinate and put
\begin{equation*}
  F:=\bigl(h\mathbb Z\cap[-u,L]\bigr)\setminus\{0\},
  \qquad O:=F\setminus S,
\end{equation*}
with
\begin{equation*}
  A_S:=\{G_s\geq0:s\in S\},
  \qquad A_O:=\{G_t\geq0:t\in O\}.
\end{equation*}
The sites in \(F\) are distinct and nonzero, so \(G_F\) is nondegenerate
and all wall probabilities used below are positive.  Since
\(F=S\mathbin{\dot\cup}O\) and \(G_0=0\) deterministically,
\begin{equation*}
  Q_{h,u,L}=\PP(\,\cdot\mid A_S\cap A_O).
\end{equation*}
Under
\(\PP(\,\cdot\mid A_S)\),
Lemma~\ref{lem:finite-wall-mtp2-closure} gives association.  Both
\(A_O\) and \(\{M_S>\varepsilon\}\) are increasing.  Association,
followed by division by \(\PP(A_O\mid A_S)>0\), gives
\begin{equation*}
  \PP(M_S>\varepsilon\mid A_S\cap A_O)
  \geq \PP(M_S>\varepsilon\mid A_S).
\end{equation*}
Complementing and using \(M_S\geq0\) on \(A_S\) therefore gives
\begin{equation}
\begin{split}
  Q_{h,u,L}(M_S\leq\varepsilon)
  &\leq \PP(M_S\leq\varepsilon\mid A_S)\\
  &=\frac{\PP(0\leq M_S\leq\varepsilon)}
          {\PP(M_S\geq0)}
   \leq C_H\varepsilon.
\end{split}
  \label{eq:unweighted-block-repulsion}
\end{equation}
If \(u=0\), then \(G_{-u}=0\).  If \(u>0\), the endpoint \(-u\)
belongs to \(F\), and the wall constraint gives \(G_{-u}\geq0\)
under \(Q_{h,u,L}\).  Thus in both cases \(e^{-G_{-u}}\leq1\), while
\eqref{eq:pin-mean-bound} gives
\begin{equation*}
  0\leq\EE_{Q_{h,u,L}}G_{-u}
  \leq C_Hu^H\leq C_HK^H.
\end{equation*}
Jensen's inequality therefore gives
\begin{equation}
  \EE_{Q_{h,u,L}}e^{-G_{-u}}
  \geq
  \exp\left\{-\EE_{Q_{h,u,L}}G_{-u}\right\}
  \geq e^{-C_HK^H}.
  \label{eq:block-endpoint-normalizer}
\end{equation}
For the event in \eqref{eq:uniform-block-minimum-repulsion}, its tilted
probability is therefore at most its unweighted probability divided by the
last display.  Combining with
\eqref{eq:unweighted-block-repulsion} proves the claim after enlarging
\(C_K\).  We may also take \(C_K\geq1\); then the same bound is
automatically valid for \(\varepsilon>1\), a form used below.
\end{proof}

\begin{lemma}
\label{lem:continuous-block-minimum}
Fix \(K<\infty\) and \(\kappa\in(0,H)\).  There is
\(C_{K,\kappa}<\infty\) with the following property.  For parameters
\(h,u,L\) satisfying \eqref{eq:block-repulsion-component-range}, under
\(\widehat Q_{h,u,L}\), and for every dyadic \(r=2^j\geq2\) whose
displayed interval is nonempty,
\begin{equation}
\begin{split}
  &\widehat Q_{h,u,L}\left(
    \inf_{r\leq s\leq\min(2r,L)}G_s<r^{H-\kappa}
  \right)\\
  &\qquad\leq
  C_{K,\kappa}\left[
    r^{-\kappa}
    +(1+\sqrt{\log r})r^{-(H-\kappa)}
  \right].
\end{split}
  \label{eq:continuous-block-minimum}
\end{equation}
\end{lemma}

\begin{proof}
At every grid site in the block, require
\begin{equation}
  \frac{G_s}{s^H}>2r^{-\kappa}.
  \label{eq:block-grid-height}
\end{equation}
Apply Lemma~\ref{lem:uniform-block-minimum-repulsion} with
\(\varepsilon=2r^{-\kappa}\).  If this number exceeds one, use the
all-\(\varepsilon\) observation at the end of its proof.  After absorbing
the factor two, the probability of failure is at most
\(C_Kr^{-\kappa}\).  On \eqref{eq:block-grid-height}, every grid value
is strictly larger than
\begin{equation*}
  2r^{-\kappa}r^H=2r^{H-\kappa}.
\end{equation*}

To control the gaps between grid sites, put
\begin{equation*}
  F:=\bigl(h\mathbb Z\cap[-u,L]\bigr)\setminus\{0\}.
\end{equation*}
Under the unweighted hard-wall law, decompose
\begin{equation}
  G=g_F+Z_F,
  \qquad
  g_F:=\EE_{Q_{h,u,L}}G
\end{equation}
into conditional mean and centred fluctuation.  Since \(r\) and \(h\) are
dyadic, \(h':=h/r\) is dyadic, and
\(F':=r^{-1}F\) is another admissible finite dyadic constraint set, with
the rescaled endpoints \(u/r,L/r\) aligned to \(h'\).  Self-similarity
gives
\begin{equation*}
  g_F(rs)=r^Hg_{F'}(s).
\end{equation*}
Applying \eqref{eq:annular-mean-modulus} on \([1,2]\) at distance
\(h/r\), then scaling back, gives
\begin{equation}
  \sup_{\substack{s,t\in[r,2r]\\|s-t|\leq h}}
  |g_F(s)-g_F(t)|
  \leq C_Hh^{2H}r^{-H}.
  \label{eq:remote-block-mean-modulus}
\end{equation}
Likewise, the rescaled centred field satisfies
\((Z_F(rs))_{s\in\mathbb R}\deq(r^HZ_{F'}(s))_{s\in\mathbb R}\)
under the corresponding conditional laws.  The centred chaining bound
\eqref{eq:centered-modulus}, begun at scale \(h/r\), therefore gives
\begin{equation}
  \EE_{Q_{h,u,L}}
  \sup_{\substack{s,t\in[r,2r]\\|s-t|\leq h}}
  |Z_F(s)-Z_F(t)|
  \leq C_Hh^H\bigl(1+\sqrt{\log(r/h)}\bigr).
  \label{eq:remote-block-centered-modulus}
\end{equation}
Both estimates are uniform in the constraint set.  If \(D_F\) denotes the
nonnegative supremum in
\eqref{eq:remote-block-centered-modulus}, then
\eqref{eq:block-endpoint-normalizer} and \(e^{-G_{-u}}\leq1\) give
\begin{equation}
  \EE_{\widehat Q_{h,u,L}}D_F
  =\frac{\EE_{Q_{h,u,L}}[e^{-G_{-u}}D_F]}
          {\EE_{Q_{h,u,L}}e^{-G_{-u}}}
  \leq e^{C_HK^H}\EE_{Q_{h,u,L}}D_F.
  \label{eq:remote-block-tilted-modulus-transfer}
\end{equation}

For all sufficiently large \(r\), the deterministic bound
\eqref{eq:remote-block-mean-modulus} is at most
\(\frac12r^{H-\kappa}\), uniformly in \(h\leq1\), because
\(2H-\kappa\geq H>0\).  To treat the centred bound uniformly in the
mesh, put \(x:=\log(1/h)\geq0\).  Then
\begin{equation}
\begin{split}
  h^H\bigl(1+\sqrt{\log(r/h)}\bigr)
  &\leq e^{-Hx}
    \bigl(1+\sqrt{\log r}+\sqrt{x}\bigr)\\
  &\leq C_H(1+\sqrt{\log r}),
\end{split}
  \label{eq:remote-block-mesh-optimization}
\end{equation}
because \(\sup_{x\geq0}e^{-Hx}\sqrt{x}<\infty\).
Equations \eqref{eq:remote-block-centered-modulus},
\eqref{eq:remote-block-tilted-modulus-transfer}, and
\eqref{eq:remote-block-mesh-optimization}, followed by Markov's inequality
at level \(\frac12r^{H-\kappa}\), give the second term on the right of
\eqref{eq:continuous-block-minimum}.  For the finitely many smaller dyadic
values of \(r\), enlarge \(C_{K,\kappa}\) so that the right-hand side is at least
one.

Finally, both endpoints \(r\) and \(\min(2r,L)\) belong to
\(h\mathbb Z\).  Every point \(s\) of the displayed interval therefore
has a site \(q\in S\) with \(|s-q|\leq h\), including when the interval
is the singleton \(\{r\}\).  Outside the grid-failure event and the
centred-modulus event, the mean and centred estimates give
\begin{equation*}
  |G_s-G_q|\leq r^{H-\kappa}.
\end{equation*}
Since \(G_q>2r^{H-\kappa}\), it follows that
\(G_s>r^{H-\kappa}\) throughout the interval.  The target event is thus
contained in the two estimated failure events, proving the claim.
\end{proof}

\begin{theorem}
\label{thm:uniform-growing-right-tail}
For every \(\eta>0\),
\begin{equation}
  \lim_{R\to\infty}\sup_{T\geq2}
  \widehat{\PP}_T\left(
    \int_R^{T-\sigma_T}
      e^{-(M_T-B_{\sigma_T+s})}\,\mathrm ds>\eta
  \right)=0,
  \label{eq:uniform-growing-right-tail-long}
\end{equation}
where the integral is zero when its lower endpoint exceeds its upper
endpoint.  Equivalently,
\begin{equation}
  \lim_{R\to\infty}\sup_{a\geq2}
  \widehat{\PP}_a\left(
    \int_R^{a-U_a}e^{-R_a(s)}\,\mathrm ds>\eta
  \right)=0.
  \label{eq:uniform-growing-right-tail}
\end{equation}
Under \(\mu_H\),
\begin{equation}
  \int_0^\infty e^{-R(s)}\,\mathrm ds<\infty
  \quad\text{almost surely},
  \qquad
  \int_R^\infty e^{-R(s)}\,\mathrm ds\longrightarrow0
  \label{eq:limit-right-functional-tail}
\end{equation}
almost surely and in probability.
The equivalence of the two finite-horizon formulations is the
self-similar identification \eqref{eq:long-horizon-scaling}.
\end{theorem}

\begin{proof}
Fix \(\eta>0\) and \(\varepsilon>0\).  The normalised left-horizon estimate
\eqref{eq:tilted-left-horizon-tail} lets us choose \(K\geq1\) so that
\begin{equation}
  \sup_{T\geq2}\widehat{\PP}_T(\sigma_T>K-1)
  \leq\varepsilon.
  \label{eq:growing-tail-left-truncation}
\end{equation}
For fixed \(T\), refine the dyadic grids
\(h\mathbb Z\cap[0,T]\).  Write \(T_h\) for the final grid point,
\(M_{T,h}\) and \(\sigma_{T,h}\) for the grid maximum and maximiser, and
\begin{equation}
  \frac{\mathrm d\widehat{\PP}_{T,h}}{\mathrm d\PP}
  :=\frac{e^{-M_{T,h}}}{\EE e^{-M_{T,h}}}.
  \label{eq:growing-tail-grid-tilt}
\end{equation}
The grid maximum, maximiser, and \(T_h\) converge almost surely to their
continuous counterparts and \(T\).  The unnormalised tilt weights therefore
converge almost surely, and bounded convergence gives convergence of their
normalisers.  More explicitly, with
\(w_h:=e^{-M_{T,h}}\) and \(w:=e^{-M_T}\), almost surely
\begin{equation*}
  \limsup_{h\downarrow0}
    w_h\mathds1_{\{\sigma_{T,h}>K\}}
  \leq
    w\mathds1_{\{\sigma_T>K-1\}}.
\end{equation*}
The reverse-Fatou inequality for these variables in \([0,1]\), convergence
of the normalisers, and \eqref{eq:growing-tail-left-truncation} show that,
on all sufficiently fine grids, the total tilted mass of components with
\(u>K\) is at most \(2\varepsilon\).  The required fineness may depend
on \(T\); the estimates below do not.  Put
\begin{equation}
  X_{T,h}^{(R)}
  :=\int_R^{T_h-\sigma_{T,h}}
       e^{-(M_{T,h}-B_{\sigma_{T,h}+s})}\,\mathrm ds,
  \label{eq:growing-tail-grid-functional}
\end{equation}
with the integral interpreted as zero if its lower endpoint exceeds its
upper endpoint.

By the exact rowwise mixture
\eqref{eq:exact-tilted-mixture}, conditional on the grid maximiser
\(\sigma_{T,h}=u\), the centred deficit path has law
\(\widehat Q_{h,u,L}\), where \(L:=T_h-u\).  Thus every component with
\(u\leq K\) satisfies the parameter range of
Lemma~\ref{lem:continuous-block-minimum}.

It suffices by monotonicity to prove the estimate for
\(R=2^J\), \(J\geq1\): for arbitrary \(R\), use the largest such power
not exceeding it.  If \(L<R\), the component tail is zero.  Otherwise,
the nonempty intervals
\begin{equation*}
  [2^j,\min(2^{j+1},L)],
  \qquad J\leq j\leq\lfloor\log_2L\rfloor,
\end{equation*}
cover \([R,L]\), with only the last block possibly truncated or a
singleton.  Fix \(\kappa\in(0,H)\).
Lemma~\ref{lem:continuous-block-minimum} and the union bound show that the
probability of a failure in some occupied block is at most
\begin{equation}
  C_{K,\kappa}\sum_{2^j\geq R}
  \left[
    2^{-\kappa j}
    +\sqrt{j+1}\,2^{-(H-\kappa)j}
  \right],
  \label{eq:growing-tail-bad-block-sum}
\end{equation}
which tends to zero with \(R\).  Off this event,
\begin{equation}
  G_s\geq2^{j(H-\kappa)}
\end{equation}
throughout every occupied block, and hence
\begin{equation}
  \int_R^Le^{-G_s}\,\mathrm ds
  \leq
  \sum_{2^j\geq R}
    2^j e^{-2^{j(H-\kappa)}}.
  \label{eq:growing-tail-deterministic-series}
\end{equation}
Let \(b_{K,\kappa}(R)\) and \(d_\kappa(R)\) denote the right-hand sides of
\eqref{eq:growing-tail-bad-block-sum} and
\eqref{eq:growing-tail-deterministic-series}, respectively.  Both tend to
zero as \(R\to\infty\).  Fix \(0<\eta'<\eta\) and take \(R\) sufficiently
large that \(d_\kappa(R)<\eta'\).  On components with \(u\leq K\), the event
\(\{X_{T,h}^{(R)}>\eta'\}\) can then occur only if some occupied block
fails.  After adding the tilted mass of the components with \(u>K\), we
obtain, for every fixed \(T\) and every sufficiently fine grid,
\begin{equation}
  \widehat{\PP}_{T,h}\bigl(X_{T,h}^{(R)}>\eta'\bigr)
  \leq2\varepsilon+b_{K,\kappa}(R).
  \label{eq:growing-tail-uniform-grid-bound}
\end{equation}

The passage to the continuous maximiser is made at each fixed \(T\), not by
using the local topology of
Proposition~\ref{prop:relative-tilted-grid-diagonal}.  On the full finite
domain, the grid maximum and maximiser converge almost surely, and uniform
continuity gives convergence of the complete centred arms and of the tail
integral.  Thus \(X_{T,h}^{(R)}\to X_T^{(R)}\) almost surely, where
\begin{equation}
  X_T^{(R)}
  :=\int_R^{T-\sigma_T}
       e^{-(M_T-B_{\sigma_T+s})}\,\mathrm ds.
  \label{eq:growing-tail-continuum-functional}
\end{equation}
The unnormalised tilt weights lie in \([0,1]\) and converge almost surely,
while their normalisers converge.  The lowered threshold
\(\eta'<\eta\) and the almost-sure convergence give
\begin{equation*}
  e^{-M_T}\mathds1_{\{X_T^{(R)}>\eta\}}
  \leq
  \liminf_{h\downarrow0}
  e^{-M_{T,h}}\mathds1_{\{X_{T,h}^{(R)}>\eta'\}}
  \quad\text{almost surely}.
\end{equation*}
Fatou's lemma and
\eqref{eq:growing-tail-uniform-grid-bound} give
\begin{equation}
  \widehat{\PP}_T\bigl(X_T^{(R)}>\eta\bigr)
  \leq\liminf_{h\downarrow0}
        \widehat{\PP}_{T,h}\bigl(X_{T,h}^{(R)}>\eta'\bigr)
  \leq2\varepsilon+b_{K,\kappa}(R).
  \label{eq:growing-tail-fixed-horizon-bound}
\end{equation}
Although the required grid fineness may depend on \(T\), the right-hand
side does not.  Hence
\begin{equation}
  \sup_{T\geq2}
  \widehat{\PP}_T\bigl(X_T^{(R)}>\eta\bigr)
  \leq2\varepsilon+b_{K,\kappa}(R).
  \label{eq:growing-tail-uniform-continuum-bound}
\end{equation}
First let \(R\to\infty\) and then \(\varepsilon\downarrow0\).  This proves
\eqref{eq:uniform-growing-right-tail-long}--%
\eqref{eq:uniform-growing-right-tail}.

For the limiting law, first fix \(S>R\) and define
\begin{equation*}
  I_a^{R,S}
  :=\int_R^{\min(S,a-U_a)}e^{-R_a(s)}\,\mathrm ds,
\end{equation*}
with value zero when the upper endpoint is below \(R\).  The marked
convergence \eqref{eq:pb1-marked-convergence}, the right-horizon divergence,
and continuity of
\(r\mapsto\int_R^S e^{-r(s)}\,\mathrm ds\) on
\(C([R,S])\) give
\begin{equation*}
  I_a^{R,S}\Longrightarrow
  I^{R,S}:=\int_R^S e^{-R(s)}\,\mathrm ds.
\end{equation*}
Indeed, the truncated finite integral agrees with the continuous marked
functional whenever \(a-U_a\geq S\), an event whose probability tends to
one.  Also \(I_a^{R,S}\) is bounded above by the complete finite tail.
Thus, for \(0<\eta'<\eta\), the Portmanteau theorem gives
\begin{equation}
\begin{split}
  &\PP_{\mu_H}\left(\int_R^S e^{-R(s)}\,\mathrm ds>\eta\right)\\
  &\qquad\leq
  \liminf_{a\to\infty}
  \widehat{\PP}_a\left(
    \int_R^{a-U_a}e^{-R_a(s)}\,\mathrm ds>\eta'
  \right).
\end{split}
  \label{eq:limit-tail-portmanteau}
\end{equation}
Use \eqref{eq:uniform-growing-right-tail} on the right and then let
\(S\to\infty\).  The integrals and their exceedance events increase
monotonically in \(S\), so this gives convergence to zero in probability
of the extended tail integrals
\begin{equation*}
  Y_R:=\int_R^\infty e^{-R(s)}\,\mathrm ds
\end{equation*}
as \(R\to\infty\).  The variables \(Y_R\) decrease almost surely in
\(R\), so they have an extended almost-sure limit.  Convergence in
probability to zero forces that limit to be zero almost surely.  In
particular, \(Y_R<\infty\) for all sufficiently large \(R\) almost
surely; adding the finite continuous-path integral on \([0,R]\) proves
\eqref{eq:limit-right-functional-tail}.
\end{proof}

\begin{proposition}
\label{prop:reciprocal-uniform-integrability}
For every \(p>1\),
\begin{equation}
  \sup_{a\geq2}\widehat{\EE}_a[J_a^{-p}]<\infty.
  \label{eq:reciprocal-uniform-lp}
\end{equation}
In particular, \((J_a^{-1})_{a\geq2}\) is uniformly integrable.
\end{proposition}

\begin{proof}
Fix \(p>1\).  Lemma~\ref{lem:endpoint-tilted-unit-arm-moment} with
\(\lambda=p\), together with
\([0,\delta]\subseteq[0,1]\) and
\([-\delta,0]\subseteq[-1,0]\), gives
\begin{equation}
  \sup_{\substack{0<\delta\leq1,\ F\subset\mathbb D\text{ finite},\
                   e\in F\cup\{0\}\\
                   I=[0,\delta]\text{ or }[-\delta,0]}}
  \EE_{\widehat Q_{F,e}}
  \exp\left(p\sup_{t\in I}G_t\right)
  \leq C_{H,p}<\infty,
  \label{eq:reciprocal-component-unit-arm-moment}
\end{equation}
over intervals contained in the relevant arm.

Use the long-horizon representation \eqref{eq:long-horizon-scaling}, so that
under \(\widehat{\PP}_a\) we may realise
\(J_a=\int_0^a e^{B_t-M_a}\,\mathrm dt\).  First work on
\(h\mathbb Z\cap[0,a]\), where \(h=2^{-r}\leq1\), and let
\(a_h:=h\lfloor a/h\rfloor\) be its final point.  Denote the grid maximum
and its leftmost maximiser by \(M_{a,h}\) and \(\sigma_{a,h}\), and set
\begin{equation*}
  Z_{a,h}:=\EE e^{-M_{a,h}},
  \qquad
  \frac{\mathrm d\widehat{\PP}_{a,h}}{\mathrm d\PP}
  :=\frac{e^{-M_{a,h}}}{Z_{a,h}}.
\end{equation*}
Write \(u:=\sigma_{a,h}\), \(L:=a_h-u\), and define the constant-extended
deficit arms by
\begin{equation*}
\begin{aligned}
  L_{a,h}(s)&:=M_{a,h}-B_{(u-s)\vee0},\\
  R_{a,h}(s)&:=M_{a,h}-B_{(u+s)\wedge a_h}.
\end{aligned}
\end{equation*}
Their exponential functional is
\begin{equation*}
  J_{a,h}
  :=\int_0^u e^{-L_{a,h}(s)}\,\mathrm ds
    +\int_0^L e^{-R_{a,h}(s)}\,\mathrm ds
  =\int_0^{a_h}e^{B_t-M_{a,h}}\,\mathrm dt.
\end{equation*}
The component horizons satisfy
\begin{equation}
  u+L=a_h\geq a-h\geq1,
\end{equation}
so at least one of \(u,L\) is at least
\(\delta:=\frac12\).  If \(u\geq\delta\), then
\begin{equation}
  J_{a,h}
  \geq\int_0^\delta e^{-L_{a,h}(s)}\,\mathrm ds
  \geq\delta
       \exp\left(-\sup_{0\leq s\leq\delta}L_{a,h}(s)\right);
  \label{eq:reciprocal-left-arm-lower}
\end{equation}
and the analogous inequality holds with the right arm whenever
\(L\geq\delta\).  Consequently, whether one or both arms are long,
\begin{equation}
\begin{split}
  J_{a,h}^{-p}\leq\delta^{-p}\biggl[
  &\mathds1_{\{u\geq\delta\}}
    \exp\left(p\sup_{0\leq s\leq\delta}L_{a,h}(s)\right)\\
  &+\mathds1_{\{L\geq\delta\}}
    \exp\left(p\sup_{0\leq s\leq\delta}R_{a,h}(s)\right)
  \biggr].
\end{split}
  \label{eq:reciprocal-unit-arm-upper}
\end{equation}
Under the rowwise exact mixture \eqref{eq:exact-tilted-mixture}, conditional on
\(\sigma_{a,h}=u\), the centred path
\(W_s:=B_{u+s}-M_{a,h}\) has the endpoint-tilted law
\(\widehat Q_{F,e}\), where
\begin{equation*}
  F:=h\mathbb Z\cap[-u,L],
  \qquad e:=-u.
\end{equation*}
The set \(F\) is finite and contained in \(\mathbb D\).  Here \(e\in F\)
when \(u>0\), while \(e=0\) in the left-endpoint component.  On
\([-\delta,0]\), respectively \([0,\delta]\), the process
\(G=-W\) is exactly the nonconstant part of \(L_{a,h}\), respectively
\(R_{a,h}\).  Thus
\eqref{eq:reciprocal-component-unit-arm-moment} bounds each nonzero term in
\eqref{eq:reciprocal-unit-arm-upper} by the same constant \(C_{H,p}\),
including both endpoint components.  Averaging with the probability weights
of the exact mixture gives
\begin{equation}
  \sup_{\substack{a\geq2,\ h=2^{-r}\leq1}}
  \widehat{\EE}_{a,h}[J_{a,h}^{-p}]
  \leq 2\delta^{-p}C_{H,p}<\infty.
  \label{eq:reciprocal-uniform-grid-lp}
\end{equation}

For fixed \(a\), refine the physical grid.  The grid maximum, maximiser,
and final point converge almost surely to \(M_a\), \(\sigma_a\), and \(a\).
Uniform continuity of \(B\) on \([0,a]\), or directly the last integral
formula above, gives \(J_{a,h}\to J_a\) almost surely.  Put
\(w_{a,h}:=e^{-M_{a,h}}\) and \(w_a:=e^{-M_a}\).  Then
\(w_{a,h}\to w_a\) almost surely, both weights lie in \([0,1]\), and bounded
convergence gives \(Z_{a,h}\to Z_a:=\EE w_a>0\).  Fatou's lemma applied to
the nonnegative weighted reciprocals now yields the complete
changing-likelihood calculation
\begin{equation*}
\begin{aligned}
  \widehat{\EE}_a[J_a^{-p}]
  &=\frac{\EE[w_aJ_a^{-p}]}{Z_a}\\
  &\leq\frac{\liminf_{h\downarrow0}
                  \EE[w_{a,h}J_{a,h}^{-p}]}{Z_a}
   =\liminf_{h\downarrow0}
      \widehat{\EE}_{a,h}[J_{a,h}^{-p}]
  \leq2\delta^{-p}C_{H,p}.
\end{aligned}
\end{equation*}
The right-hand side is independent of \(a\), so taking the supremum proves
\eqref{eq:reciprocal-uniform-lp}.  Finally, for \(K>0\),
\begin{equation}
  \sup_{a\geq2}
  \widehat{\EE}_a\left[
    J_a^{-1}\mathds1_{\{J_a^{-1}>K\}}
  \right]
  \leq
  K^{1-p}\sup_{a\geq2}\widehat{\EE}_a[J_a^{-p}]
  \longrightarrow0
  \qquad(K\to\infty).
  \label{eq:reciprocal-uniform-integrability-tail}
\end{equation}
This proves uniform integrability directly.
\end{proof}

\begin{theorem}
\label{thm:reciprocal-functional-convergence}
Under the law \(\mu_H\) of
\eqref{eq:unique-entrance-mixture}, define
\begin{equation}
  J_\infty
  :=\int_0^Ue^{-L(s)}\,\mathrm ds
    +\int_0^\infty e^{-R(s)}\,\mathrm ds.
  \label{eq:proved-limit-functional}
\end{equation}
Then
\begin{equation}
  0<J_\infty<\infty
  \quad\text{almost surely},
  \label{eq:limit-functional-positive-finite}
\end{equation}
and
\begin{equation}
  J_a\Longrightarrow J_\infty,
  \qquad
  \widehat{\EE}_a[J_a^{-1}]
  \longrightarrow
  D_H:=\EE_{\mu_H}[J_\infty^{-1}]
  \in(0,\infty).
  \label{eq:proved-reciprocal-convergence}
\end{equation}
\end{theorem}

\begin{proof}
For fixed \(R<\infty\), put
\begin{equation}
\begin{aligned}
  J_a^{(R)}
  &:=
    \int_0^{U_a}e^{-L_a(s)}\,\mathrm ds
    +\int_0^{\min(R,a-U_a)}e^{-R_a(s)}\,\mathrm ds,\\
  J_\infty^{(R)}
  &:=
    \int_0^Ue^{-L(s)}\,\mathrm ds
    +\int_0^Re^{-R(s)}\,\mathrm ds.
\end{aligned}
  \label{eq:truncated-persistence-functionals}
\end{equation}
Define on the marked space \(\mathcal Y\)
\begin{equation*}
  \Phi_R(u,v,\ell,r)
  :=\int_0^u e^{-\ell(s)}\,\mathrm ds
    +\int_0^R e^{-r(s)}\,\mathrm ds.
\end{equation*}
This map is continuous.  Indeed, if \(u_n\to u<\infty\), choose
\(K>u+1\).  Eventually \(u_n\leq K\); local uniform convergence of
\(\ell_n\) to \(\ell\) makes the exponentials converge uniformly and remain
uniformly bounded on \([0,K]\), which, together with \(u_n\to u\), proves
convergence of the left integrals.  The right integral is continuous by
uniform convergence on the fixed interval \([0,R]\).

Put
\begin{equation*}
  \widetilde J_a^{(R)}
  :=\Phi_R(U_a,V_a,\overline L_a,\overline R_a).
\end{equation*}
The marked convergence \eqref{eq:pb1-marked-convergence} and the continuous
mapping theorem give
\(\widetilde J_a^{(R)}\Longrightarrow J_\infty^{(R)}\).  On the event
\(\{a-U_a\geq R\}\), one has
\(\widetilde J_a^{(R)}=J_a^{(R)}\).  The prelimit right horizons diverge in
probability, so the probability of the complementary event tends to zero.
Consequently, Slutsky's theorem gives
\begin{equation}
  J_a^{(R)}\Longrightarrow J_\infty^{(R)}.
  \label{eq:truncated-functional-convergence}
\end{equation}

Theorem~\ref{thm:uniform-growing-right-tail} says explicitly that, for every
\(\eta>0\),
\begin{equation*}
  \lim_{R\to\infty}\limsup_{a\to\infty}
  \widehat{\PP}_a\left(
    \lvert J_a-J_a^{(R)}\rvert>\eta
  \right)=0.
\end{equation*}
Under \(\mu_H\), \(U<\infty\) almost surely and the left arm is continuous
on \([0,U]\), so its integral is finite.  The limiting-law conclusion of
the same theorem gives finiteness of the right integral.  Hence
\begin{equation*}
  J_\infty^{(R)}\uparrow J_\infty<\infty
  \quad\text{almost surely},
\end{equation*}
and therefore \(J_\infty^{(R)}\Longrightarrow J_\infty\) as
\(R\to\infty\).  These two limit statements and the fixed-\(R\) convergence
\eqref{eq:truncated-functional-convergence} are precisely the three
hypotheses of the converging-together theorem.  Thus
\begin{equation}
  J_a\Longrightarrow J_\infty.
  \label{eq:full-functional-convergence}
\end{equation}
Moreover, continuity and finiteness of \(R\) on \([0,1]\) give
\begin{equation*}
  J_\infty\geq\int_0^1e^{-R(s)}\,\mathrm ds>0
  \quad\text{almost surely}.
\end{equation*}
Together with the finiteness above, this proves
\eqref{eq:limit-functional-positive-finite}.

Every \(J_a\) is strictly positive, and the reciprocal map is continuous on
\((0,\infty)\).  The continuous mapping theorem therefore gives
\begin{equation}
  J_a^{-1}\Longrightarrow J_\infty^{-1}.
\end{equation}
Proposition~\ref{prop:reciprocal-uniform-integrability} makes the variables
on the left uniformly integrable.  To spell out the expectation passage,
for every \(K<\infty\) the bounded continuous truncation gives
\begin{equation*}
  \widehat{\EE}_a[J_a^{-1}\wedge K]
  \longrightarrow
  \EE_{\mu_H}[J_\infty^{-1}\wedge K].
\end{equation*}
The uniform tail estimate
\eqref{eq:reciprocal-uniform-integrability-tail}, followed by
\(K\to\infty\), removes the truncation on the prelimit side.  Portmanteau
applied first to bounded truncations of
\(x\mapsto(x-K)_+\) transfers the same tail bound to the limit, so it also
removes the limiting truncation.  This proves convergence of expectations
and \(D_H<\infty\).  Since \(J_\infty^{-1}>0\) almost surely, one also has
\(D_H>0\), completing \eqref{eq:proved-reciprocal-convergence}.
\end{proof}

\begin{proof}[Proof of Theorem~\ref{thm:rough-sharp-persistence}]
Theorem~\ref{thm:unique-entrance-mixture} gives the probability law
\(\mu_H=\int K_u\,\pi_H(\mathrm du)\) and the marked convergence
\eqref{eq:rough-tilted-marked-convergence}.  By the definition of \(K_u\),
its first coordinate is the finite value \(u\), its left arm is constant
after \(u\), and its right arm is a continuous path on the whole half-line;
integration against \(\pi_H\) gives the asserted support properties.
Theorem~\ref{thm:reciprocal-functional-convergence}, for this same law and
the same arm coordinates, gives the definition, positivity, and finiteness
of \(J_\infty\), the convergence \(J_a\Longrightarrow J_\infty\), and
exactly \eqref{eq:rough-reciprocal-limit} with \(D_H\in(0,\infty)\).
This is the hypothesis of
Proposition~\ref{prop:exact-persistence-compiler}, which then gives
\eqref{eq:rough-sharp-persistence-asymptotic} and
\eqref{eq:rough-persistence-constant}.
\end{proof}

\section*{Acknowledgements}

The author thanks Frank Aurzada for helpful feedback on an earlier draft.

\begingroup

\raggedright

\printbibliography

\endgroup

\end{document}